\documentclass[10pt, a4paper, reqno, english]{amsart}

\usepackage{hyperref}
\usepackage{amssymb,mathrsfs}
\usepackage[all]{xy}
\usepackage{enumitem}
\usepackage{tikz}
\usetikzlibrary{matrix,arrows}
\usepackage{comment}
\usepackage{url}

\newcommand{\vardel}{\varDelta}
\newcommand{\eps}{\epsilon}

\newcommand{\kap}{\kappa}

\newcommand{\ome}{\omega}
\DeclareMathOperator{\id}{\mathbf{1}}

\newcommand{\bfy}{\mathbf{y}}

\newcommand{\del}{\delta}

\newcommand{\tet}{\theta}

\newcommand{\alp}{\alpha}

\newcommand{\bfW}{\mathbf{W}}

\newcommand{\calL}{\mathcal{L}}
\newcommand{\calP}{\mathcal{P}}

\newcommand{\bfl}{\mathbf{l}}

\newcommand{\calK}{\mathcal{K}}

\newcommand\ZZ{\mathbb{Z}}
\newcommand{\Z}{\mathbb{Z}}
\newcommand\NN{\mathbb{N}}
\newcommand{\N}{\mathbb{N}}
\newcommand\CC{\mathbb{C}}
\newcommand\RR{\mathbb{R}}
\newcommand\R{\mathbb{R}}
\newcommand\QQ{\mathbb{Q}}

\newcommand\FF{\mathbb{F}}
\newcommand{\KK}{\mathbb{K}} 
\newcommand\p{\mathfrak{p}} 
\newcommand{\idealnorm}{\mathfrak{N}}
\newcommand\PP{\mathbb{P}} 
\newcommand{\A}{\mathbb{A}} 
\newcommand\GG{\mathbb{G}} 
\newcommand{\OO}{\mathcal{O}} 
\newcommand{\ii}[1]{\mathcal{J}_{#1}} 
\newcommand{\ic}[1]{\mathcal{I}_{#1}} 
\newcommand{\gen}[1]{\nu_{#1}} 
\DeclareMathOperator{\spec}{Spec} 
\DeclareMathOperator{\pic}{Pic} 
\DeclareMathOperator{\eff}{Eff} 
\DeclareMathOperator{\meas}{meas} 
\DeclareMathOperator{\cone}{cone} 
\DeclareMathOperator{\Fr}{Fr} 

\newtheorem{theorem}{Theorem}
\newtheorem{lemma}[theorem]{Lemma}
\newtheorem{prop}[theorem]{Proposition}

\theoremstyle{definition}

\newtheorem{proposition}[theorem]{Proposition}

\newtheorem{remark}[theorem]{Remark}

\newtheorem{assumption}[theorem]{Assumption}

\numberwithin{theorem}{section}
\numberwithin{equation}{section}

\begin{document}

\setcounter{tocdepth}{2}

\title{Hyperbola method on toric varieties}
  
\author{Marta Pieropan} 
\author{Damaris Schindler} 

\address{Utrecht University, Mathematical Institute, Budapestlaan 6, 3584 CD Utrecht, the Netherlands \emph{and} EPFL SB MATH CAG, B\^at. MA, Station 8, 1015 Lausanne, Switzerland}
\email{m.pieropan@uu.nl}
\address{Mathematisches Institut, Universitaet Goettingen, Bunsenstrasse 3-5, 37073 Goettingen, Germany}
\email{damaris.schindler@mathematik.uni-goettingen.de}
\date{\today}


\subjclass[2010]{11P21 (11A25, 11G50, 14G05, 14M25)}
\keywords{Hyperbola method, $m$-full numbers, Campana points, toric varieties.}

\begin{abstract}
	We develop a very general version of the hyperbola method which extends the known method by Blomer and  Br{\"u}dern for products of projective spaces to complete smooth split
	toric varieties. We use it 
	to count Campana points of bounded log-anticanonical height 
	on complete smooth split toric $\QQ$-varieties with torus invariant boundary. We apply the strong duality principle in linear programming to show the compatibility of our results with the conjectured asymptotic.
\end{abstract}

\maketitle
\tableofcontents

\section{Introduction}

This paper stems from an investigation of the universal torsor method \cite{MR1679841, MR3552013} in relation to the problem of counting Campana points of bounded height on log Fano varieties in the framework of \cite[Conjecture 1.1]{PSTVA}. 
Campana points are a notion of points that interpolate between rational points and integral points on certain log smooth pairs, or orbifolds, introduced and first studied by Campana \cite{MR2097416, MR2831280, MR3477539}. The study of the distribution of Campana points over number fields was initiated only quite recently and the literature on this topic is still sparse \cite{MR2931315, MR2968632, BY, Shute2, PSTVA, Shute, Streeter,   Xiao, BBKOPWproc}.
In this paper we deal with toric varieties, which constitute a fundamental family of examples for the study of the distribution of rational points \cite{MR1353919, MR1369408, MR1423638, MR1620682, MR1679841, Bre01b}, via a combination of the universal torsor method with a very general version of the hyperbola method, which we develop. 

We use the universal torsor method, instead of exploiting the toric group structure, because we hope to extend our approach to a larger class of log Fano varieties in the future. Indeed, 
the hyperbola method is well suited to deal with subvarieties \cite{MR3320489, Sch, MR3731304, BB, mignot1, Mignot, mignot2, BroHu}, and all log Fano varieties admit neat embeddings in toric varieties \cite{MR3307753, MR3275656} which  can be exploited for the universal torsor method.\par

One of the key technical innovations in this article is the development of a very general form of the hyperbola method, which is motivated by work of Blomer and  Br{\"u}dern in the case of products of projective spaces \cite{BB}. Mignot \cite{mignot2} has adjusted these ideas to complete smooth split toric varieties with simplicial effective cone. 
 With our approach we extend the work of Blomer and  Br{\"u}dern to complete smooth split toric varieties with additional flexibility to change the height function.\par
Let $f: \N^s\rightarrow \R_{\geq 0}$ be an arithmetic function for which one has asymptotics for summing the function $f$ over boxes, see Property I in Section \ref{sec:hyperbola_method}. Let $B$ be a large real parameter, $\calK$ a finite index set and $\alp_{i,k}\geq 0$ for $1\leq i\leq s$ and $k\in \calK$. The goal is then to use this information from sums over boxes to deduce an asymptotic formula for sums of the form 
\begin{equation*}
S^{f}(B):=\sum_{\substack{\prod_{i=1}^s y_i^{\alp_{i,k}}\leq B,\ \forall k\in \calK\\ y_i\in \N, 1\leq i\leq s }}f(\bfy).
\end{equation*}
We define the polyhedron $\calP \subset \R^s$ given by
\begin{equation*}
\sum_{i=1}^s \alp_{i,k} \varpi_i^{-1} t_i\leq 1, \quad k\in \calK
\end{equation*}
and
\begin{equation*}
t_i\geq 0,\quad 1\leq i\leq s.
\end{equation*}
Here the parameters $\varpi_i$, $1\leq i\leq s$ are defined in Property I for the function $f$. The linear function $\sum_{i=1}^s t_i$ takes its maximal value on a face of $\calP$ which we call $F$. We write $a$ for its maximal value.

\begin{theorem}\label{lemhyp3}
Let $f:\N^s\rightarrow \R_{\geq 0}$ be a function that satisfies Property I from Section 
\ref{sec:hyperbola_method}.
Assume that $\calP$ is bounded and non-degenerate, and that $F$ is not contained in a coordinate hyperplane of $\R^s$. Let $k=\dim F$. Then we have
\begin{gather*}
S^f(B)= (s-1-k)!  C_{f,M}c_{\calP} (\log B)^kB^{a} 
+O\left(  C_{f,E} (\log \log B)^{s}  (\log B)^{k-1} B^{a} \right),
\end{gather*}
where $C_{f,M}$ and $C_{f,E}$ are the constants in Property I and $c_\calP$ is the constant in equation (\ref{eqn65}).
\end{theorem}

In comparison to earlier versions of the hyperbola method in work of Blomer and  Br{\"u}dern \cite{BB} or Mignot \cite{mignot2}, we only obtain a saving of a power of a $\log B$, but we can work with the weaker assumption of using only Property I. In contrast to \cite{BB} and \cite{mignot2} we no longer need  to assume that we can evaluate the function $f$ on lower-dimensional boxes after fixing a number of variables, called Property II in their work.\par
The case treated in \cite{BB} would in our notation correspond to an index set $\calK$ with one element where all the $\alp_{i,k}=\alp$ for all $1\leq i\leq s$ and some $\alp>0$, and $k=s-1$. Our attack to evaluate the sum $S^f$ starts in a similar way as in \cite{BB}. We cover the region given by the conditions $\prod_{i=1}^s y_i^{\alp_{i,k}}\leq B$, $k\in \calK$ with boxes of different side lengths on which we can evaluate the function $f$. One important ingredient in the hyperbola method in \cite{BB} is a combinatorial identity for the generating series
$$\sum_{\substack{j_1+\ldots +j_s\leq J\\ j_i\geq 0,\ 1\leq i\leq s}}t^{j_1+\ldots + j_s},$$
which needs to be evaluated for $J$ going to infinity. By induction the authors give a closed expression. For us this part of the argument breaks down, as we have in general more complicated polytopes that arise in the summation condition for $S^f$ and are not aware of comparable combinatorial identities for the tuples $(j_1,\ldots, j_s)$ lying in general convex polytopes. Instead, we approximate the number of integer points in certain intersections of hyperplanes with a convex polytope by lattice point counting arguments and then use asymptotic evaluations for sums of the form
$$\sum_{0\leq m\leq M} m^\ell \tet^m$$
for $0<\tet<1$ and $\ell,M\in \N$.\par
In comparison to Mignot's work \cite{mignot2}, we can deal with polytopes for which $|\mathcal{K}|>1$, and where $k$ is no longer restricted to the case $k=s-1$.

\

Our main application of Theorem \ref{lemhyp3} is a proof of \cite[Conjecture 1.1]{PSTVA} for smooth split toric varieties over $\QQ$ with the log-anticanonical height.
\begin{theorem}
\label{thm:toric}
Let $\Sigma$ be the fan of a complete smooth split toric variety $X$ over $\QQ$. Let $\{\rho_1,\dots,\rho_s\}$ be the set of rays of $\Sigma$.
For each  $i\in\{1,\dots,s\}$ fix a positive integer $m_i$
and denote by $D_i$ the  torus invariant divisor corresponding to $\rho_i$. Assume that $L:=\sum_{i=1}^s\frac 1{m_i} D_i$ is ample. 
Let $H_L$ be the height defined by $L$ as in Section \ref{sec:heights}.
 Let $\mathscr X$ be the toric scheme defined by $\Sigma$ over $\ZZ$, and for each $i\in\{1,\dots,s\}$, let $\mathscr D_i$ be the closure of $D_i$ in $\mathscr X$.
 For every $B>0$, let $N(B)$ be the number of Campana $\ZZ$-points on the Campana orbifold $(\mathscr X, \sum_{i=1}^s(1-1/m_i)\mathscr D_i)$ (in the sense of \cite[Definition 3.4]{PSTVA}) of height $H_L$ at most $B$.
Then for sufficiently large $B>0$,
 \begin{equation}
 \label{eq:mainthm_asympt}
 N(B) = c B(\log B)^{r-1} + O(
	B(\log B)^{r-2}(\log\log B)^{s}),
  \end{equation}
where $r$ is the rank of the Picard group of $X$, and $c$ is a positive constant  compatible with the prediction in \cite[\S3.3]{PSTVA}.
\end{theorem}
Our application of the hyperbola method recovers Salberger's result \cite{MR1679841} and improves on the error term by saving a factor $(\log B)^{1-\varepsilon}$ where \cite{MR1679841} saves only a factor $(\log B)^{1-1/f-\varepsilon}$, where $f\geq 2$ is an integer that depends on the toric variety.

Theorem \ref{thm:toric} could also be deduced from work of de la Bret\`eche \cite{Bre01a}, \cite{Bre01b}, who developed a multi-dimensional Dirichlet series approach to count rational points of bounded height on toric varieties, or from work of Santens \cite{Santens}. Another approach could be via  harmonic analysis of the height zeta function, even though such a proof would probably be more involved than  the case of compactifications of vector groups \cite{PSTVA}.
Our proof proceeds via the universal torsor method introduced by Salberger in \cite{MR1679841} in combination with Theorem \ref{lemhyp3}. 
One of our main motivations for this approach is that it opens a path to counting Campana points on subvarieties of toric varieties.\par 
When we apply Theorem \ref{lemhyp3} to prove Theorem \ref{thm:toric}, we need to verify that both the exponent of $B$ as well as the power of $\log B$ match the prediction in \cite{PSTVA}. The exponent $a$ in Theorem \ref{lemhyp3} is the result of a linear optimization problem. Similarly, the construction of the height function leading to the exponent one of $B$ in Theorem \ref{thm:toric} involves another linear optimization problem. We use the strong duality property in linear programming \cite[Chapter 6]{Dantzig} to recognize that the exponents 
are indeed compatible, and that this holds heuristically also in the more general setting where the height is not necessarily log-anticanonical. For the compatibility of the exponents of $\log B$ we exploit a different duality setup, which involves the Picard group of $X$. 

In upcoming work, we show how Theorem \ref{lemhyp3} can be used to count rational points and Campana points of bounded height on certain subvarieties of toric varieties.

\

The paper is organized as follows. 
In Section \ref{sec:preliminaries} we provide some auxiliary estimates on variants of geometric sums which are used later in Section \ref{sec:hyperbola_method}. 
In Section \ref{sec:volume_section_polytope} we study volumes of slices of polytopes under small deformations.
In Section \ref{sec:hyperbola_method} we develop the hyperbola method and give a proof of Theorem \ref{lemhyp3}.
Sections \ref{sec:m-full} and \ref{sec:Campana_points_on_toric_varieties} are dedicated to the application of the hyperbola method to prove Theorem \ref{thm:toric}.
In Section \ref{sec:m-full} we study estimates for $m$-full numbers of bounded size subject to certain divisibility conditions, and we produce the estimates in boxes for the function $f$ associated to the counting problem in Theorem \ref{thm:toric}.
In Section \ref{sec:Campana_points_on_toric_varieties}  we describe the heights associated to semiample $\QQ$-divisors on toric varieties over number fields, we study some combinatorial properties of the polytopes that play a prominent role in the application of the hyperbola method, and we show that the heuristic expectations coming from the hyperbola method agree with the prediction in \cite[Conjecture 1.1]{PSTVA} on split toric varieties for Campana points of bounded height, where the height does not need to be anticanonical. We conclude the section with the proof of Theorem \ref{thm:toric}.\par

\

This article subsumes our previous work on the hyperbola method, which had appeared on arxiv.org under arXiv:2001.09815v2, and which proved Theorems \ref{lemhyp3} and \ref{thm:toric} under additional technical assumptions on the corresponding polytopes.

\subsection{Notation}
	We denote by $\NN$ or $\ZZ_{>0}$ the set of positive integers.
	We use $\sharp S$ or $| S |$ to indicate the cardinality of a finite set $S$.
	Bold letters denote $s$-tuples of real numbers, and for given $\mathbf x\in\RR^s$ 
	we denote by $x_1,\dots,x_s\in\RR$ the elements such that  $\mathbf x=(x_1,\dots,x_s)$. 
	For any subset $S\subset\RR^s$ we denote by $\cone(S)$ the cone generated by $S$.

	We denote by $\FF_p$ the finite field with $p$ elements and by $\overline{\FF}_p$ an algebraic closure. 		For a number field $\KK$, we denote by $\OO_{\KK}$ the ring of integers, 
	and by $\idealnorm(\mathfrak a)$ the norm of an ideal $\mathfrak a$ of $\OO_{\KK}$.
	We denote by $\Omega_{\KK}$ the set of places of $\KK$, by $\Omega_f$ the set of finite places, 
	and by $\Omega_\infty$ the set of infinite places. 
	For every place $v$ of $\KK$, we denote by  $\KK_v$ the completion of $\KK$ at $v$, 
	and we define $|\cdot|_v=|N_{\KK_v/\QQ_{\tilde v}}(\cdot)|_{\tilde v}$, 
	where $\tilde v$ is the place of $\QQ$ below $v$ and $|\cdot|_{\tilde v}$ is the usual real 
	or $p$-adic absolute value on $\QQ_{\tilde v}$.
	We denote by $|\cdot|$ the usual absolute value on $\RR$.

	We denote  the Picard group and the effective cone of a smooth variety $X$ 
	by $\pic(X)$ and $\eff(X)$, respectively. 
	For a divisor $D$ on $X$ we denote by $[D]$ its class in $\pic(X)$.
	We say that a $\QQ$-divisor $D$ on $X$ is semiample if there exists a positive integer $t$ 
	such that $tD$ has integer coefficients and is base point free.

\section{Preliminaries}
\label{sec:preliminaries}

In the hyperbola method in the next section we need good approximations for finite sums of the form
$$g_\ell (M,\tet):=\sum_{0\leq m\leq M} m^\ell \tet^m,$$
for some $0<\tet <1$ and natural numbers $\ell,M\geq 0$ (here and in the following we understand $0^0:=1$). 
In this subsection we also write $g_\ell(M)$ for $g_\ell(M,\tet)$.

We will use the following result.

\begin{lemma}\label{lemgl2}
For an integer $\ell\geq 0$ and $0<\tet < 1$ and a real number $M>\ell$ we have
$$(\tet -1)^{\ell+1}g_\ell(M)=(-1)^{\ell+1} \ell! +O_\ell(1-\tet)+O_\ell \left( \tet^M M^\ell\right).$$
\end{lemma}

Lemma \ref{lemgl2} can be deduced from the following statement. 

\begin{lemma}\label{lemgl}
Assume that $M>\ell\geq 0$ and $\tet >0$. Then we have
\begin{gather*}
(\tet -1)^{\ell+1}g_\ell(M)= \sum_{0\leq m<\ell+1}\tet^m \sum_{h=0}^m \binom{\ell+1}{h}(-1)^{\ell+1-h}(m-h)^\ell \\
+ \tet^M \sum_{0<m\leq \ell+1}\tet^m \sum_{h=m}^{\ell+1}\binom{\ell+1}{h}(-1)^{\ell+1-h}\sum_{k=0}^\ell \binom{\ell}{k} M^k (m-h)^{\ell-k}.
\end{gather*}
\end{lemma}

For the proof of Lemma \ref{lemgl} and Lemma \ref{lemgl2} we need the following identity. For an integer $0\leq \alp\leq \ell$ we have
\begin{equation}\label{eqnbinom}
\sum_{h=0}^{\ell+1}\binom{\ell+1}{h} (-1)^{\ell+1-h}h^\alp =0.
\end{equation}
To see this consider the identity
$$(t-1)^{\ell+1}=\sum_{h=0}^{\ell+1}\binom{\ell+1}{h} t^h (-1)^{\ell+1-h}.$$
Now take derivatives with respect to $t$ and then set $t=1$.\par

\begin{proof}[Lemma \ref{lemgl} implies Lemma \ref{lemgl2}]
We start in observing that
\begin{gather*}
(\tet -1)^{\ell+1}g_\ell(M)= \sum_{0\leq m<\ell+1}\tet^m \sum_{h=0}^m \binom{\ell+1}{h}(-1)^{\ell+1-h}(m-h)^\ell 
+O_\ell \left( \tet^M M^\ell\right)\\
= \sum_{0\leq m\leq \ell} \sum_{h=0}^m \binom{\ell+1}{h}(-1)^{\ell+1-h}(m-h)^\ell \\
+O_\ell(|\tet-1|)+O_l \left( \tet^M M^l\right).
\end{gather*}
We further compute
\begin{gather*}
(\tet -1)^{\ell+1}g_\ell(M)= \sum_{0\leq m\leq \ell+1} \sum_{h=0}^m \binom{\ell+1}{h}(-1)^{\ell+1-h}(m-h)^\ell \\ 
- \sum_{h=0}^{\ell+1} \binom{\ell+1}{h}(-1)^{\ell+1-h}(\ell+1-h)^\ell  \\
+O_\ell(|\tet-1|)+O_\ell \left( \tet^M M^\ell\right).
\end{gather*}
Note that the term in the second line is equal to zero by equation (\ref{eqnbinom}). Hence we have
\begin{gather*}
(\tet -1)^{\ell+1}g_\ell(M)= \sum_{0\leq m\leq \ell+1} \sum_{h=0}^m \binom{\ell+1}{h}(-1)^{\ell+1-h}(m-h)^\ell \\ 
+O_\ell(|\tet-1|)+O_\ell \left( \tet^M M^\ell\right).
\end{gather*}
We now switch the summation of $m$ and $h$ to obtain
\begin{equation*}
\begin{split}
(\tet -1)^{\ell+1}g_\ell(M)&= \sum_{0\leq h\leq \ell+1}  \binom{\ell+1}{h}(-1)^{\ell+1-h}\sum_{h\leq m\leq \ell+1} (m-h)^\ell \\ 
&+O_\ell(|\tet-1|)+O_\ell \left( \tet^M M^\ell\right)\\
&= \sum_{0\leq h\leq \ell+1}  \binom{\ell+1}{h}(-1)^{\ell+1-h}\sum_{0\leq t\leq \ell+1-h} t^\ell \\ 
&+O_\ell(|\tet-1|)+O_\ell \left( \tet^M M^\ell\right).
\end{split}
\end{equation*}
By the Faulhaber formulas $\sum_{0\leq t\leq \ell+1-h} t^\ell$ is a polynomial in $h$ with leading term
$$\frac{(\ell+1-h)^{\ell+1}}{\ell+1} = \frac{(-1)^{\ell+1}}{\ell+1}h^{\ell+1} + \mbox{ lower order terms in } h.$$
Using equation (\ref{eqnbinom}) we hence obtain
\begin{gather*}
(\tet -1)^{\ell+1}g_\ell(M)=  \sum_{0\leq h\leq \ell+1}  \binom{\ell+1}{h}(-1)^{\ell+1-h} \frac{(-1)^{\ell+1}}{\ell+1}h^{\ell+1} \\
+O_\ell(|\tet-1|)+O_\ell \left( \tet^M M^\ell\right)\\
= \frac{1}{\ell+1}\sum_{0\leq h\leq \ell+1}  \binom{\ell+1}{h}(-1)^{h} h^{\ell+1} \\
+O_\ell(|\tet-1|)+O_\ell \left( \tet^M M^\ell\right).
\end{gather*}
By equation (1.13) in \cite{Gould} we have
$$\sum_{0\leq h\leq \ell+1}  \binom{\ell+1}{h}(-1)^{h} h^{\ell+1} = (-1)^{\ell+1} (\ell+1)!$$
Hence we get
\begin{gather*}
(\tet -1)^{\ell+1}g_\ell(M)=(-1)^{\ell+1} \ell! +O_\ell(|\tet-1|)+O_\ell \left( \tet^M M^\ell\right). \qedhere
\end{gather*}
\end{proof}

We finish this section with a proof of Lemma \ref{lemgl}. 

\begin{proof}[Proof of Lemma \ref{lemgl}]
We compute
\begin{equation*}
\begin{split}
(\tet-1)^{\ell+1}g_\ell(M)&= (\tet-1)^{\ell+1}\sum_{0\leq m\leq M} m^\ell \tet^m \\
&= \sum_{h=0}^{\ell+1}\binom{\ell+1}{h}\tet^h (-1)^{\ell+1-h}\sum_{0\leq m\leq M} m^\ell \tet^m\\
&= \sum_{h=0}^{\ell+1} \binom{\ell+1}{h} (-1)^{\ell+1-h}\sum_{0\leq m\leq M}m^\ell \tet^{m+h}\\
&= \sum_{h=0}^{\ell+1} \binom{\ell+1}{h} (-1)^{\ell+1-h} \sum_{h\leq m\leq M+h} (m-h)^\ell \tet^m.
\end{split}
\end{equation*}
We split the last summation into three ranges depending on the size of $m$ and get
\begin{equation*}
\begin{split}
(\tet-1)^{\ell+1}g_\ell(M)&= \sum_{h=0}^{\ell+1} \binom{\ell+1}{h} (-1)^{\ell+1-h} \sum_{\ell+1\leq m\leq M} (m-h)^\ell \tet^m\\
&+ \sum_{h=0}^{\ell} \binom{\ell+1}{h} (-1)^{\ell+1-h} \sum_{h\leq m < \ell+1} (m-h)^\ell \tet^m\\
&+\sum_{h=1}^{\ell+1} \binom{\ell+1}{h} (-1)^{\ell+1-h} \sum_{M < m\leq M+h} (m-h)^\ell \tet^m\\
&= \sum_{\ell+1\leq m\leq M} \tet^m \sum_{k=0}^\ell \binom{\ell}{k} m^k \sum_{h=0}^{\ell+1} \binom{\ell+1}{h} (-1)^{\ell+1-h} (-h)^{\ell-k}\\
&+\sum_{0\leq m <\ell+1} \tet^m  \sum_{h=0}^{m} \binom{\ell+1}{h} (-1)^{\ell+1-h}(m-h)^\ell \\
&+\sum_{M<m\leq M+\ell+1} \tet^m \sum_{h=m-M}^{\ell+1} \binom{\ell+1}{h} (-1)^{\ell+1-h} (m-h)^\ell
\end{split}
\end{equation*}
We now use the identity (\ref{eqnbinom}) for the third last line and deduce that
\begin{gather}
(\tet-1)^{\ell+1}g_\ell(M)=\sum_{0\leq m <\ell+1} \tet^m  \sum_{h=0}^{m} \binom{\ell+1}{h} (-1)^{\ell+1-h}(m-h)^\ell \\
+\sum_{0<m\leq \ell+1} \tet^{M+m} \sum_{h=m}^{\ell+1} \binom{\ell+1}{h} (-1)^{\ell+1-h} (M+m-h)^\ell
\end{gather}
Now the lemma follows in expanding each of the terms $(M+m-h)^\ell$.
\end{proof}

\section{Volumes of certain sections of polytopes}
\label{sec:volume_section_polytope}
In this section we provide some estimates on the volumes of intersections 
	of convex polytopes with certain hyperplanes. 
	These will be used in the next section in the development of our generalized form of the hyperbola method.	
	The vector space $\RR^s$ is endowed with a fixed inner product that is used to define orthogonality and the Lebesgue measures.

		\begin{prop}
	\label{prop:volume_section_polytope}
 	Let $\mathcal P\subseteq \RR^s$ be an $s$-dimensional convex polytope with $s\geq1$. 
 	Let $\mathcal F$ be a proper face of $\mathcal P$. 
 	Let $H\subseteq\RR^s$ be a hyperplane such that $H\cap \mathcal P=\mathcal F$. 
 	Let $w\in\RR^s$ such that $\mathcal P\subseteq H+\RR_{\geq0}w$.
	For $\delta> 0$, let $H_\delta:=H+\delta w$. Let $k:=\dim \mathcal F$. 
	We denote by $\meas_j$ the $j$-dimensional measure 
	induced by the Lebesgue measure on $\RR^s$. Then
	\begin{enumerate}[label=(\roman*), ref=(\roman*)]
	\item 
	\label{item:volume_s-1}
 	for $\delta>0$ sufficiently small,
	\[
	\meas_{s-1} ( H_\delta \cap \mathcal P)= 
	c \delta^{s-1-k}+O(\delta^{s-k}),
	\]
	where $c$ is a positive constant that depends on $\mathcal P$, $\mathcal F$ and $H$.
	\item
	\label{item:volume_boundary}
	If $s\geq 2$, for $\delta>0$ sufficiently small we have
	\[
	\meas_{s-2} (\partial( H_\delta \cap \mathcal P))= \begin{cases}
	c' \delta^{s-2-k}+O(\delta^{s-1-k}),  & \text{ if } k\leq s-2,\\
	c' +O(\delta),  & \text{ if } k= s-1,
	\end{cases}
	\]
	where $c'$ is a positive constant that depends on $\mathcal P$, $\mathcal F$ and $H$.
	\end{enumerate}
	Let $T\subseteq \RR^s$ be a hyperplane such that $T\cap \mathcal P$ is a face of $\mathcal P$. Let $d:=\dim\mathcal P\cap T$.
	Let $u\in\RR^s$ be a vector such that $\mathcal P\subseteq T+\RR_{\geq0}u$.
 	For $\kappa>0$, let $T_\kappa=T+{\kappa}u$ and $T_{\leq\kappa}=T+[0,\kappa] u$.  Then
	\begin{enumerate}[resume, label=(\roman*), ref=(\roman*)]
	\item
	\label{item:volume_s-1_kappa}
 	for $\delta$ and $\kappa$ sufficiently small and positive,
 	\[
 	\meas_{s-1} ( H_\delta \cap \mathcal P \cap T_{\leq \kappa}) \ll \begin{cases} 
 	\kappa \delta^{s-1-k} & \text{ if } \mathcal F\not\subseteq T,\\
 	\min\{\kappa,\delta\} \delta^{s-2-k} & \text{ if } \mathcal F\subseteq T,
 	\end{cases}
 	\]
 	where the implicit constant is independent of $\kappa$ and $\delta$.
 		\item
	\label{item:volume_s-2_delta_kappa}
	For $\delta$ and $\kappa$ sufficiently small and positive,
	\[
	\meas_{s-2} ( H_\delta \cap \mathcal P \cap T_{\kappa}) \ll \begin{cases}
	\delta^{s-1-k} & \text{ if } \mathcal F\not\subseteq T,\\
	\delta^{d-k-1}\min\{\delta,\kappa\}^{s-d-1} & \text{ if } \mathcal F\subseteq T \text{ and } k\leq s-2,\\
	1 & \text{ if } k=s-1,
	\end{cases}
	\]
	where the implicit constant is independent of $\kappa$ and $\delta$.
 	\item
 	\label{item:volume_boundary_kappa}
 	If $s\geq 2$, for $\delta$ and $\kappa$ sufficiently small and positive,
 	\begin{gather*}
	\meas_{s-2} \partial (H_{\del}\cap \calP\cap T_{\leq\kappa})  
	\ll\begin{cases}
	\kappa \del^{\max\{s-2-k,0\}} + \del^{s-1-k} & \text{ if }  \mathcal F\not\subseteq T, \\
  	\delta^{\max\{s-2-k, 0\}}  & \text{ if } \mathcal F\subseteq T,
	 \end{cases}
	\end{gather*}
	where the implicit constant is independent of $\kappa$ and $\delta$.
 	\end{enumerate}
 	Let $V\subseteq\RR^s$ be an affine space of dimension $m$ with $k\leq m\leq s-1$. We denote by $p_V:\RR^n\to V$ the orthogonal projection onto $V$.  Then
	\begin{enumerate}[resume, label=(\roman*), ref=(\roman*)]
	\item
	\label{item:volume_proj} 	
 	for $\delta$ and $\kappa$ sufficiently small and positive such that $\delta\leq\kappa$,
	\begin{equation*}
	\meas_mp_V(H_\delta\cap\mathcal P\cap T_{\leq \kappa})
	\ll \begin{cases}
	\kappa \delta^{m-k} & \text{ if $\mathcal F\nsubseteq T$ and $\dim p_V(\mathcal F)\geq 1$,}\\
	\delta^{m-k} & \text{ otherwise,}\\
	\end{cases}
	\end{equation*}
	where the implicit constant is independent of $\kappa$ and $\delta$.
	\end{enumerate}
	\end{prop}
	\begin{proof}	
	If $\mathcal F\cap T=\emptyset$, 
	then for $\delta$ and  $\kappa$ small enough we have 
	$H_\delta \cap \mathcal P \cap T_{\leq\kappa}=\emptyset$. 
	Hence we can assume that $\mathcal F\cap T\neq\emptyset$.
	 		
	\emph{Step 1.}
	We first prove the case where $\mathcal P$ is a simplex. 
	We denote by $v_0, \dots, v_k$ 
	the vertices of $\mathcal F$ and by $v_{k+1},\dots, v_s$ the vertices of $\mathcal P$ 
	not contained in $\mathcal F$. Since $\mathcal F\cap T\neq\emptyset$, we can assume that $v_0\in T$.
	Up to a translation, which is a volume preserving automorphism of $\RR^s$, 
	we can assume that $v_0$ is the origin of $\RR^s$. 
	We observe that $v_1,\dots, v_s$ form a basis of the vector space $\RR^s$. 
	Let $C$ be the cone (with vertex $v_0$) generated by $v_{k+1},\dots, v_s$. 
	We observe that $C\cap H=\{v_0\}$ as $H\cap \mathcal P=\mathcal F$ and $C$ 
	is contained in the cone (with vertex $v_0$) generated by $\mathcal P$.
	Let $Q:= \mathcal F+ C$. Then $\mathcal P\subseteq Q$ and  $H\cap Q=\mathcal F$.
	
	We prove \ref{item:volume_s-1} by induction on $k$. If $k=0$,
	for $\delta$ small enough, $H_\delta\cap\mathcal P=H_\delta\cap C=\delta(H_1\cap C)$.
	Thus $\meas_{s-1}(H_\delta\cap\mathcal P)=\delta^{s-1}\meas_{s-1}(H_1\cap C)$. 
	Assume now that $k\geq 1$, and hence $s\geq 2$.	
	Let $L\subseteq \RR^s$ be the hyperplane that contains $v_1,\dots,v_s$. 
	Then $L^+=L+\RR_{\leq 0} v_1$ is the halfspace with boundary $L$ that contains $v_0$, 
	and $\mathcal P=Q\cap L^+$. Let $L^-=L+\RR_{\geq 0} v_1$. 
	Then 
	\begin{align*}
	\begin{split}
	\label{eq:volume:Qsplit}
	&\meas_{s-1}(H_\delta\cap\mathcal P)
	=\meas_{s-1}(H_\delta\cap Q)-\meas_{s-1}(H_\delta\cap Q\cap L^-)\\
	\end{split}
	\end{align*}
	Let $H_1^+=H_{1}+\RR_{\leq 0}w$ be the half space with boundary $H_1$ that contains $\mathcal F$. 
	Then $H_1^+\cap Q$ is bounded, and for $\delta$ small enough,  
	$H_\delta\cap Q\cap L^-=H_\delta\cap (H_1^+\cap Q\cap L^-)$ 
	and $H_1^+\cap Q\cap L^-$ is an $s$-dimensional polytope that intersects $H$ 
	in the $(k-1)$-dimensional face with vertices $v_1,\dots,v_k$. 
	By induction hypothesis we have 
	\begin{equation*}
	\label{eq:volume:Qcap L^-}
	\meas_{s-1}(H_\delta\cap Q\cap L^-)=O(\delta^{s-k}).
	\end{equation*} 
	We observe that $\meas_{s-1}(H_\delta\cap Q)=\meas_{s-1}((H_\delta\cap Q)-\delta w)$, and
	\[
	(H_\delta\cap Q)-\delta w=H\cap (\mathcal F+C-\delta w)=\mathcal F+(H\cap (C-\delta w))
	\]
	as $\mathcal F\subseteq H$.
	Hence, there is a positive constant $a$ (which is the determinant of the matrix  of a 
	suitable linear change of variables in $H$) such that  
	\[\meas_{s-1}(H_\delta\cap Q)=a\meas_k(\mathcal F)\meas_{s-k-1}(H\cap (C-\delta w)).\]
	We conclude the proof of \ref{item:volume_s-1} for $\mathcal P$ simplex as  
	\[
	\meas_{s-k-1}(H\cap (C-\delta w))
	=\meas_{s-k-1}(H_\delta\cap C)=\delta^{s-k-1}\meas_{s-k-1}(H_1\cap C).
	\]

	For part \ref{item:volume_s-1_kappa}, we observe that 
	$$
	\meas_{s-1}(H_\delta\cap\mathcal P\cap T_{\leq\kappa}) \leq \meas_{s-1}(H_{\delta}\cap Q\cap T_{\leq\kappa}),
	$$
	as $\mathcal P\subseteq Q$.
	 Now
	$$
	Q\cap T_{\leq\kappa}= (\mathcal F+C)\cap T_{\leq\kappa}\subseteq( \mathcal F\cap T_{\leq\kappa}) + (C\cap T_{\leq\kappa}).
	$$	 
	 Indeed, for $f\in \mathcal F$ and $c\in C$ such that $f+c\in T_{\leq\kappa}$, since $\mathcal F, C\subseteq T+\RR_{\geq0}u$ we can write $f=t+\alpha u$ and $c=t'+\alpha'u$ with $t,t'\in T$ and $\alpha,\alpha'\geq 0$. Since $\alpha+\alpha'\leq \kappa$, we get $\alpha,\alpha'\leq\kappa$, thus $f\in\mathcal F\cap T_{\leq\kappa}$ and $c\in C\cap T_{\leq\kappa}$. Thus
	 \begin{multline} \label{eq:volume:decomposition}
	 H_{\delta}\cap Q\cap T_{\leq\kappa}
	 \subseteq H_\delta\cap (( \mathcal F\cap T_{\leq\kappa}) + (C\cap T_{\leq\kappa})) 
	 = ( \mathcal F\cap T_{\leq\kappa}) + (H_\delta\cap C\cap T_{\leq\kappa}).
	  \end{multline}
	  as $\mathcal F\subseteq H$.
	 Thus
	$$
	\meas_{s-1}(H_{\delta}\cap Q\cap T_{\leq\kappa})
	\ll \meas_k(\mathcal F\cap T_{\leq\kappa})
	\meas_{s-1-k}(H_{\delta}\cap C\cap T_{\leq\kappa}),
	$$ 
	where the implicit constant is independent of $\kappa$ and $\delta$.
	Since $\mathcal P\subseteq T+\RR_{\geq0} u$ and $\mathcal P\nsubseteq T$, 	
	there is $j\in\{1,\dots,s\}$ 
	such that $(\RR_{\geq0} v_{j})\cap T_{\leq\kappa}$ is bounded, i.e., 
	$(\RR_{\geq0} v_{j})\cap T_{\leq\kappa}=[0,1]\kappa a_{j}v_{j}$ for some 
	$a_{j}>0$ independent of $\kappa$. 
	Let $$T_{\leq\kappa, j}:=\left \{\sum_{i=1}^s\lambda_i v_i:0\leq \lambda_{j}\leq \kappa a_{j} \right \}.$$

	If $F\not\subseteq T$, we can choose $j\leq k$. Then 
	$\mathcal F\cap T_{\leq\kappa}
	\subseteq \mathcal F\cap T_{\leq\kappa, j}$.
	Since $\mathcal F$ is a simplex (it is a face of a simplex),
	$$\meas_k(\mathcal F\cap T_{\leq\kappa})
	\ll \kappa a_{j}\meas_{k-1}(F_{j})\ll\kappa,$$ 
	where $F_{j}$ is the maximal face of $\mathcal F$ that does not contain $v_{j}$.
	Moreover, $$\meas_{s-1-k}(H_\delta \cap C\cap T_{\leq\kappa})
	\leq \meas_{s-1-k}(H_\delta \cap C)\ll \delta^{s-1-k},$$
	where the implicit constant is independent of $\kappa$ and $\delta$.

	If $F\subseteq T$, then $j\geq k+1$, and $C\cap T_{\leq\kappa} \subseteq C\cap T_{\leq\kappa, j}$. 	
	Hence, \begin{multline*}
	\meas_{s-1-k}(H_\delta\cap C\cap T_{\leq\kappa}) \\
	\leq\delta^{s-1-k}\meas_{s-1-k}(H_1\cap C\cap (\delta^{-1}T_{\leq\kappa,j}))
	\ll  \delta^{s-2-k}\min\{\delta,\kappa\},
	\end{multline*}
	and \ref{item:volume_s-1_kappa} for $\mathcal P$ simplex follows.
	
	For part \ref{item:volume_s-2_delta_kappa}, we observe that 
	$$
	\meas_{s-2}(H_\delta\cap \mathcal P \cap T_\kappa) \leq \meas_{s-2}(H_\delta\cap Q \cap T_\kappa),
	$$
	as $\mathcal P\subseteq Q$.
	We have
	$
	H_\delta\cap Q\cap T_\kappa=H_\delta\cap (\mathcal F+C)\cap T_\kappa.
	$
	
	If $\mathcal F\subseteq T$, and $d>k$. Let $C'=C\cap T$. Up to rearranging the indices $\{k+1,\dots,s\}$, we can assume that $C'$ is the cone (with vertex $v_0$) generated by $v_{k+1},\dots,v_{d}$. Let $C''$ be the cone (with vertex $v_0$) generated by $v_{d+1},\dots,v_s$. Then $C=C'+C''$, and
	 $$H_\delta\cap Q\cap T_\kappa=\mathcal F+(H_\delta\cap (C'+(C''\cap T_\kappa))).$$
	  Thus 
	\begin{align*}
	&\meas_{s-2}(H_\delta\cap Q\cap T_\kappa)
	\ll\meas_{s-k-2}(H_\delta\cap (C'+(C''\cap T_\kappa)))\\
	&\ll \delta^{s-k-2}\meas_{s-k-2}(H_1\cap (C'+(C''\cap T_\frac{\kappa}{\delta})))\\
	&\ll \delta^{s-k-2}\meas_{s-k-2}(p(C''\cap T_\frac{\kappa}{\delta}))
	\ll \delta^{d-k-1}\min\{\delta,\kappa\}^{s-d-1},
	\end{align*}
	where $p$ is the projection onto $H_1$ along $v_{d}$,
	as $$(H_1\cap (C'+(C''\cap T_\frac{\kappa}{\delta})))\subseteq C'\cap H_1+p(C''\cap T_\frac{\kappa}{\delta}).$$ The bound holds also if $d=k$, as $\delta^{-1}\min\{\delta,\kappa\}^{s-k-1}\leq \min\{\delta,\kappa\}^{s-k-2}.$
	
	If $\mathcal F\not\subseteq T$, let $F'=\mathcal F\cap T$. Up to rearranging the indices $\{1,\dots,k\}$ we can assume that the vertices of the simplex $F'$ are $v_0,\dots,v_a$, where $a=\dim F'$. 
	Let $C'$ be the cone (with vertex $v_0$) generated by $v_{a+1},\dots,v_k$.
 	Then $\mathcal F\subseteq F'+C'\subseteq H$, and hence
	$$
	H_\delta\cap Q\cap T_\kappa
	\subseteq (F'+C'+(H_\delta\cap C))\cap T_\kappa 
	\subseteq F'+ (C'+(H_\delta\cap C))\cap T_\kappa,
	$$
	as $F'\subseteq T$. Note that $C'\cap T_\kappa$ is bounded as $v_{a+1},\dots,v_k\notin T$. Thus
	$$
	\meas_{s-2}(H_\delta\cap Q\cap T_\kappa)\ll\meas_{s-2-a}((C'+(H_\delta\cap C))\cap T_\kappa).
	$$
	Then
	\begin{align*}
	&\meas_{s-2-a}((C'+(H_\delta\cap C))\cap T_\kappa)\\
	&\ll \kappa^{s-2-a}\meas_{s-2-a}\left(\left(C'+(H_{\frac \delta \kappa}\cap C)\right)\cap T_1\right)\\
	 &\ll \kappa^{s-2-a}\meas_{s-k-1}(p(H_{\frac \delta \kappa}\cap C))
	 \ll \kappa^{k-a-1}\delta^{s-k-1},
	\end{align*}
	where $p$ is the projection onto $T_1$ along $v_{a+1}$,
	as $$\left(C'+(H_{\frac \delta \kappa}\cap C)\right)\cap T_1\subseteq C'\cap T_1+p(H_{\frac \delta \kappa}\cap C).$$

Now we turn to part \ref{item:volume_proj}. Up to translating $V$ along an orthogonal direction, we can assume that $v_0\in V$. We denote by $V^{\perp}$ the affine space through $v_0$ orthogonal to $V$.	
	If $p_V|_{H_\delta}$ is not surjective, then $\dim p_V(H_\delta\cap\mathcal P\cap T_{\leq\kappa})<m$ and hence $\meas_mp_V(H_\delta\cap\mathcal P\cap T_{\leq \kappa})=0$. Thus we can assume without loss of generality that $p_V|_{H_\delta}: H_\delta\to V$ is surjective, i.e., $V^{\perp}\nsubseteq H$. 
	By \eqref{eq:volume:decomposition} 
	$$
	 H_\delta\cap\mathcal P\cap T_{\leq \kappa}\subseteq (\mathcal F\cap T_{\leq\kappa})+(H_\delta\cap C).
	$$	
	If $\mathcal F\nsubseteq V^{\perp}$, there exists $j\in\{1,\dots,k\}$ such that $v_{j}\notin V^{\perp}$. Without loss of generality we can assume that $j=1$. 
	 If $\mathcal F\nsubseteq T$ there exists $j'\in\{1,\dots,k\}$ such that $\RR_{\geq 0}v_{j'}\cap T_{\leq \kappa}$ is bounded. Then $v_1+a v_{j'}\notin T\cup V^{\perp}$ for some $a\in\RR$ positive, and hence $\RR_{\geq 0}(v_1+av_{j'})\cap T_{\leq \kappa}=[0,1]\kappa b v$ for some $b\in\RR_{> 0}$ and $p_V|_{\RR (v_1+bv_{j'})}$ is injective.
	Let $$ 
	v=\begin{cases} 
	\kappa(v_1+av_{j'}) & \text{ if } \mathcal F\nsubseteq T \text{ and } \mathcal F\nsubseteq V^{\perp}, \\ 
	v_1 & \text{ if }\mathcal F\subseteq T \text{ or } \mathcal F\subseteq V^{\perp}.
	\end{cases}
	$$
	Then $v, v_2,\dots,v_s$ is a basis of $\RR^s$. Let $A_1,\dots,A_k\in\RR$ be positive constants such that 
	\begin{equation*}\label{eq:volume:projF}
	\mathcal F\cap T_{\leq\kappa}\subseteq [- A_1, A_1] v\times \prod_{i=2}^k[-A_i, A_i]v_i.
	\end{equation*}
	For $i\in\{k+1,\dots,s\}$, let $a_i\in\RR_{>0}$ such that $\RR v_i\cap H_1=a_iv_i$. Let $A_{k+1},\dots,A_{s-1}\in\RR$ be positive constants such that $H_1\cap C\subseteq \prod_{i=k+1}^{s-1}[-A_i, A_i](a_iv_i-a_sv_s)$.
	Since $H_{\delta}\cap C=\delta(H_1\cap C)$ we have 
	$$H_{\delta}\cap C\subseteq \prod_{i=k+1}^{s-1}[-\delta A_i, \delta A_i](a_iv_i-a_sv_s).$$ 
	
	If $C\nsubseteq V$, we can assume without loss of generality that $e_s\notin V$. Then there is $e\in V^{\perp}$ such that $v, v_2,\dots,v_{s-1},e$ is a basis of $\RR^s$.
	Let $A_s\in\RR$ be a positive constant such that 
	\begin{equation*}
	(F\cap T_{\leq\kappa})+(H_\delta\cap C)\subseteq [- A_1,  A_1] v\times \prod_{i=2}^{k}[-A_i, A_i]v_i\times \prod_{i=k+1}^{s-1}[-\delta A_i, \delta A_i]v_i\times [-A_s, A_s]e.
	\end{equation*}
	Let $e_{m+1},\dots, e_s$ be a basis of $V^{\perp}$ with $e_s=e$. Then there is a set $I\subseteq\{k+1,\dots,s-1\}$ of cardinality at least $\dim (H_\delta\cap C)-(\dim V^{\perp}-1)=m-k$ such that $\{v_i:i\in I\}\cup\{e_{m+1},\dots,e_s\}$ is a set of linearly independent vectors, and $\{v_i:i\in I\}\cup\{v, e_{m+1},\dots,e_s\}$ is a set of linearly independent vectors if $\mathcal F\nsubseteq V^{\perp}$.
	Without loss of generality we can assume that $I\supseteq\{k+1,\dots,m\}$. Let $e_i=p_V(v_i)$ for $i\in\{k+1,\dots,m\}$ and complete to a basis $e_1,\dots,e_m$ of $V$ with $e_1=p_V(v)$ if $\mathcal F\nsubseteq V^{\perp}$.
	Then there are positive constants $B_1,\dots,B_m\in \RR$ such that 
	\begin{equation*}\label{eq:volume:CnotinV}
	p_V(H_\delta\cap\mathcal P\cap T_{\leq\kappa})\subseteq \prod_{i=1}^{k}[-B_i, B_i]e_i\times \prod_{i=k+1}^{m}[-\delta B_i, \delta B_i]e_i.
	\end{equation*} 
	Recalling the definition of $v$, we conclude that if $C\nsubseteq V$
	\begin{equation}\label{eq:volume:projsimplex}
	\meas_m(p_V(H_\delta\cap\mathcal P\cap T_{\leq\kappa}))\ll\begin{cases}\kappa\delta^{m-k} & \text{ if } \mathcal F\nsubseteq T \text{ and } \mathcal F\nsubseteq V^{\perp} \\
	\delta^{m-k} & \text{ otherwise.} 	
	\end{cases}
	\end{equation}
	If $C\subseteq V$, then $m\geq s-k$ and $p_V(H_\delta\cap C)=H_\delta\cap C$. For $i\in\{2,\dots,s-k\}$, let $e_i= a_{s-i+1}v_{s-i+1}-a_sv_s$, and complete to a basis $e_1,\dots,e_m$ of $V$, with $e_1=p_V(v)$ if $\mathcal F\nsubseteq V^{\perp}$. Then there are positive constants $B_1,\dots,B_m\in\RR$ such that 
	\begin{gather*}\label{eq:volume:CinV}
	p_V(H_\delta\cap\mathcal P\cap T_{\leq\kappa})\subseteq [- B_i, B_i]e_1 \times \prod_{i=2}^{s-k}[-\delta B_i, \delta B_i]e_i\times \prod_{i=s-k+1}^{m}[- B_i, B_i]e_i.
	\end{gather*}
	Since $m\leq s-1$, we have that $m-k\leq s-k-1$. Thus,
	recalling the definition of $v$, we conclude that for $\delta < 1$ the bounds \eqref{eq:volume:projsimplex} hold also if $C\subseteq V$.
			
	Thus we proved \ref{item:volume_s-1}, \ref{item:volume_s-1_kappa}, \ref{item:volume_s-2_delta_kappa}, and \ref{item:volume_proj} for $\mathcal P$ simplex.
	
	\emph{Step 2.} We complete the proof of \ref{item:volume_s-1}, \ref{item:volume_s-1_kappa}, and \ref{item:volume_s-2_delta_kappa}. If $\mathcal P$ is not a simplex, 
	let $\mathcal P_1,\dots,\mathcal P_N\subseteq\RR^s$ be simplices of dimension $s$ 
	such that $\mathcal P=\bigcup_{i=1}^N\mathcal P_i$ is a triangulation of $\mathcal P$. 
	Then 
	\begin{gather}	
	\meas_{s-1}(H_\delta\cap\mathcal P)=\sum_{i=1}^N\meas_{s-1}(H_\delta\cap\mathcal P_i),\\
	\meas_{s-1}(H_\delta\cap\mathcal P\cap T_{\leq\kappa})
	=\sum_{i=1}^N\meas_{s-1}(H_\delta\cap\mathcal P_i\cap T_{\leq\kappa}),\\
	\meas_{s-2} ( H_\delta \cap \mathcal P \cap T_{\kappa})
	=\sum_{i=1}^N\meas_{s-2} ( H_\delta \cap \mathcal P_i \cap T_{\kappa}),\\
	\meas_m (p_V(H_\delta\cap \mathcal P\cap T_{\leq\kappa}))\leq\sum_{i=1}^N\meas_m (p_V(H_\delta\cap \mathcal P_i\cap T_{\leq\kappa})).
	\end{gather}

	Since $\bigcup_{i=1}^N(\mathcal P_i\cap\mathcal F)=\mathcal F$, 
	there is at least one index $i\in\{1,\dots,N\}$ such that $\dim(\mathcal P_i\cap\mathcal F)=k$. 
	For $\delta$ small enough, we have $H_\delta\cap\mathcal P_i\neq\emptyset$ if and only if 
	$\mathcal P_i\cap\mathcal F\neq\emptyset$. 
	Therefore, parts \ref{item:volume_s-1}, \ref{item:volume_s-1_kappa} 
	and \ref{item:volume_s-2_delta_kappa} follow from Step 1. If $\delta\leq\kappa$, \ref{item:volume_proj} follows from Step 1.
	
	\emph{Step 3.} It remains to prove \ref{item:volume_boundary} and \ref{item:volume_boundary_kappa}.
	Let $\mathcal F_1,\dots, \mathcal F_M$ be the faces of $\mathcal P$. Then 
	\begin{gather}
	\partial (H_{\del}\cap \calP)
	=\bigcup_{i=1}^M (H_{\del}\cap \mathcal F_i),\\
	\partial (H_{\del}\cap \calP\cap T_{\leq\kappa})
	=\bigcup_{i=1}^M (H_{\del}\cap \mathcal F_i\cap T_{\leq\kappa})
	\cup (H_{\del}\cap \calP\cap T)
	\cup (H_{\del}\cap \calP\cap (T_{\kappa})).
	\end{gather}
	
	If $k=0$, then $\mathcal F$ is a vertex of $\mathcal P$. 
	Up to a translation, which is a volume preserving automorphism of $\RR^s$, 
	we can assume that $\mathcal F$ is the origin of $\RR^s$. 
	We denote by $C$ the cone with vertex $\mathcal F$ generated by $\mathcal P$. 
	Then, for $\delta$ small enough, 
	$\partial(H_\delta\cap\mathcal P)=\delta\partial(H_1\cap Q)$,
	and $\meas_{s-2}(\partial(H_\delta\cap\mathcal P))=\delta^{s-2}\meas_{s-2} (\partial(H_1\cap Q))$.
	 If $s=2$, part \ref{item:volume_boundary} holds. Hence, it remains to prove it 
	 for $s\geq3$ and $k\geq1$. 
	Let $\mathcal F_1,\dots,\mathcal F_M$ be the faces of $\mathcal P$, 
	then $\partial (H_\delta\cap\mathcal P)=\bigcup_{i=1}^M(H_\delta\cap\mathcal F_i)$. 
	For $\delta$ small enough, we have $H_\delta\cap\mathcal F_i\neq\emptyset$ 
	if and only if $\mathcal F_i\cap\mathcal F\neq\emptyset$ and $\mathcal F_i\nsubseteq\mathcal F$. 
	Moreover, for $\delta$ small enough we can assume that $H_\delta$ 
	does not contain any vertex of $\mathcal P$. 
	Since $\meas_{s-2}(H_\delta\cap\mathcal F_i)=0$ whenever $H_\delta\cap\mathcal F_i$ 
	has dimension strictly smaller than $s-2$, we have 
	$\meas_{s-2}(\partial(H_\delta\cap\mathcal P))=\sum_{i=1}^M\meas_{s-2}(H_\delta\cap\mathcal F_i)$, 
	where the sum actually runs over the maximal faces of $\mathcal P$ that intersect $\mathcal F$. 
	Let $\widetilde {\mathcal F}$ be an $(s-1)$-dimensional face of $\mathcal P$ that intersects 
	$\mathcal F$ such that $\widetilde {\mathcal F}\neq\mathcal F$, 
	and let $\tilde k=\dim(\widetilde{\mathcal F}\cap\mathcal F)$. 
	By  part \ref{item:volume_s-1} applied replacing $\mathcal P$ by $\widetilde{\mathcal F}$, we have 
	$\meas_{s-2}(H_\delta\cap\widetilde{\mathcal F})
	=\widetilde c \delta^{s-2-\tilde k}+O(\delta^{s-1-\widetilde k})$.
	If $k\leq s-2$, there is an $(s-1)$-dimensional face $\widetilde {\mathcal F}$ of $\mathcal P$ 
	that contains $\mathcal F$, hence $\tilde k=k$, and we conclude.
	If $k=s-1$, then there is an $(s-1)$-dimensional face $\widetilde {\mathcal F}$ 
	of $\mathcal P$ such that $\tilde k=s-2$, hence 
	$\meas_{s-2} (\partial( H_\delta \cap \mathcal P))= c +O(\delta)$.

	For $\delta$ and $\kappa$ small enough, we have $H_{\del}\cap \mathcal F_i\cap T_{\leq\kappa}\neq\emptyset$ 
	 only if $H\cap \mathcal F_i\cap T\neq\emptyset$ and $\mathcal F_i\nsubseteq \mathcal F$. 
	Moreover, for $\delta$ and $\kappa$ small enough we can assume that $H_{\del}$ does not contain any 
	vertex of $\mathcal P$ and $T_{\leq\kappa}$ does not contain any vertex of $\mathcal P$ not in $\mathcal P\cap T$. 
	Since $\meas_{s-2}(H_\delta\cap\mathcal F_i \cap \mathcal T_{\leq\kappa})=0$ whenever 
	$H_\delta\cap\mathcal F_i\cap T_{\leq\kappa}$ has dimension strictly smaller than $s-2$, we have 
	$$\meas_{s-2}(\partial (H_{\del}\cap \calP\cap T_{\leq\kappa}))
	=\sum_{i=1}^M\meas_{s-2}(H_\delta\cap\mathcal F_i\cap T_{\leq\kappa})
	+\meas_{s-2}(H_{\del}\cap \calP\cap T_\kappa),$$ 
	where the sum actually runs over the maximal faces of $\mathcal P$ that intersect $\mathcal F\cap T$. 
	
	Let $\widetilde{\mathcal F}$ be an $(s-1)$-dimensional face of $\mathcal P$ that intersects $\mathcal F\cap T$
	such that $\widetilde{\mathcal F}\neq\mathcal F$, and let 
	$\tilde k=\dim(\widetilde{\mathcal F}\cap\mathcal F)$. By part \ref{item:volume_s-1_kappa} and \ref{item:volume_s-1} applied replacing 
	$\mathcal P$ by $\widetilde{\mathcal F}$, we have
	$$
	\meas_{s-2}(H_\delta\cap\widetilde{\mathcal F}\cap T_{\leq\kappa})
	\ll \begin{cases}
	\kappa\delta^{s-2-\tilde k} & \text{ if }\mathcal F\cap\widetilde{\mathcal F}\nsubseteq T,\\
	\min\{\kappa,\delta\}\delta^{s-3-\tilde k} &\text{ if } \mathcal F\cap\widetilde{\mathcal F}\subseteq T \text{ and } \widetilde{\mathcal F}\nsubseteq T,\\
	\delta^{s-2-\widetilde k} & \text{ if } \widetilde{\mathcal F}\subseteq T.
	\end{cases}
	$$
	If $k\leq s-2$, there is an $(s-1)$-dimensional face $\widetilde{\mathcal F}$ of $\mathcal P$ that contains $\mathcal F$, hence $\tilde k=k$ and we get 
	$$
	\meas_{s-2} \partial (H_{\del}\cap \widetilde{\mathcal F}\cap T_{\leq\kappa}) \ll \kappa \del^{s-2-k} 
	$$
	whenever $\widetilde{\mathcal F}\cap\mathcal F\nsubseteq T$, which is the case if $\mathcal F\nsubseteq T$.
	Then part \ref{item:volume_boundary_kappa} follows by part \ref{item:volume_s-2_delta_kappa}.	
\end{proof}

In the next section we apply the proposition above to the polytope $\mathcal P$  in Theorem \ref{lemhyp3} with $H=\{\sum_{i=1}^s t_i=a\}$, $H_\delta=\{\sum_{i=1}^s t_i=a-\delta\}$, $T=\{t_{i_0}=0\}$,  $T_{\leq\kappa}=\{0\leq t_{i_0}\leq\kappa\}$ for given $i_0\in\{1,\dots,s\}$, and $\mathcal P_{i_0,\kappa}=\mathcal P\cap T_{\leq\kappa}$. Since $\mathcal P$ is a full dimensional polytope, 
$$\meas_{s-1} ( H_\delta \cap \mathcal P)= 
	c_\calP \delta^{s-1-k}+O(\delta^{s-k})$$
	for $\delta$ sufficiently small, where $c_\calP$ is a positive constant depending only on $\mathcal P$ and $H$.
	Since the face $F=H\cap\mathcal P$ is not contained in any coordinate hyperplane of $\RR^s$, we have $F\nsubseteq T$ and hence
	$$
	\meas_{s-1} ( H_\delta \cap \mathcal P \cap T_{\leq \kappa}) \ll 
 	\kappa \delta^{s-1-k}, 
	$$
	and for $V$ any coordinate subspace of dimension $m\geq k$, 
	$$
	\meas_mp_V(H_\delta\cap\mathcal P\cap T_{\leq \kappa})\ll
	\kappa \delta^{m-k},
	$$
	Notice that if $m=k$, then the volume of the projection is independent of the choice of $\delta$, hence holds also for $\delta=0$.
	
\section{Hyperbola method}
\label{sec:hyperbola_method}

We consider a function $f: \N^s\rightarrow \R_{\geq 0}$ with the following property.\\

{\em Property I}: Assume that there are non-negative real constants $C_{f,M}\leq C_{f,E}$ and $\vardel>0$ and $\varpi_i>0$, $1\leq i\leq s$ such that for all $B_1,\ldots, B_s\in \R_{\geq 1}$ we have
$$\sum_{1\leq y_i\leq B_i,\ 1\leq i\leq s} f(\bfy)= C_{f,M} \prod_{i=1}^s B_i^{\varpi_i} +O\left(C_{f,E}  \prod_{i=1}^s B_i^{\varpi_i} \left(\min_{1\leq i\leq s} B_i\right)^{-\vardel}\right)$$
where the implied constant is independent of $f$.\\

Let $B$ be a large real parameter, $\calK$ a finite index set and $s\in \N$. Let $\alp_{i,k}\geq 0$ for $1\leq i\leq s$ and $k\in \calK$.

Our goal is to evaluate the sum (if finite)
$$S^{f}(B)=\sum_{\substack{\prod_{i=1}^s y_i^{\alp_{i,k}}\leq B,\ \forall k\in \calK\\ y_i\in \N, 1\leq i\leq s }}f(\bfy).$$

We start with a heuristic for the expected growth of the sum $S^f(B)$. Let $B$ be a large real parameter. Consider the contribution to the sum $S^f(B)$ from a dyadic box where each $y_i\sim B^{t_i\varpi_i^{-1}}$ say for real parameters $t_i\geq 0$ (for example we could think of $\tfrac{1}{2} B^{t_i\varpi_i^{-1}}\leq y_i\leq B^{t_i\varpi_i^{-1}}$, $1\leq i\leq s$). Such a box is expected to contribute about
$$\sum_{y_i\sim B^{t_i\varpi_i^{-1}}, 1\leq i\leq s} f(\bfy) \sim B^{\sum_{i=1}^s t_i}$$
to the sum $S^f(B)$. 
In order for such a box to lie in the summation range we roughly speaking need
\begin{equation}\label{eqn1a}
\sum_{i=1}^s \alp_{i,k} \varpi_i^{-1} t_i\leq 1, \quad k\in \calK
\end{equation}
and
\begin{equation}\label{eqn1b}
t_i\geq 0,\quad 1\leq i\leq s.
\end{equation}
The system of equations (\ref{eqn1a}) and (\ref{eqn1b}) defines a polyhedron $\calP \subset \R^s$. We make the following two assumptions on $\calP$.\\

\begin{assumption}\label{assp10}
We assume that $\calP$ is bounded and non-degenerate in the sense that it is not contained in a $s-1$ dimensional subspace of $\R^s$. 
\end{assumption}

\begin{assumption}\label{assp11}
The face $F$ on which the function $\sum_{i=1}^s t_i$ takes its maximum on $\calP$ is not contained in a coordinate hyperplane of $\R^s$ (with coordinates $t_i$). 
\end{assumption}

The linear function
$$\sum_{i=1}^s t_i $$
takes its maximum on a face of $\calP$. We call the maximal value $a$ and assume that this maximum is obtained on a $k$-dimensional face of $\calP$.\par

Let $H_{\del}$ be the hypersurface given by
$$\sum_{i=1}^s t_i=a-\del.$$
It comes equipped with an $s-1$ dimensional measure which is obtained from the pull-back of the standard Lebesgue measure to any of its coordinate plane projections. In the following we write $\meas$ for this measure. 
Note that Proposition \ref{prop:volume_section_polytope}\ref{item:volume_s-1} shows that Assumption \ref{assp10} implies that the following holds.
\\

There is a constant $c_{\calP}$ such that for $\del >0$ sufficiently small (in terms of $\calP$) we have
\begin{equation}\label{eqn65}
|\meas_{s-1} (H_{\del} \cap \calP)-c_{\calP}\del^{s-1-k}|\leq C\del^{s-k},
\end{equation}
for a sufficiently large constant $C$, depending only on $\calP$.\\

Let $\bfl\in \Z_{\geq 0}^s$ and $1<\tet \leq2$ a parameter to be chosen later. We define box counting functions

$$B_f(\bfl,\tet)= \sum_{\substack{y_i\in [\tet_i^{l_i},\tet_i^{l_i+1}), 1\leq i\leq s}} f(\bfy),$$
where 
$$\tet_i=\tet^{\varpi_i^{-1}},\quad 1\leq i\leq s.$$
Assume that $f$ satisfies Property I. Then by inclusion-exclusion we evaluate
\begin{equation*}
\begin{split}
B_f(\bfl,\tet)&=  C_{f,M} \prod_{i=1}^s \left(\tet_i^{\varpi_i (l_i+1)} - \tet_i^{\varpi_i l_i}\right)+O\left(C_{f,E} \prod_{i=1}^s \tet_i^{l_i\varpi_i} \left(\min_i  \tet_i^{l_i}\right)^{-\vardel}\right)\\
&=  C_{f,M} (\tet-1)^s \prod_{i=1}^s \tet^{l_i}+O\left(C_{f,E} \prod_{i=1}^s \tet^{l_i} \left(\min_i  \tet_i^{l_i}\right)^{-\vardel}\right).
\end{split}
\end{equation*}
We deduce that
\begin{gather*}
B_f(\bfl,\tet)= C_{f,M} (\tet-1)^s \tet^{\sum_{i=1}^s l_i}+O\left(C_{f,E} \tet^{\sum_{i=1}^s l_i} \left(\min_i  \tet^{\varpi_i^{-1}l_i}\right)^{-\vardel}\right).
\end{gather*}

Recall that we assumed in Property I that
\begin{equation*}
C_{f,M}\leq C_{f,E}.
\end{equation*}
Hence we have 
\begin{equation}\label{eqn22}
B_f(\bfl,\tet)\ll C_{f,E} \tet^{\sum_{i=1}^s l_i}.
\end{equation}
We note that the sum
$$E_1^f:=  \sum_{\substack{\prod_{i=1}^s y_i^{\alp_{i,k}}\leq B,\ k\in \calK \\ y_i\in \N,\ 1\leq i\leq s\\ \prod_{i=1}^s y_i^{\varpi_i}>B^{a}}}f(\bfy)$$
is empty.\par
Let $\tilde{A}$ be a large natural number, which we view as a parameter to be specified later. 

We set
$$S_{1,f}:=  \sum_{\substack{\prod_{i=1}^s y_i^{\alp_{i,k}}\leq B,\ k\in \calK\\ y_i\in \N,\ 1\leq i\leq s\\ y_i\geq (\log B)^{\tilde{A}} \forall 1\leq i\leq s}}f(\bfy),$$
and note that
$$S_{1,f}=  \sum_{\substack{\prod_{i=1}^s y_i^{\alp_{i,k}}\leq B,\ k\in \calK\\ y_i\in \N,\ 1\leq i\leq s\\ \prod_i y_i^{\varpi_i}\leq B^{a}\\ y_i\geq (\log B)^{\tilde{A}} \forall 1\leq i\leq s}}f(\bfy).$$

We now cover the sum $S_{1,f}$ with boxes of the form $B_f(\bfl,\tet)$. Let $\calL^+$ be the set of $\bfl\in \Z_{\geq 0}^s$ such that the following inequalities hold
\begin{gather*}
\sum_{i=1}^s \alp_{i,k} \varpi_i^{-1}l_i\leq \frac{\log B}{\log \tet},\quad k\in \calK\\
l_i+1\geq \frac{\varpi_i \tilde{A} \log \log B}{\log \tet},\quad 1\leq i\leq s.
\end{gather*}
Similarly, let $\widetilde{\calL}^-$ be the set of $\bfl\in \Z_{\geq 0}^s$ such that the following inequalities hold
\begin{gather*}
\sum_{i=1}^s \alp_{i,k} \varpi_i^{-1}(l_i +1)\leq \frac{\log B}{\log \tet},\quad k\in \calK\\
l_i\geq \frac{\varpi_i \tilde{A} \log \log B}{\log \tet},\quad 1\leq i\leq s.
\end{gather*}

Let $C_5$ be a positive constant such that
$$\sum_{i=1}^s \alp_{i,k} \varpi_i^{-1}\leq C_5, \quad k\in \calK.$$
We define $\calL^-$ to be the set of $\bfl\in \Z_{\geq 0}^s$ such that the following inequalities hold
\begin{gather*}
\sum_{i=1}^s \alp_{i,k} \varpi_i^{-1}l_i \leq \frac{\log B}{\log \tet} - C_5, \quad k\in \calK\\
l_i\geq \frac{\varpi_i \tilde{A} \log \log B}{\log \tet},\quad 1\leq i\leq s.
\end{gather*}

Then we have

$$\sum_{\bfl\in \calL^-}B_f(\bfl,\tet) = S_{1,f}^{-}\leq S_{1,f}\leq S_{1,f}^+ = \sum_{\bfl \in \calL^+}B_f(\bfl,\tet),$$
where we read the last line as a definition for $S_{1,f}^-$ and $S_{1,f}^+$.
Note that the coverings into boxes do not depend on the function $f$ but only on the summation conditions on the variables $y_i$, $1\leq i\leq s$.\par
Let $r^+(l)$ (resp. $r^-(l)$) be the set of $\bfl\in \calL^+$ (resp. $\calL^-$) such that 
$$\sum_{i=1}^s l_i=l.$$
We recall that
\begin{gather*}
B_f(\bfl,\tet)= C_{f,M} (\tet-1)^s \tet^{\sum_{i=1}^s l_i}+O\left(C_{f,E} \tet^{\sum_{i=1}^s l_i} \left(\min_i  \tet^{\varpi_i^{-1}l_i}\right)^{-\vardel}\right).
\end{gather*}
This leads to
\begin{gather*}
\sum_{\bfl\in \calL^-}B_f(\bfl,\tet)=C_{f,M} \sum_{\bfl\in \calL^-} \tet^{\sum_{i=1}^s l_i} (\tet-1)^s 
+O\left( C_{f,E}  \sum_{\bfl\in \calL^-}\tet^{\sum_{i=1}^s l_i}  (\log B)^{-\vardel \tilde{A}}\right)
\\ = (\tet -1)^s C_{f,M} \sum_{l\leq a \log B/\log \tet} r^-(l) \tet^l 
+O\left( C_{f,E}  \sum_{\bfl\in \calL^-}\tet^{\sum_{i=1}^s l_i}  (\log B)^{-\vardel \tilde{A}}\right)
\end{gather*}
Every vector $\bfl\in \calL^-$ satisfies the bound
$$\sum_{i=1}^s l_i \leq a \frac{\log B}{\log \tet}, $$
and hence
$$l_i  \leq a \frac{\log B}{\log \tet} \quad 1\leq i\leq s.$$
This leads to the bound
$$ \sum_{\bfl\in \calL^-}\tet^{\sum_{i=1}^s l_i}\ll \left(\frac{\log B}{\log \tet}\right)^s B^{a}.$$
We deduce that
\begin{gather*}
\sum_{\bfl\in \calL^-}B_f(\bfl,\tet)=(\tet -1)^s C_{f,M} \sum_{l\leq a \log B/\log \tet} r^-(l) \tet^l  \\
+O\left( C_{f,E}  \left(\frac{\log B}{\log \tet}\right)^s B^{a} (\log B)^{-\vardel \tilde{A}}\right).
\end{gather*}

Let $\tilde{r}(l)$ be the number of $\bfl\in \Z_{\geq 0}^s$ such that $\sum_{i=1}^s l_i=l$ and the following inequalities hold
\begin{gather*}
\sum_{i=1}^s \alp_{i,k} \varpi_i^{-1}l_i \leq \frac{\log B}{\log \tet}- C_5 ,\quad k\in \calK.
\end{gather*}
Note that $\tilde{r}(l)$ is the number of lattice points in the polytope given by 
\begin{gather*}
\sum_{i=1}^s t_i=l,\\
\sum_{i=1}^s \alp_{i,k} \varpi_i^{-1}t_i \leq \frac{\log B}{\log \tet} - C_5,\quad k\in \calK\\
t_i\geq 0,\quad 1\leq i\leq s.
\end{gather*}

Finally, for $1\leq i_0\leq s$ let $r_{\tilde{A}, i_0}(l)$ be the number of $\bfl\in \Z_{\geq 0}^s$ such that $\sum_{i=1}^s l_i=l$ and 
$$l_{i_0} \leq \frac{\varpi_i \tilde{A} \log \log B}{\log \tet},$$
and
\begin{gather*}
\sum_{i=1}^s \alp_{i,k} \varpi_i^{-1}l_i \leq \frac{\log B}{\log \tet},\quad k\in \calK.
\end{gather*}

Note that we have
$$r^-(l)=\tilde{r}(l)+O\left(\sum_{1\leq i_0\leq s} r_{\tilde{A},i_0}(l)\right).$$
We now stop a moment to introduce some more auxiliary polytopes. We recall that $\calP\subset \R^s$ is the polytope given by the system of equations (\ref{eqn1a}) and (\ref{eqn1b}).\par
For $1\leq i_0\leq s$ and $\kap>0$ we introduce the polytope $\calP_{i_0,\kap}$ given by the system of equations
\begin{equation*}
\sum_{i=1}^s \alp_{i,k} \varpi_i^{-1} t_i\leq 1,\quad k\in \calK
\end{equation*}
and
\begin{equation*}
t_i\geq 0,\quad 1\leq i\leq s,
\end{equation*}
and 
\begin{equation*}
t_{i_0}\leq \kap.
\end{equation*}
I.e. $\calP_{i_0,\kap}$ is obtained from intersecting $\calP$ with the halfspace $t_{i_0}\leq \kap$. Let $H_{\del}$ be defined as before, i.e. the hyperplane given by
$$\sum_{i=1}^s t_i= a-\del.$$

By Proposition \ref{prop:volume_section_polytope}\ref{item:volume_s-1_kappa} and under the Assumption of \ref{assp10} and \ref{assp11} we have the following property.\par
Let $0\leq k\leq s-1$. Assume that $\kappa>0$ and $\del>0$ are sufficiently small in terms of the data describing $\calP$. Then we have
\begin{equation}\label{eqn65a}
\meas (H_{\del}\cap \calP_{i_0,\kap}) \ll \kap\del^{s-1-k}.
\end{equation}

\begin{remark}
Note that in the case $k=0$ and where the maximal face $F$ is not contained in a coordinate hyperplane, the intersection $H_{\del}\cap \calP_{i_0,\kap}$ is empty for $\eps$ and $\del$ sufficiently small.\end{remark}

In our applications it will take $\kap$ of size $\kap\ll \frac{\log \log B}{\log B}$. We next evaluate the function $r^-(l)$ asymptotically.\par

\begin{lemma}\label{rmin}
Assume that $0\leq k\leq s-1$. Assume that Assumption \ref{assp10} and Assumption \ref{assp11} hold. 
Let $l$ be an integer with
$$(a-\del) \left(\frac{\log B}{\log \tet}-C_5\right) \leq l\leq a \left(\frac{\log B}{\log \tet}-C_5\right) ,$$
for $\del$ sufficiently small, depending only on $\calP$, as in Proposition \ref{prop:volume_section_polytope}\ref{item:volume_s-1},\ref{item:volume_s-1_kappa}. \color{black}

Then we have
\begin{gather*}
r^-(l)=c_{\calP} \left(\frac{\log B}{\log \tet}\right)^{s-1} \left( a- \frac{\log \tet}{\log B} l\right)^{s-1-k} \\
+O\left( \left(\frac{\log B}{\log \tet}\right)^{s-2}+ \left(\frac{\log B}{\log \tet}\right)^{s-1} \left( a- \frac{\log \tet}{\log B} l\right)^{s-k}\right)\\
+ O\left(   \frac{\log \log B}{\log B} \left(\frac{\log B}{\log \tet}\right)^{s-1} \left( a- \frac{\log \tet}{\log B} l\right)^{s-1-k}  \right) .
\end{gather*}
Here we read $0^0=1$.\par
If 
$$l>  a \left(\frac{\log B}{\log \tet}-C_5\right) ,$$
then we have $r^-(l)=0$. 
\end{lemma}

\begin{remark}
Note that exactly the same asymptotic also holds for $r^+(l)$, but then in the range
$$(a-\del) \frac{\log B}{\log \tet} \leq l\leq a \frac{\log B}{\log \tet}.$$
\end{remark}

\begin{proof}
We recall that $\tilde{r}(l)$ counts lattice points in the polytope $P(l,B,\tet)$ given by 
\begin{gather*}
\sum_{i=1}^s t_i=l,\\
\sum_{i=1}^s \alp_{i,k} \varpi_i^{-1}t_i \leq \frac{\log B}{\log \tet}- C_5 ,\quad k\in \calK,\\
t_i\geq 0,\quad 1\leq i\leq s.
\end{gather*}
We observe that $P(l,B,\tet)$ is equal to the polytope $\left(\frac{\log B}{\log \tet} -C_5\right)\calP$, i.e. the polytope $\calP$ blown up by a factor of $\frac{\log B}{\log \tet}-C_5$, intersected with the hyperplane $\left(\frac{\log B}{\log \tet} -C_5\right) H_{\del'}$ given by
$$\sum_{i=1}^s t_i= \left(\frac{\log B}{\log \tet} -C_5\right) (a-\del'),$$
where $\del' \geq 0$ is chosen such that
$$l= \left(\frac{\log B}{\log \tet} -C_5\right) (a-\del').$$

I.e. our task is to count lattice points in the polytope $\left(\frac{\log B}{\log \tet} -C_5\right)(\calP\cap H_{\del'})$, which is the same as counting integer lattice points in the projection of this polytope to one of the coordinate hyperplanes. Here we can apply Davenport's lemma \cite{davenport}.

\color{black}
By equation (\ref{eqn65}) \color{black} we have
\begin{gather*}
\meas P(l,B,\tet)= c_{\calP} \left(\frac{\log B}{\log \tet}-C_5\right)^{s-1} \left( a- \left(\frac{\log B}{\log \tet}-C_5\right)^{-1}l\right)^{s-1-k} \\
+O\left( \left(\frac{\log B}{\log \tet}\right)^{s-1}  \left( a-\left(\frac{\log B}{\log \tet}-C_5\right)^{-1}l\right)^{s-k}\right).
\end{gather*}
We can rewrite this as
\begin{gather*}
\meas P(l,B,\tet)= c_{\calP} \left(\frac{\log B}{\log \tet}\right)^{s-1} \left( a-\frac{\log \tet}{\log B}l\right)^{s-1-k} \\
+O\left( \left(\frac{\log B}{\log \tet}\right)^{s-1}  \left( a-\frac{\log \tet }{\log B}l\right)^{s-k} + \left(\frac{\log B}{\log \tet}\right)^{s-2}\right).
\end{gather*}

The measure of the projection of $\calP\cap H_{\del'}$ to various coordinate spaces is bounded. Hence the measure of the projections of dimension at most $s-2$ of the blown-up polytope $\left(\frac{\log B}{\log \tet} -C_5\right)(\calP\cap H_{\del'})$ is bounded by
$$ \ll \left( \frac{\log B}{\log \tet}\right)^{s-2}.$$
By Davenport's lemma \cite{davenport} and for $\del$ sufficiently small we find that
\color{black}
\begin{gather*}
\tilde{r}(l)=c_{\calP} \left(\frac{\log B}{\log \tet}\right)^{s-1} \left( a- \frac{\log \tet}{\log B}l\right)^{s-1-k} \\
+O\left( \left(\frac{\log B}{\log \tet}\right)^{s-2}+ \left(\frac{\log B}{\log \tet}\right)^{s-1} \left( a- \frac{\log \tet}{\log B}l\right)^{s-k}\right).
\end{gather*}
Finally, by the same arguments as before and equation (\ref{eqn65a}) \color{black} we have for any $1\leq i_0\leq s$ that
\begin{gather*}
 r_{\tilde{A},i_0}(l) \ll_{\tilde{A}}  \frac{\log \log B}{\log B}\left(\frac{\log B}{\log \tet}\right)^{s-1} \left( a- \frac{\log \tet}{\log B} l\right)^{s-1-k}+\left(\frac{\log B}{\log \tet}\right)^{s-2}.
\end{gather*}
Here we use again that all volumes of projections of the corresponding polytope are bounded.
\color{black}
\end{proof}

Let $\del_0>0$ be a parameter. We write
$$S_{1,f}^- = M_{1,f}^-+ E_{2,f}^- + E_{3,f}^-$$ with
$$M_{1,f}^-= (\tet -1)^s C_{f,M} \sum_{(a-\del_0)( \log B/\log -C_5)  \leq l\leq a (\log B/\log \tet - C_5)} r^-(l) \tet^l$$
and
$$E_{2,f}^-\ll C_{f,E}  \left(\frac{\log B}{\log \tet}\right)^s B^{a} (\log B)^{-\vardel \tilde{A}}$$
and
$$E_{3,f}^- \ll (\tet -1)^s C_{f,M} \sum_{l\leq (a-\del_0) (\log B/\log \tet - C_5)} r^-(l) \tet^l.$$
First we bound the error term $E_{3,f}^-$. For this we observe that if $l\leq (a-\del_0) \log B/\log \tet$, then 
$$\tet^l \ll B^{a-\del_0}.$$
Moreover, as each of the $l_i$ in the counting function $r^-(l)$ is bounded by $\ll \frac{\log B}{\log \tet}$ we have
$$\sum_{l\leq a \log B/\log \tet} r^-(l) \ll \left(\frac{\log B}{\log \tet}\right)^s.$$
This gives the estimate
$$E_{3,f}^-\ll (\tet -1)^s C_{f,M} \left(\frac{ \log B}{\log \tet}\right)^{s} B^{a-\del_0}.$$
We conclude that
$$S_{1,f}^-= M_{1,f}^- + O\left(C_{f,E}  \left(\frac{\log B}{\log \tet}\right)^s B^{a} (\log B)^{-\vardel \tilde{A}}\right),$$
as long as 
\begin{equation}\label{eqndelkap}
\Delta \tilde{A} \log \log B \leq \del_0 \log B.
\end{equation}

We now use Lemma \ref{rmin} to first evaluate the main term $M_{1,f}^-$. For $B$ sufficiently large and $\del_0$ sufficiently small, we have
\begin{equation*}
\begin{split}
&M_{1,f}^-= (\tet -1)^s C_{f,M} \sum_{(a-\del_0) \left(\frac{\log B}{\log \tet}- C_5\right) \leq l\leq a \left(\frac{\log B}{\log \tet}- C_5\right)} r^-(l) \tet^l \\
&= (\tet -1)^s C_{f,M}   \sum_{(a-\del_0) \left(\frac{\log B}{\log \tet}- C_5\right) \leq l\leq a \left(\frac{\log B}{\log \tet}- C_5\right)} c_{\calP} \left(\frac{\log B}{\log \tet}\right)^{s-1} \left( a- \frac{\log \tet}{\log B} l\right)^{s-1-k}\tet^l \\
&+O\left(C_{f,M} (\tet -1)^s  \sum_{l\leq a \frac{\log B}{\log \tet}}  \left(\frac{\log B}{\log \tet}\right)^{s-2} \tet^l \right)\\
&+O\left( C_{f,M}(\tet -1)^s   \sum_{l\leq a \frac{\log B}{\log \tet}} \left(\frac{\log B}{\log \tet}\right)^{s-1} \left( a- \frac{\log \tet}{\log B}l\right)^{s-k} \tet^l \right)\\
&+O\left( \frac{\log \log B}{\log B} (\tet -1)^s C_{f,M}  \sum_{l\leq a \frac{\log B}{\log \tet}}  \left(\frac{\log B}{\log \tet}\right)^{s-1} \left( a- \frac{\log \tet}{\log B}l\right)^{s-1-k}\tet^l \right).
\end{split}
\end{equation*}
We can also write $M_{1,f}^-$ as
\begin{equation*}
\begin{split}
& (\tet -1)^s C_{f,M} \sum_{(a-\del_0) \left(\frac{\log B}{\log \tet}- C_5\right) \leq l\leq a \left(\frac{\log B}{\log \tet}- C_5\right)} c_{\calP} \left(\frac{\log B}{\log \tet}\right)^{k} \left( a \frac{\log B}{\log \tet} - l\right)^{s-1-k}\tet^l \\
&+O\left(C_{f,M} (\tet -1)^s  \sum_{l\leq a \frac{\log B}{\log \tet}}   \left(\frac{\log B}{\log \tet}\right)^{s-2} \tet^l \right)\\
&+O\left( C_{f,M}(\tet -1)^s   \sum_{l\leq a \frac{\log B}{\log \tet}}  \left(\frac{\log B}{\log \tet}\right)^{k-1} \left(a\frac{\log B}{\log \tet} - l\right)^{s-k} \tet^l \right)\\
&+O\left( \frac{\log \log B}{\log B} (\tet -1)^s C_{f,M}  \sum_{l\leq a \frac{\log B}{\log \tet}}   \left(\frac{\log B}{\log \tet}\right)^{k} \left( a \frac{\log B}{\log \tet} - l\right)^{s-1-k}\tet^l\right)
\end{split}
\end{equation*}
or 
\begin{equation*}
\begin{split}
&C_{f,M}c_{\calP} (\log B)^k \left(\frac{\tet -1}{\log \tet}\right)^k (\tet-1)^{s-k} \\ & \sum_{(a-\del_0) \left(\frac{\log B}{\log \tet}- C_5\right) \leq l\leq a \left(\frac{\log B}{\log \tet}- C_5\right)}  \left( a \frac{\log B}{\log \tet} - l\right)^{s-1-k}\tet^l \\
&+O\left(C_{f,M} (\tet -1)^{s-1} \left(\frac{\log B}{\log \tet}\right)^{s-2} B^{a}  \right)\\
&+O\left( C_{f,M} (\log B)^{k-1}  \left(\frac{\tet -1}{\log \tet}\right)^{k-1}(\tet-1)^{s-k+1}  \sum_{0\leq l\leq a \frac{\log B}{\log \tet}} \left(\frac{\log B}{\log \tet} a- l\right)^{s-k} \tet^l \right)\\
&+O\left( \frac{\log \log B}{\log B} C_{f,M} (\log B)^k \left(\frac{\tet -1}{\log \tet}\right)^k (\tet-1)^{s-k} \sum_{0\leq l\leq a \frac{\log B}{\log \tet}}  \left( a \frac{\log B}{\log \tet} - l\right)^{s-1-k}\tet^l \right)
\end{split}
\end{equation*}
In the second line we computed the geometric series. 
In the following we assume that for $B$ large we choose $\tet$ in a way such that $a \frac{\log B}{\log \tet}$ is an integer.
Under this assumption $M_{1,f}^-$ becomes
\begin{equation*}
\begin{split}
&C_{f,M}c_{\calP} (\log B)^k \left(\frac{\tet -1}{\log \tet}\right)^k (\tet-1)^{s-k} \sum_{aC_5\leq m\leq (a-\del_0) C_5 + \del_0  \frac{\log B}{\log \tet}}  m^{s-1-k}\tet^{a\frac{\log B}{\log \tet} -m} \\
&+O\left( C_{f,M} (\tet -1)^{s-1} \left(\frac{\log B}{\log \tet}\right)^{s-2} B^{a}  \right)\\
&+O\left(  C_{f,M} (\log B)^{k-1}  \left(\frac{\tet -1}{\log \tet}\right)^{k-1}(\tet-1)^{s-k+1} \sum_{0\leq m\leq a \frac{\log B}{\log \tet}} m^{s-k} \tet^{a \frac{\log B}{\log \tet} -m} \right)\\
\\ &+O\left(C_{f,M}\frac{\log \log B}{\log B} (\log B)^k \left(\frac{\tet -1}{\log \tet}\right)^k (\tet-1)^{s-k} \sum_{0\leq m\leq a \frac{\log B}{\log \tet}}  m^{s-1-k}\tet^{a \frac{\log B}{\log \tet} -m}   \right).
\end{split}
\end{equation*}
We further rewrite this as
\begin{equation*}
\begin{split}
M_{1,f}^-&=   C_{f,M}c_{\calP} (\log B)^kB^{a} \left(\frac{\tet -1}{\log \tet}\right)^k (\tet-1)^{s-k} \sum_{a C_5\leq m\leq (a-\del_0) C_5 + \del_0  \frac{\log B}{\log \tet}} m^{s-1-k}\tet^{-m} \\
&+O\left( C_{f,M} (\tet -1)^{s-1} \left(\frac{\log B}{\log \tet}\right)^{s-2} B^{a}  \right)\\
&+O\left(  C_{f,M} (\log B)^{k-1} B^{a} \left(\frac{\tet -1}{\log \tet}\right)^{k-1}(\tet-1)^{s-k+1} \sum_{0\leq m\leq a \frac{\log B}{\log \tet}} m^{s-k} \tet^{-m} \right)\\
&+O\left(    C_{f,M}\frac{\log \log B}{\log B} (\log B)^kB^{a} \left(\frac{\tet -1}{\log \tet}\right)^k (\tet-1)^{s-k} \sum_{0\leq m\leq a \frac{\log B}{\log \tet}} m^{s-1-k}\tet^{-m}   \right)
\end{split}
\end{equation*}
We recall the notation
$$g_\ell(M,\tet)=\sum_{0\leq m\leq M} m^\ell \tet^m.$$
With this notation $M_{1,f}^- $ equals
\begin{equation*}
\begin{split}
&C_{f,M}c_{\calP} (\log B)^kB^{a} \left(\frac{\tet -1}{\log \tet}\right)^k (\tet-1)^{s-k} g_{s-1-k}((a-\del_0) C_5 + \del_0  \log B/\log \tet, \tet^{-1}) \\
&+O\left( C_{f,M}  (\log B)^kB^{a} \left(\frac{\tet -1}{\log \tet}\right)^k (\tet-1)^{s-k}\right)\\
&+O\left( C_{f,M} (\tet -1)^{s-1} \left(\frac{\log B}{\log \tet}\right)^{s-2} B^{a}  \right)\\
&+O\left(  C_{f,M} (\log B)^{k-1} B^{a} \left(\frac{\tet -1}{\log \tet}\right)^{k-1}(\tet-1)^{s-k+1} g_{s-k}(a \log B/\log \tet, \tet^{-1}) \right)\\
&+O\left( C_{f,M}\frac{\log \log B}{\log B} (\log B)^kB^{a} \left(\frac{\tet -1}{\log \tet}\right)^k (\tet-1)^{s-k} g_{s-1-k}(a\log B/\log \tet, \tet^{-1})\right).
\end{split}
\end{equation*}
We now apply Lemma \ref{lemgl2} and obtain
\begin{equation*}
\begin{split}
M_{1,f}^- &=  C_{f,M}c_{\calP} (\log B)^kB^{a} \left(\frac{\tet -1}{\log \tet}\right)^k \tet^{s-k} (s-1-k)!\\
&+O\left( C_{f,M} (\tet -1)^{s-1} \left(\frac{\log B}{\log \tet}\right)^{s-2} B^{a}  \right)\\
&+O\left(  C_{f,M} (\log B)^{k-1} B^{a} \right)\\
&+O\left(  C_{f,M}   B^{a} (\log B)^k (\tet-1)\right)\\
&+ O\left(  C_{f,M} B^{a-\del_0}(\log B)^k \left( \frac{\log B}{\log \tet}\right)^{s-1-k} +  C_{f,M} (\log B)^{k-1} \left( \frac{\log B}{\log \tet}\right)^{s-k}  \right)\\
&+O\left( C_{f,M} \frac{\log\log B}{\log B}\left( (\log B)^kB^{a}+ B^{a-\del_0}(\log B)^k \left( \frac{\log B}{\log \tet}\right)^{s-1-k} \right) \right).
\end{split}
\end{equation*}
Next we need to choose $\tet$. We assume that
\begin{equation}\label{eqn80b}
1+\frac{1}{(\log B)^{10A}}<\tet< 1+\frac{1}{(\log B)^A},
\end{equation}
where $A>s$ is a fixed parameter. Then we have
$$\left(\frac{\tet -1}{\log \tet}\right)^k \tet^{s-k}= 1+O\left(\frac{1}{(\log B)^A}\right),$$
and
$$(\tet-1) \left(\frac{\tet -1}{\log \tet}\right)^{s-2} = O\left(\frac{1}{(\log B)^A}\right).$$
Moreover we assume that we now take $\del_0$ sufficiently small such that Lemma \ref{rmin} holds. Note that for $B$ sufficiently large, we automatically have
$$   20s (A+1)  \log \log B\leq  \del_0 \log B.$$
We deduce that
\begin{equation*}
\begin{split}
M_{1,f}^- &=  (s-1-k)! C_{f,M}c_{\calP} (\log B)^kB^{a} +O\left( C_{f,M}  B^{a} (\log B)^{k-1} \log \log B \right).
\end{split}
\end{equation*}

We now turn to the treatment of the error term. Recall that we have
$$S_{1,f}^-= M_{1,f}^- + O\left(C_{f,E}  \left(\frac{\log B}{\log \tet}\right)^s B^{a} (\log B)^{-\vardel \tilde{A}}\right)$$
We now observe under the assumption of equation (\ref{eqn80b}) that we have
$$\frac{1}{\log \tet}\ll (\log B)^{10A},$$
and 
$$S_{1,f}^-= M_{1,f}^- + O\left(C_{f,E} ( \log B)^{s+10As} B^{a} (\log B)^{-\vardel \tilde{A}}\right)$$
Let $A^\dagger$ be a positive real parameter. If we take $\tilde{A}$ sufficiently large depending on $A$, $A^\dagger$, $s$ and $\vardel$, then we get
$$C_{f,E} ( \log B)^{s+10As} B^{a} (\log B)^{-\vardel \tilde{A}} \ll C_{f,E} B^{a} (\log B)^{-A^\dagger}.$$
Observe that the same calculations are also valid for $S_{1,f}^+$ in place of $S_{1,f}^-$. We deduce that
\begin{equation*}
\begin{split}
S_{1,f}&=(s-1-k)!  C_{f,M}c_{\calP} (\log B)^kB^{a} \\
&+O\left(  C_{f,M}(\log \log B) (\log B)^{k-1} B^{a}   \right) + O\left( C_{f,E} B^{a} (\log B)^{-A^\dagger}\right).
 \end{split}
 \end{equation*}

We recall that we have made the assumption that 
$$a \log B/\log \tet$$
is integral. Hence we need to show that for every $B$ sufficiently large there is a $\tet$ in the range (\ref{eqn80b}) such that this expression is integral. Note that the conditions on $\tet$ in (\ref{eqn80b}) translate into saying that
$$a \frac{\log B}{\log(1+(\log B)^{-A})}< a\frac{\log B}{\log \tet}< a \frac{\log B}{\log (1+(\log B)^{-10A})}.$$
or
$$\log B (\log B)^A \ll a \frac{\log B}{\log \tet} \ll (\log B)^{10A} (\log B),$$
which for $B$ growing certainly contains an integer.\par

More generally, let $W_i>0$, $1\leq i\leq s$ be parameters such that there exists $N>0$ with 
$$(\log B)^{\tilde{A}}<W_i< (\log B)^N,\quad 1\leq i\leq s.$$
Set
$$S_{1,f}(\bfW,B):=  \sum_{\substack{\prod y_i^{\alp_{i,k}}\leq B,\ k\in \calK\\y_i\in \N, 1\leq i\leq s \\ y_i \geq W_i \, \forall 1\leq i\leq s}}f(\bfy).$$
We rephrase our findings in the following lemma.
\begin{lemma}\label{lem40}
Let $f: \N^s\rightarrow \R_{\geq 0}$ be a function satisfying Property I. Assume that Assumption \ref{assp10} and Assumption \ref{assp11} hold. Then, for any $\tilde A$ sufficiently large in terms of $s$ and $\Delta$, we have
\begin{equation*}
\begin{split}
S_{1,f}(\bfW,B)&=(s-1-k)!  C_{f,M}c_{\calP} (\log B)^kB^{a} +O\left(  C_{f,E}(\log \log B) (\log B)^{k-1} B^{a}   \right) ,
 \end{split}
 \end{equation*}
where the implied constants may depend on $N$, $\tilde{A}$ and the polytope $\calP$.
\end{lemma}

Next we turn to the treatment of the contributions where some variables in the sum $S^f(B)$ can be small. Let $N>0$ be a parameter as above and consider $W< (\log B)^N$. For $i_0\in \{1,\ldots, s\}$ set
$$S_{i_0}(W,B):= \sum_{\substack{\prod y_i^{\alp_{i,k}}\leq B,\ k\in \calK\\y_i\in \N, 1\leq i\leq s \\ y_{i_0} \leq W}}f(\bfy). $$
We cover the region of summation by dyadic boxes, i.e. we set $\theta=2$ above. Note that for any value of $\bfl$ we have
$$B_f(\bfl,\theta)\ll C_{f,E} \theta^{\sum_{i=1}^s l_i}.$$

We define the set $\mathcal{L}_{i_0}$ to be the set of lattice points $\bfl\in\ZZ_{\geq 0}^s$ which lie in the polytope given by
$$l_{i_0} \leq \frac{\varpi_{i_0} \log W}{\log \tet}= \frac{\varpi_{i_0}\log W}{\log B} \frac{\log B}{\log \tet},$$
and
\begin{gather*}
\sum_{i=1}^s \alp_{i,k} \varpi_i^{-1}l_i \leq \frac{\log B}{\log \tet},\quad k\in \calK.
\end{gather*}
Let $r_{W,i_0}(l)$ be the number of $\bfl\in \Z_{\geq 0}^s$ in $\mathcal{L}_{i_0}$ with $\sum_{i=1}^s l_i=l$. With this notation we find that

$$S_{i_0}(W,\mathbf{B})\leq \sum_{\bfl\in \mathcal{L}_{i_0}}B_f(\bfl,\theta) \ll C_{f,E}\sum_{l\leq a \frac{\log B}{\log \theta}} r_{W,i_0}(l) \theta^l$$
We observe that $r_{W,i_0}(l)$ is the number of lattice points in the intersection of the polytope $\frac{\log B}{\log \theta} \mathcal{P}_{i_0,\kappa}$, where $\kappa=  \frac{\varpi_{i_0}\log W}{\log B}$, intersected with the hyperplane $\frac{\log B}{\log \tet} H_{\delta}$ with $\delta =a- \frac{\log \tet}{\log B} l$. Let $\del_1>0$ be a parameter to be chosen later. As above, we observe that

\begin{gather*}
\sum_{l\leq (a-\del_1)\frac{\log B}{\log \theta}} r_{W,i_0}(l) \theta^l \ll  (\log B)^s B^{a-\del_1}.
\end{gather*}

Hence, if we assume that 
\begin{equation}\label{eqndel1}
(s+1) \log \log B \leq \del_1 \log B,
\end{equation}
then we have
$$S_{i_0}(W,B)\ll C_{f,E}\sum_{(a-\del_1) \frac{\log B}{\log \theta} \leq l\leq a \frac{\log B}{\log \theta}} r_{W,i_0}(l) \theta^l + C_{f,E} (\log B)^{-1} B^a.$$

If $k=0$, then for $\kappa$ and $\delta$ sufficiently small the intersection of $H_\delta$ and $\mathcal{P}_{i_0,\kappa}$ is empty and there is nothing to bound. Hence in the following we may assume $k\geq 1$.\par
By enlarging $B$ by at most a constant factor depending on $a$, we may assume that $a\frac{\log B}{\log \tet}$ is integral. If $l=\frac{\log B}{\log \tet} a$, then $\del=0$ and the dimension of the intersection of $\mathcal{P}_{i_0,\kappa}\cap H_0$ is at most $k$. If the dimension of the intersection of $\mathcal{P}_{i_0,\kappa}\cap H_0$ is at most $k-1$, then we observe that all projections of this intersection to coordinate spaces are bounded, and hence we obtain the bound
$$r_{W,i_0}(l)\ll \left(\frac{\log B}{\log \tet}\right)^{k-1}.$$

Finally, assume that the dimension of the intersection $\mathcal{P}_{i_0,\kappa}\cap H_0$ is equal to $k$. Consider the projection of $\mathcal{P}_{i_0,\kappa}\cap H_0$ to a $k$-dimensional coordinate subspace $V$. Then we aim to bound the number of integer lattice points which are contained in the projection of $\frac{\log B}{\log \tet}(\mathcal{P}_{i_0,\kappa}\cap H_0)$ to $V$. By Proposition \ref{prop:volume_section_polytope}\ref{item:volume_proj} the $k$-dimensional volume of the projection of $\mathcal{P}_{i_0,\kappa}\cap H_0$ to $V$ is bounded by $\ll \kappa$ and the lower dimensional volumes of projections to coordinate spaces of dimension at most $k-1$ are bounded.

Then, by Davenport's lemma \cite{davenport} we have
$$r_{W,i_0}(l)\ll \left(\frac{\log B}{\log \tet}\right)^{k} \frac{\log W}{\log B} + \left(\frac{\log B}{\log \tet}\right)^{k-1}.$$
Next we consider the case $\frac{\log B}{\log \tet} (a-\del_1)\leq l<\frac{\log B}{\log \tet} a$, i.e. $0<\del\leq \del_1$. As we are only interested in an upper bound for $S_{i_0}(W,B)$, we may enlarge $W$ to $W=(\log B)^N$ with $N$ sufficiently large, such that we have
$$(s+1)\frac{\log\log B}{\log B} \leq \frac{ N \varpi_{i_0} \log \log B}{\log B}.$$
Then we can choose $\del_1$ such that (\ref{eqndel1}) is satisfied and such that $\del_1\leq \kappa$.

By Proposition \ref{prop:volume_section_polytope}\ref{item:volume_proj} the polytope $\mathcal{P}_{i_0,\kappa}\cap H_\del$ has the following properties if $\kappa$ is sufficiently small and $\del\leq \kappa$:\\
(i) If we project $\mathcal{P}_{i_0,\kappa}\cap H_\del$ to coordinate spaces of dimension $k-1$ or smaller, then the volume is bounded by an absolute constant.\\
(ii) Let $k-1< m\leq s-1$. Then the $m$-dimensional volume of the projection of $\mathcal{P}_{i_0,\kappa}\cap H_\del$ to any $m$-dimensional coordinate space is bounded by
$$\ll \kappa \del^{m-k}.$$

Then by Davenport's lemma \cite{davenport} we obtain

\begin{align*}
r_{W,i_0}(l)&\ll  \left(\frac{\log B}{\log \tet}\right)^{k-1} +\sum_{m=k}^{s-1} \left(\frac{\log B}{\log \tet}\right)^{m}   \kappa \del^{m-k} \\
&\ll \left(\frac{\log B}{\log \tet}\right)^{k-1} + \sum_{m=k}^{s-1} \left(\frac{\log B}{\log \tet}\right)^{m}  \left( a-l\frac{\log \tet}{\log B}\right)^{m-k} \frac{\log\log B}{\log B}\\
&\ll \left(\frac{\log B}{\log \tet}\right)^{k-1} +\left(\frac{\log B}{\log \tet}\right)^{k-1}\sum_{m=k}^{s-1}  \left( a\frac{\log B}{\log \tet}-l\right)^{m-k} \frac{\log\log B}{\log \tet}
\end{align*}
 
Recall that $\tet=2$. With this we obtain
\begin{align*}
S_{i_0}(W,B)&\ll C_{f,E} (\log B)^{k-1}(\log \log B) B^a + C_{f,E} \sum_{(a-\del_1) \frac{\log B}{\log \tet} \leq l< a \frac{\log B}{\log \tet}} r_{W,i_0}(l) \tet^l\\
&\ll C_{f,E} (\log B)^{k-1}(\log \log B) B^a \\
&+  C_{f,E} \left(\frac{\log B}{\log \tet}\right)^{k-1}\sum_{m=k}^{s-1} \sum_{l< a\frac{\log B}{\log \tet}}  \left(a\frac{\log B}{\log \tet}-l\right)^{m-k}  \left( \frac{\log\log B}{\log \tet} \right) \tet^l
\end{align*}
Again using that $\tet=2$ we obtain.
\begin{align*}
S_{i_0}(W,B)&\ll C_{f,E} (\log B)^{k-1}(\log \log B) B^a \\
&+ C_{f,E} (\log B)^{k-1}(\log \log B)  \sum_{m=k}^{s-1}  \sum_{l< a\frac{\log B}{\log \tet}}\left(a\frac{\log B}{\log \tet}-l\right)^{m-k}   \tet^l \\
&\ll C_{f,E} (\log B)^{k-1}(\log \log B) B^a \\
&+ C_{f,E} (\log B)^{k-1}(\log \log B)B^a  \sum_{m=k}^{s-1}  \sum_{l< a\frac{\log B}{\log \tet}}\left(a\frac{\log B}{\log \tet}-l\right)^{m-k}   \tet^{-(a \frac{\log B}{\log \tet} - l)} \\
&\ll C_{f,E} (\log B)^{k-1}(\log \log B) B^a \\
&+C_{f,E}  (\log B)^{k-1} (\log \log B)  B^a   \sum_{u\leq a\frac{\log B}{\log \tet}} u^{s-k} \tet^{-u}  \\
\end{align*}

By Lemma \ref{lemgl2} or in observing that the last sum is absolutely convergent, we find that
\begin{align*}
S_{i_0}(W,B)\ll C_{f,E} (\log B)^{k-1}(\log \log B) B^a.
\end{align*}
\color{black}

This completes the proof of Theorem \ref{lemhyp3}. 
\section{$m$-full numbers}
\label{sec:m-full}
Let $m\geq 1$ be a natural number. We recall that an integer $y$ is called $m$-full if for each prime divisor $p$ of $y$, we have that $p^m$ divides $y$.
We introduce the function that counts the number of $m$-full natural numbers less than $B$
\begin{equation*}
\begin{split}
F_m(B):&=\sharp \left\{1\leq y\leq B, v_p(y)\in\{0\}\cup\ZZ_{\geq m}\forall p\right\}\\
&= \sharp\left\{1\leq y\leq B, y \mbox{ is $m$-full}\right\}.
\end{split}
\end{equation*}

\begin{lemma}[{\cite{erdos, MR0095804}}]
For each $m\geq 1$ and $B>0$ we have
\begin{equation}\label{eqnm}
F_m(B)=C_mB^{\frac{1}{m}}+O_m\left(B^{\kap_m}\right),
\end{equation}
where $C_1=1$, $\kap_1=0$, and for $m\geq 2$, 
\begin{equation}
\label{eq:C_m_explicit}
C_m=\prod_p\left(1+\sum_{j=m+1}^{2m-1}p^{-\frac jm}\right) \qquad
\kap_m=\frac1{m+1}.
\end{equation}
\end{lemma}

For a square-free positive integer $d$, we define
\begin{equation}
\label{eq:FmBd_defin}
F_m(B,d):= \sharp\left\{1\leq y\leq B, d\mid y,\ y\mbox{ is $m$-full}\right\}.
\end{equation}
In this section we prove an asymptotic formula for the function $F_m(B,d)$.
We will first do it for the case that $d$ is a prime. We will then inductively on the number of prime factors of $d$ provide a general asymptotic formula. 
First we provide a form of inclusion-exclusion lemma, which expresses $F_m(B,p)$ for a prime number $p$ in terms of sums of the function $F_m(B)$. Before we state the lemma, we introduce a convenient piece of notation. For $r\geq 1$ and $k\in \Z$ let 
$$\rho_m(k,r):= \sharp\left\{1\leq k_1,\ldots, k_r\leq (m-1):\ k=k_1+\ldots + k_r\right\}.$$
Note that $\rho_m(k,r)$ is zero, unless $r\leq k\leq r(m-1)$.\par

\begin{lemma}\label{lem5}
For $m\geq 2$ one has
\begin{equation*}
\begin{split}
F_m(B,p) = &F_m(Bp^{-m}) +\sum_{r=1}^{\infty} \sum_{k=2r}^{2r(m-1)}\rho_m(k,2r)F_m\left(Bp^{-(2r+1)m-k}\right)\\
&+ \sum_{r=1}^\infty\sum_{k=2r-1}^{(2r-1)(m-1)}\rho_m(k,2r-1)F_m\left(Bp^{-(2r-1)m-k}\right)\\
&-\sum_{r=1}^\infty \sum_{k=2r-1}^{(2r-1)(m-1)}\rho_m(k,2r-1)F_m\left(Bp^{-2rm-k}\right)\\
&-\sum_{r=1}^\infty\sum_{k=2r}^{2r(m-1)}\rho_m(k,2r)F_m\left(Bp^{-2rm-k}\right).
\end{split}
\end{equation*}

\end{lemma}

Note that the summations in Lemma \ref{lem5} are in fact finite, as $F_m(P)=0$ if $P<1$. 

\begin{proof}
We start the proof in reinterpreting terms of the shape $F_m\left(Bp^{-K}\right)$ for some $K>0$ as
$$F_m\left(Bp^{-K}\right)= \sharp\left\{1\leq p^Kl\leq B, l \mbox{ is $m$-full}\right\}.$$
Then the right hand side in the identity in Lemma \ref{lem5} becomes
\begin{equation}\label{eqn20}
\begin{split}
RHS = &\sharp\left\{1\leq p^ml\leq B, l \mbox{ is $m$-full}\right\} \\
&+\sum_{r=1}^{\infty} \sum_{k=2r}^{2r(m-1)}\rho_m(k,2r)\sharp\left\{1\leq p^{(2r+1)m+k}l\leq B, l \mbox{ is $m$-full}\right\}\\
&+ \sum_{r=1}^\infty\sum_{k=2r-1}^{(2r-1)(m-1)}\rho_m(k,2r-1)\sharp\left\{1\leq p^{(2r-1)m+k}l\leq B, l \mbox{ is $m$-full}\right\}\\
&-\sum_{r=1}^\infty \sum_{k=2r-1}^{(2r-1)(m-1)}\rho_m(k,2r-1)\sharp\left\{1\leq p^{2rm+k}l\leq B, l \mbox{ is $m$-full}\right\}\\
&-\sum_{r=1}^\infty\sum_{k=2r}^{2r(m-1)}\rho_m(k,2r)\sharp\left\{1\leq p^{2rm+k}l\leq B, l \mbox{ is $m$-full}\right\}.
\end{split}
\end{equation}
For any $K\geq 0$ we use the identity
\begin{equation}\label{eqn20b}
\begin{split}
\sharp\left\{1\leq p^{K}l\leq B, l \mbox{ is $m$-full}\right\} &= \sharp\left\{1\leq p^{K}l\leq B, l \mbox{ is $m$-full}, p\nmid l\right\}\\
&+ \sharp\left\{1\leq p^{K+m}l\leq B, l \mbox{ is $m$-full}\right\}\\
&+\sum_{k=1}^{m-1}\sharp\left\{1\leq p^{K+m+k}l\leq B, l \mbox{ is $m$-full}, p\nmid l\right\}
\end{split}
\end{equation}
We use this identity for the terms in the third line in (\ref{eqn20}). We observe that the terms counting $1\leq p^{2rm+k}l\leq B$ with $l$ $m$-full identically cancel with the fourth line in (\ref{eqn20}). Hence we obtain
\begin{equation*}
\begin{split}
RHS = &\sharp\left\{1\leq p^ml\leq B, l \mbox{ is $m$-full}\right\} \\
&+\sum_{r=1}^{\infty} \sum_{k=2r}^{2r(m-1)}\rho_m(k,2r)\sharp\left\{1\leq p^{(2r+1)m+k}l\leq B, l \mbox{ is $m$-full}\right\}\\
&+ \sum_{r=1}^\infty\sum_{k=2r-1}^{(2r-1)(m-1)}\rho_m(k,2r-1)\sharp\left\{1\leq p^{(2r-1)m+k}l\leq B, l \mbox{ is $m$-full}, p\nmid l\right\}\\
&+ \sum_{r=1}^\infty\sum_{k=2r-1}^{(2r-1)(m-1)}\sum_{k_r=1}^{m-1}\rho_m(k,2r-1)\sharp\left\{1\leq p^{2rm+k+k_r}l\leq B, l \mbox{ is $m$-full}, p\nmid l\right\}\\
&-\sum_{r=1}^\infty\sum_{k=2r}^{2r(m-1)}\rho_m(k,2r)\sharp\left\{1\leq p^{2rm+k}l\leq B, l \mbox{ is $m$-full}\right\}.
\end{split}
\end{equation*}
Recalling the definition of the functions $\rho_m(k,r)$ we can further rewrite this as
\begin{equation}\label{eqn21}
\begin{split}
RHS = &\sharp\left\{1\leq p^ml\leq B, l \mbox{ is $m$-full}\right\} \\
&+\sum_{r=1}^{\infty} \sum_{k=2r}^{2r(m-1)}\rho_m(k,2r)\sharp\left\{1\leq p^{(2r+1)m+k}l\leq B, l \mbox{ is $m$-full}\right\}\\
&+ \sum_{r=1}^\infty\sum_{k=2r-1}^{(2r-1)(m-1)}\rho_m(k,2r-1)\sharp\left\{1\leq p^{(2r-1)m+k}l\leq B, l \mbox{ is $m$-full}, p\nmid l\right\}\\
&+ \sum_{r=1}^\infty\sum_{k=2r}^{2r(m-1)}\rho_m(k,2r)\sharp\left\{1\leq p^{2rm+k}l\leq B, l \mbox{ is $m$-full}, p\nmid l\right\}\\
&-\sum_{r=1}^\infty\sum_{k=2r}^{2r(m-1)}\rho_m(k,2r)\sharp\left\{1\leq p^{2rm+k}l\leq B, l \mbox{ is $m$-full}\right\}.
\end{split}
\end{equation}
We now use the identity (\ref{eqn20b}) for the terms in the last line in equation (\ref{eqn21}). The resulting terms with $p^{2rm+k}l\leq B$ and $p\nmid l$ cancel with the terms in the fourth line. Moreover, the terms with $1\leq p^{(2r+1)m+k}l\leq B$ and $l$ $m$-full identically cancel with the terms in the second line in (\ref{eqn21}). Hence we obtain
\begin{equation*}
\begin{split}
RHS &= \sharp\left\{1\leq p^ml\leq B, l \mbox{ is $m$-full}\right\} \\
&+ \sum_{r=1}^\infty\sum_{k=2r-1}^{(2r-1)(m-1)}\rho_m(k,2r-1)\sharp\left\{1\leq p^{(2r-1)m+k}l\leq B, l \mbox{ is $m$-full}, p\nmid l\right\}\\
&-\sum_{r=1}^\infty\sum_{k=2r}^{2r(m-1)}\sum_{k_{2r+1}=1}^{m-1}\rho_m(k,2r)\sharp\left\{1\leq p^{(2r+1)m+k+k_{2r+1}}l\leq B, l \mbox{ is $m$-full}, p\nmid l\right\}.
\end{split}
\end{equation*}
Again using the definition of the functions $\rho_m(k,2r)$ we can rewrite this as
\begin{equation}\label{eqn23}
\begin{split}
RHS &= \sharp\left\{1\leq p^ml\leq B, l \mbox{ is $m$-full}\right\} \\
&+ \sum_{r=1}^\infty\sum_{k=2r-1}^{(2r-1)(m-1)}\rho_m(k,2r-1)\sharp\left\{1\leq p^{(2r-1)m+k}l\leq B, l \mbox{ is $m$-full}, p\nmid l\right\}\\
&-\sum_{r=1}^\infty\sum_{k=2r+1}^{(2r+1)(m-1)}\rho_m(k,2r+1)\sharp\left\{1\leq p^{(2r+1)m+k}l\leq B, l \mbox{ is $m$-full}, p\nmid l\right\}.
\end{split}
\end{equation}
The last two sums in (\ref{eqn23}) cancel except for the terms with $r=1$ in the second line. Hence we get
\begin{equation*}
\begin{split}
RHS =& \sharp\left\{1\leq p^ml\leq B, l \mbox{ is $m$-full}\right\} \\
&+ \sum_{k=1}^{m-1}\sharp\left\{1\leq p^{m+k}l\leq B, l \mbox{ is $m$-full}, p\nmid l\right\}.
\end{split}
\end{equation*}
Similarly as in (\ref{eqn20b}) we now observe that on the right hand side we count exactly all $1\leq l\leq B$ such that $p\mid l$ and $l$ is $m$-full, which completes the proof of the lemma.
\end{proof}

Lemma \ref{lem5} now allows us to deduce an asymptotic formula for $F_m(B,p)$ given that we know (\ref{eqnm}). We recall that the sums in Lemma \ref{lem5} are all finite and hence we can first reorder them to take into account cancellation between different sums and then complete the resulting series to infinity. First we rewrite the expression for $F_m(B,p)$ in Lemma \ref{lem5} as

\begin{equation}\label{eqn25}
F_m(B,p) = F_m(Bp^{-m}) +\sum_{\mu=m+1}^\infty a_m(\mu) F_m(Bp^{-\mu}),
\end{equation}
with coefficients $a_m(\mu)$ that are given by
\begin{equation*}
\begin{split}
a_m(\mu)= &\sum_{r=1}^{\infty} \sum_{k=2r}^{2r(m-1)}\rho_m(k,2r)\id_{[\mu=(2r+1)m+k]}\\
&+ \sum_{r=1}^\infty\sum_{k=2r-1}^{(2r-1)(m-1)}\rho_m(k,2r-1)\id_{[\mu=(2r-1)m+k]}\\
&-\sum_{r=1}^\infty \sum_{k=2r-1}^{(2r-1)(m-1)}\rho_m(k,2r-1)\id_{[\mu=2rm+k]}\\
&-\sum_{r=1}^\infty\sum_{k=2r}^{2r(m-1)}\rho_m(k,2r)\id_{[\mu=2rm+k]}.
\end{split}
\end{equation*}
Note that all the appearing sums are in fact finite and hence we can reorder them freely. Our next goal is to get more understanding on the coefficients $a_m(\mu)$ (in particular their size) and hence we group them in a generating series. Define
$$G_m(x):=\sum_{\mu=m+1}^\infty a_m(\mu) x^\mu.$$
A first rough bound on the coefficients $a_m(\mu)$ can be obtained via the estimate
$$a_m(\mu)\leq 2\sum_{1\leq t\leq \frac{\mu}{m}} \sum_{k=t}^{t(m-1)}\rho_m(k,t)= 2 \sum_{1\leq t\leq \frac{\mu}{m}}  (m-1)^t .$$
For $m\geq 3$ we obtain
$$a_m(\mu)\leq 2(m-1)^{\frac{\mu}{m}+1},$$
whereas for $m=2$ we have the estimate
$$a_m(\mu)\leq 2 \frac{\mu}{m}.$$
In particular we deduce that there is some constant $R_m$ only depending on $m$ such that the power series $G_m(x)$ is absolutely convergent for $|x|<R_m$. Moreover, in choosing $R_m$ sufficiently small we can also assume that the sum  
$$\sum_{t=1}^{\infty} \sum_{k=t}^{t(m-1)}\rho_m(k,t)x^{tm+k}$$
is absolutely convergent. Our next goal is to write the generating series $G_m(x)$ as a fractional function and in this way realise that it has a larger radius of absolute convergence than the bound that is obtained from the very rough estimate on $a_m(\mu)$. For this we observe that for $|x|<R_m$ we can express $G_m(x)$ as
\begin{equation*}
\begin{split}
G_m(x)= &\sum_{r=1}^{\infty} \sum_{k=2r}^{2r(m-1)}\rho_m(k,2r)x^{(2r+1)m+k}\\
&+ \sum_{r=1}^\infty\sum_{k=2r-1}^{(2r-1)(m-1)}\rho_m(k,2r-1)x^{(2r-1)m+k}\\
&-\sum_{r=1}^\infty \sum_{k=2r-1}^{(2r-1)(m-1)}\rho_m(k,2r-1)x^{2rm+k}\\
&-\sum_{r=1}^\infty\sum_{k=2r}^{2r(m-1)}\rho_m(k,2r)x^{2rm+k}.
\end{split}
\end{equation*}
We can now compute the generating function $G_m(x)$ as
\begin{equation*}
\begin{split}
G_m(x)= &\sum_{r=1}^{\infty} x^{(2r+1)m}\left(\sum_{k=1}^{m-1}x^k\right)^{2r}+ \sum_{r=1}^\infty x^{(2r-1)m}\left(\sum_{k=1}^{m-1}x^k\right)^{2r-1}\\
&-\sum_{r=1}^\infty x^{2rm}\left(\sum_{k=1}^{m-1}x^k\right)^{2r-1} -\sum_{r=1}^\infty x^{2rm}\left(\sum_{k=1}^{m-1}x^k\right)^{2r}.
\end{split}
\end{equation*}
In the area of absolute convergence one may reorder the sums as
\begin{equation*}
\begin{split}
G_m(x)= &\sum_{t=1}^{\infty} (-1)^{t} x^{(t+1)m}\left(\sum_{k=1}^{m-1}x^k\right)^{t}+ \sum_{t=1}^{\infty} (-1)^{t+1} x^{tm}\left(\sum_{k=1}^{m-1}x^k\right)^{t}\\
=& (x^m-1)(-1)x^m\left(\sum_{k=1}^{m-1}x^k\right) \left(1-  (-1)x^m\left(\sum_{k=1}^{m-1}x^k\right) \right)^{-1}\\
=& (1-x^m)x^mx \frac{x^{m-1}-1}{x-1}\left(1+x^{m+1} \frac{x^{m-1}-1}{x-1}\right)^{-1}\\
=&  (1-x^m)x^{m+1} \frac{x^{m-1}-1}{x-1}(x-1)\left(x-1+x^{m+1} (x^{m-1}-1)\right)^{-1}\\
=&  (1-x^m)x^{m+1} (x^{m-1}-1)\left(x-1+x^{m+1} (x^{m-1}-1)\right)^{-1}\\
=&  (1-x^m)x^{m+1} (x^{m-1}-1)\left(x-1+x^{2m}-x^{m+1}\right)^{-1}
\end{split}
\end{equation*}
We observe that
$$x^{2m}-x^{m+1}+x-1= (x^m-1)(x^m     +1)-x(x^m-1) = (x^m-1)(x^m-x+1).$$
Hence we obtain
$$G_m(x)=- x^{m+1} (x^{m-1}-1)\left( x^m-x+1\right)^{-1}.$$
In the interval $x\in (0,1)$ the function $x^m-x$ takes its minimum at $x=m^{-\frac{1}{m-1}}$, and at this point $x^m-x=m^{-\frac{1}{m-1}}(m^{-1}-1)$. In particular we observe that the Taylor series for $G_m(x)$ is absolutely convergent in the interval $x\in (0,1)$.\par

We can now deduce an asymptotic for $F_m(B,p)$. 

\begin{lemma}\label{lem6}
Let $m\geq 2$. 
Let $p$ be a prime number. Then we have
$$F_m(B,p)=C_mB^{\frac{1}{m}} \left(p^{-1}+G_m\left(p^{-\frac{1}{m}}\right)\right)+ O_m\left(B^{\kap_m}p^{-m\kap_m}\right).$$
Here the implicit constant is independent of $p$.
\end{lemma}

\begin{proof}
We start in recalling equation (\ref{eqn25}) 
\begin{equation*}
F_m(B,p) = F_m(Bp^{-m}) +\sum_{\mu=m+1}^\infty a_m(\mu) F_m(Bp^{-\mu}).
\end{equation*}
From the asymptotic formula in (\ref{eqnm}) we deduce that
\begin{equation*}
\begin{split}
F_m(B,p) = &C_m(Bp^{-m})^{\frac{1}{m}} +C_m\sum_{\mu=m+1}^\infty a_m(\mu) (Bp^{-\mu})^{\frac{1}{m}}\\ &+O_m\left((Bp^{-m})^{\kap_m}+ \sum_{\mu=m+1}^\infty |a_m(\mu)| (Bp^{-\mu})^{\kap_m}  \right),\\
=& C_mB^{\frac{1}{m}} \left(p^{-1}+G_m\left(p^{-\frac{1}{m}}\right)\right)\\& +O_m\left(B^{\kap_m}p^{-m\kap_m}+B^{\kap_m}p^{-m\kap_m}\sum_{\mu=m+1}^{\infty}|a_m(\mu)|2^{(-\mu+m)\kap_m}\right).
\end{split}
\end{equation*}
The last sum is absolutely convergent (consider the generating function $x^{-m}G_m(x)$ at the point $x=2^{-\kap_m}$) and hence  we have established the asymptotic
\[F_m(B,p)=C_mB^{\frac{1}{m}} \left(p^{-1}+G_m\left(p^{-\frac{1}{m}}\right)\right)+ O_m\left(B^{\kap_m}p^{-m\kap_m}\right). \qedhere
\]
\end{proof}

Next we aim to generalize Lemma \ref{lem6} to obtain an asymptotic formula for $F_m(B,d)$ for a general square-free number $d$. For this we start with a generalization of Lemma \ref{lem5}.

\begin{lemma}\label{lem7}
Let $d>0$ be a square-free integer and $p$ a prime with $p\mid d$. Write $d'=d/p$. For $m\geq 2$ one has
\begin{equation*}
\begin{split}
F_m(B,d) = &F_m(Bp^{-m},d') +\sum_{r=1}^{\infty} \sum_{k=2r}^{2r(m-1)}\rho_m(k,2r)F_m\left(Bp^{-(2r+1)m-k},d'\right)\\
&+ \sum_{r=1}^\infty\sum_{k=2r-1}^{(2r-1)(m-1)}\rho_m(k,2r-1)F_m\left(Bp^{-(2r-1)m-k},d'\right)\\
&-\sum_{r=1}^\infty \sum_{k=2r-1}^{(2r-1)(m-1)}\rho_m(k,2r-1)F_m\left(Bp^{-2rm-k},d'\right)\\
&-\sum_{r=1}^\infty\sum_{k=2r}^{2r(m-1)}\rho_m(k,2r)F_m\left(Bp^{-2rm-k},d'\right). 
\end{split} 
\end{equation*}
\end{lemma}

The proof of Lemma \ref{lem7} is exactly the same as the proof of Lemma \ref{lem5} where the condition $l$ is $m$-full is replaced by the condition that $l$ is $m$-full and $d'\mid l$. Moreover, as in equation (\ref{eqn25}) one can rewrite the identity from Lemma \ref{lem7} as
\begin{equation}\label{eqn27}
F_m(B,d) = F_m\left(Bp^{-m},d'\right) +\sum_{\mu=m+1}^\infty a_m(\mu) F_m\left(Bp^{-\mu},d'\right).
\end{equation}
Via induction on the number of prime factors of $d$ we now establish the following lemma.

\begin{lemma}\label{lem8}
 Let $d>0$ be a square-free integer. Write $\ome(d)$ for the number of prime divisors of $d$. Then for each integer $m\geq 2$ there exists a positive constant $K_m$ such that we have
$$F_m(B,d)=C_mB^{\frac{1}{m}}\prod_{p\mid d} \left(p^{-1}+G_m\left(p^{-\frac{1}{m}}\right)\right)+ O_m\left(K_m^{\ome(d)}B^{\kap_m}d^{-m\kap_m}\right).$$
Here the implicit constant is independent of $d$ and
\begin{equation}
\label{eq:G_m_explicit}
p^{-1}+G_{m}\left(p^{-\frac{1}{m}}\right)=\frac1{1+p-p^\frac{m-1}m}.
\end{equation}
For $m=1$ the asymptotic holds with $C_1=1$, $K_1=1$, $\kap_1=0$ (and $G_1=0$).
\end{lemma}

\begin{proof}
For $m=1$ the statement is immediate. Let us assume that $m\geq 2$.
If $d$ is prime, then the statement follows from Lemma \ref{lem6} (or note that if $d=1$ then the statement reduces to the assumption in (\ref{eqnm})). Let $d>0$ be squarefree and $q$ a prime with $q\mid d$. Assume that we have established the asymptotic
$$F_m\left(B,d'\right)=C_mB^{\frac{1}{m}}\prod_{p\mid d'} \left(p^{-1}+G_m\left(p^{-\frac{1}{m}}\right)\right)+ O_m\left(K_m^{\ome(d')}B^{\kap_m}d'^{-m\kap_m}\right),$$
with a constant $K_m$ given by
$$K_m:= 1+\sum_{\mu=m+1}^{\infty}|a_m(\mu)| 2^{-\kap_m(\mu-m)}.$$
Note that $K_m$ is indeed a convergent sum. Then by Lemma \ref{lem7} and equation (\ref{eqn27}) we deduce that
\begin{equation*}
\begin{split}
F_m(B,d)=& F_m\left(Bq^{-m},d'\right) +\sum_{\mu=m+1}^\infty a_m(\mu) F_m\left(Bq^{-\mu},d'\right)\\
=&C_mB^{\frac{1}{m}}\prod_{p\mid d'} \left(p^{-1}+G_m\left(p^{-\frac{1}{m}}\right)\right)\left(q^{-1}+\sum_{\mu=m+1}^\infty a_m(\mu) q^{-\frac{\mu}{m}}\right)\\
&+ O_m\left(K_m^{\ome(d')}B^{\kap_m}q^{-m\kap_m}d'^{-m\kap_m}\left(1+\sum_{\mu=m+1}^{\infty}|a_m(\mu)| q^{-\kap_m(\mu-m)}\right)\right).
\end{split}
\end{equation*}
By definition of $K_m$ we obtain 
\begin{equation*}
\begin{split}
F_m(B,d)=&C_mB^{\frac{1}{m}}\prod_{p\mid d} \left(p^{-1}+G_m\left(p^{-\frac{1}{m}}\right)\right)\\ &+ O_m\left(K_m^{\ome(d')}B^{\kap_m}q^{-m\kap_m}d'^{-m\kap_m}\left(1+\sum_{\mu=m+1}^{\infty}|a_m(\mu)| 2^{-\kap_m(\mu-m)}\right)\right)\\
=&C_mB^{\frac{1}{m}}\prod_{p\mid d} \left(p^{-1}+G_m\left(p^{-\frac{1}{m}}\right)\right)+ O_m\left(K_m^{\ome(d')+1}B^{\kap_m}(qd')^{-m\kap_m}\right)\\
=& C_mB^{\frac{1}{m}}\prod_{p\mid d} \left(p^{-1}+G_m\left(p^{-\frac{1}{m}}\right)\right)+ O_m\left(K_m^{\ome(d)}B^{\kap_m}d^{-m\kap_m}\right).
\qedhere
\end{split}
\end{equation*}
\end{proof}

Next we given an upper bounds for the leading constant in Lemma \ref{lem8}. For squarefree $d$ we introduce the notation
\begin{equation}
\label{eq:c_md_explicit}
c_{m,d}=C_m\prod_{p\mid d} \left(p^{-1}+G_m\left(p^{-\frac{1}{m}}\right)\right).
\end{equation}
We observe that $c_{1,d}=1/d$ and we recall that 
\begin{equation}\label{eqn10}
	p^{-1}+G_{m}\left(p^{-\frac{1}{m}}\right)=\frac1{1+p-p^\frac{m-1}m}.
\end{equation}

	For every fixed $m\geq 1$ there is a positive constant $c_2(m)<1$ such that
	$$p^{\frac{m-1}{m}}\leq c_2(m) p,$$
	holds for all primes $p\geq 2$. Hence we deduce from equation (\ref{eqn10}) 
	that there exists a constant $c_3(m)$, only depending on $m$, such that
	\begin{equation}\label{eqn11}
	\prod_{p\mid d}  \left(p^{-1}+G_{m}\left(p^{-\frac{1}{m}}\right)\right) 
	\leq c_3(m)^{\ome(d)}\frac{1}{d}.
	\end{equation}
	Hence, we get
	\begin{equation}
	\label{eq:c_md_estimate}
	c_{m,d}
	\ll C_m c_3(m)^{\ome(d)}\frac{1}{d}
	\ll_{m,\varepsilon}d^{-1+\varepsilon}.
	\end{equation}
	
	Moreover, for $m\geq 2$ we have 
\begin{equation}
\label{eq:inequality_mkappa_m}
m k_m=\frac m{m+1}=1-\frac 1{m+1}\geq \frac 23. 
\end{equation}
	
\subsection{An auxiliary counting function}
\label{sec:counting_function_f}
For any $s$-tuple of positive integers $\mathbf m=(m_1,\dots,m_s)$ and any $s$-tuple of squarefree positive integers $\mathbf d=(d_1,\dots,d_s)$,
let $f_{\mathbf m, \mathbf d} : \ZZ_{>0}^s \to \RR$ be the function defined by 
$$ 
f_{\mathbf m, \mathbf d}(y_1,\dots,y_s)=\begin{cases}
1 & \text{ if $d_i\mid y_i$ and $y_i$ is $m_i$-full $ \forall i\in\{1,\dots,s\}$,}\\
0 & \text{ otherwise.}
\end{cases}
$$
\begin{lemma}
\label{lem:counting_function_f}
Let $\mathbf m=(m_1,\dots,m_s)\in\ZZ_{>0}^s$. For every $s$-tuple of squarefree positive integers $\mathbf d=(d_1,\dots,d_s)$ and every $\varepsilon>0$ we have
\begin{multline*}
\sum_{1\leq y_i\leq B_i, 1\leq i\leq s} f_{\mathbf m, \mathbf d}(y_1,\dots,y_s) \\
=\left(\prod_{i=1}^s c_{m_i,d_i}\right)\prod_{i=1}^s B_i^{\frac1{m_i} }
+O_{\mathbf m, \varepsilon}\left( \left(\prod_{i=1}^sd_i\right)^{-\frac23+\varepsilon}\left( \prod_{i=1}^s B_i^{\frac1{m_i}}\right)\left( \min_{1\leq i\leq s} B_i\right)^{-\delta}\right)
\end{multline*}
for all $B_1,\dots,B_s>0$, 
where $\delta=\min \{ 1/3 , \min_{1\leq i\leq s}1/(m_i(m_i+1)) \}$.
\end{lemma}
\begin{proof}
We observe that 
$$
\sum_{1\leq y_i\leq B_i, 1\leq i\leq s} f_{\mathbf m, \mathbf d}(y_1,\dots,y_s) = \prod_{i=1}^s F_{m_i}(B_i,d_i).
$$ 
For $i\in\{1,\dots,s\}$ such that $m_i\geq 2$ apply Lemma \ref{lem8}, for $i\in\{1,\dots, s\}$ such that $m_i=1$ use the estimate $F_1(B_i,d_i)= B_i/d_i + O((B_i/d_i)^{2/3})$.
Then apply \eqref{eq:c_md_estimate} and \eqref{eq:inequality_mkappa_m} to estimate the error term.
\end{proof}

Lemma \ref{lem:counting_function_f} implies that the function $f_{\mathbf m, \mathbf d}$ satisfies Property I with the constants
$$ C_{f,M} = \prod_{i=1}^s c_{m_i,d_i}, \quad C_{f,E} =  \left(\prod_{i=1}^sd_i\right)^{-\frac23+\varepsilon},$$
and
$$\varpi_i= \frac{1}{m_i}, \quad 1\leq i\leq s$$
and
$$\vardel = \min \{ 1/3 , \min_{1\leq i\leq s}1/(m_i(m_i+1)) \}.$$

\section{Campana points on toric varieties}
\label{sec:Campana_points_on_toric_varieties}
\subsection{Toric varieties over number fields}

	Let $X$ be a complete smooth split toric variety over a number field $\KK$. 
	Let $T\subseteq X$ be the dense torus.
	Let $\Sigma$ be the fan that defines $X$. 
	We denote by $\{\rho_1,\dots,\rho_s\}$ the set of rays of $\Sigma$ 
	and by $\Sigma_{max}$ the set of maximal cones of $\Sigma$. 
	For every maximal cone $\sigma$ we define $\ii{\sigma}$ to be the set of indices 
	$i\in\{1,\dots,s\}$ such that the ray $\rho_i$ belongs to the cone $\sigma$, 
	and we set $\ic{\sigma}=\{1,\dots,s\}\smallsetminus \ii{\sigma}$.
	Then we have $|\ii{\sigma}|=n$ and $|\ic{\sigma}|=r$ for every maximal cone $\sigma$ of $\Sigma$, 
	where $n$ is the dimension of $X$ and $r$ is the rank of the Picard group of $X$. 
	In particular, $s=n+r$.
	For each $i\in\{1,\dots,s\}$, we denote by $D_i$ the prime toric invariant divisor 
	corresponding to the ray  $\rho_i$. We fix a canonical divisor $K_X:=-\sum_{i=1}^s D_i$.

	By \cite{MR1299003} the Cox ring of $X$ is $\KK[y_1,\dots,y_s]$ 
	where the degree of the variable $y_i$ is 
	the class of the divisor $D_i$ in $\pic(X)$.
	For every $\mathbf y=(y_1,\dots,y_s)\in \CC^{s}$ and every $D=\sum_{i=1}^s a_i D_i$, 
	let \[\mathbf y^D:=\prod_{i=1}^sy_i^{a_i}.\]
	Let $Y\to X$ be the universal torsor of $X$
	as in \cite[\S8]{MR1679841}. We recall that the variety $Y$ is an open subset of 
	$\A^{s}_{\KK}$ whose complement is defined by $\mathbf y^{D_\sigma}=0$ for all 
	maximal cones $\sigma$, where 
	$
	D_\sigma:=\sum_{i\in\ic{\sigma}}D_i
	$ 
	for all $\sigma\in\Sigma_{max}$.
	
	The integral model $\pi:\mathscr{Y}\to\mathscr{X}$ of the universal torsor $Y\to X$  
	as in \cite[Remarks 8.6]{MR1679841} gives a parameterization of the rational points 
	on $X$ via integral points in $\OO_{\KK}^s=\A^s(\OO_{\KK})$ as follows. 
	Let $\mathcal{C}$ be a set of ideals of $\OO_{\KK}$ that form a system of representatives 
	for the class group of $\KK$. 
	We fix a basis of $\pic(X)$, and
	for every divisor $D$ on $X$ we write $\mathfrak c^{D}:=\prod_{i=1}^r \mathfrak c_i^{b_i}$ where 
	$[D]=(b_1,\dots,b_r)$  with respect to the fixed basis of $\pic(X)$.
	Then, as in \cite[\S2]{MR3514738},
 	\[
 	X(\KK)=\mathscr X(\OO_{\KK})=
 	\bigsqcup_{\mathfrak c\in\mathcal C^{r}}\pi^{\mathfrak c}
 	\left(\mathscr{Y}^{\mathfrak c}(\OO_{\KK})\right),
 	\]
 	where  $\pi^{\mathfrak c}:\mathscr{Y}^{\mathfrak c}\to\mathscr X$ is the twist 
 	of $\pi$ defined in \cite[Theorem 2.7]{MR3552013}.
 	The fibers of $\pi^{\mathfrak c}|_{\mathscr{Y}^{\mathfrak c}(\OO_{\KK})}$ are all isomorphic to 
 	$(\OO_{\KK}^\times)^{r}$, and
	$\mathscr{Y}^{\mathfrak c}(\OO_{\KK})\subseteq \A^s(\OO_{\KK})$ is the subset of 
	points $\mathbf y\in \bigoplus_{i=1}^s\mathfrak c^{D_i}$ that satisfy
	\begin{equation}
	\label{eq:coprimality_condition}
	\sum_{\sigma\in\Sigma_{\max}}\mathbf y^{D_\sigma}\mathfrak c^{-D_\sigma}=\OO_{\KK}.
	\end{equation}
	
	Let $N$ be the lattice of cocharacters of $X$. Then $\Sigma\subseteq N\otimes_\ZZ\RR$.
	For every $i\in\{1,\dots,s\}$, let $\gen{i}$ be the unique generator of $\rho_i\cap N$.
	For every torus invariant divisor $D=\sum_{i=1}^s a_i D_i$ of $X$ and for every 
	$\sigma\in\Sigma_{max}$, let $u_{\sigma,D}$ be the character of $N$ determined by 
	$u_{\sigma,D}(\gen{j})=a_j$ for all $j\in\ii{\sigma}$, and define 
	$D(\sigma):=D-\sum_{i=1}^su_{\sigma,D}(\gen{i})D_i$. 
	Then $D$ and $D(\sigma)$ are linearly equivalent.
	For every $i,j\in\{1,\dots,s\}$, let $\beta_{\sigma,i,j}:=-u_{\sigma,D_{j}}(\gen{i})$.
	Then, for every $i,j \in\{1,\dots,s\}$, we have $\beta_{\sigma,i,j}=0$ whenever 
	$j\in\ic{\sigma}$, and whenever $i \neq j$ are both in $\ii{\sigma}$. Hence,
	\begin{equation}
	\label{eq:D_rho_sigma}
	D_{j}(\sigma)=\begin{cases}
	D_{j} & \text{if } j\in\ic{\sigma},\\
	\sum_{i\in\ic{\sigma}}\beta_{\sigma,i,j}D_{i} & \text{if } j\in\ii{\sigma}.
	\end{cases}
	\end{equation}

	\begin{lemma}
	\label{lem:relation_beta_sigma_sigma'}
	For every $i,j\in\{i,\dots,s\}$ and $\sigma,\sigma'\in\Sigma_{max}$ we have
	\[
	\beta_{\sigma, i,j}
	=-\sum_{l\in \ii{\sigma'}}\beta_{\sigma', i, l}\beta_{\sigma, l, j}. 
	\]
	\end{lemma}
	\begin{proof}
	From the equality $D_{j}(\sigma')=(D_{j}(\sigma))(\sigma')$
	we get
	\begin{align*}
	0&= \sum_{l=1}^s u_{\sigma,D_{j}}(\gen{l})D_{l}(\sigma') 
	= \sum_{l=1}^s u_{\sigma,D_{j}}\left(\gen{l})(D_{l}-\sum_{i=1}^s u_{\sigma',D_{l}}(\gen{i})D_i\right)\\
	&=\sum_{i=1}^s\left( u_{\sigma,D_{j}}(\gen{i}) 
	- \sum_{l=1}^s u_{\sigma,D_{j}(\gen{l})}	u_{\sigma',D_{l}}(\gen{i})D_i\right)D_i. 
	\qedhere
	\end{align*}
	\end{proof}

\subsection{Polytopes}
\label{sec:polytopes_toric}

	In this section, we fix a semiample $\QQ$-divisor $L=\sum_{i=1}^s a_i D_i$, 
	and we study a number of polytopes associated to $L$. 
	The content of this section is purely combinatorial, 
	in particular, it does not depend on the base field $\KK$ where $X$ is defined.
	 
	For each $\sigma\in\Sigma_{\max}$, we write $L(\sigma)=\sum_{i=1}^s\alpha_{i, \sigma}D_i$. 
	Then $\alpha_{i,\sigma}=0$ for all $i\in\ii{\sigma}$ by construction.
	\begin{remark}	
	\label{rem:basepointfree_L(sigma)}
	 Since $L$ is semiample, $L(\sigma)$ is effective for all $\sigma\in\Sigma_{\max}$ 
	 by \cite[Proposition 6.1.1]{MR2810322}; that is, $\alpha_{i,\sigma}\geq0$ for all 
	 $i\in\{1,\dots,s\}$ and all $\sigma\in\Sigma_{\max}$.
	 If, moreover, $L$ is ample, then $\alpha_{i,\sigma}>0$ for all $i\in\ic{\sigma}$ and 
	 all $\sigma\in\Sigma_{\max}$ by \cite[Theorem 6.1.14]{MR2810322}. 
	 \end{remark}
	 
	\begin{assumption}
	\label{assp:L}
	We assume that for every $i\in\{1,\dots,s\}$ there exists $\sigma\in\Sigma_{\max}$ 
	such that $\alpha_{i,\sigma}>0$.
	\end{assumption}
	
	We observe that Assumption \ref{assp:L} is satisfied 
	if $L$ is ample by Remark \ref{rem:basepointfree_L(sigma)}, 
	or if $L$ is  linearly equivalent to an effective divisor 
	$\sum_{i=1}^s b_i D_i$ with $b_i>0$ for all $i\in\{1,\dots,s\}$ 
	by Lemma \ref{lem:effective_cone} below.
	
\subsubsection{The polytope $P_L$} 
	We describe a classical polytope associated to $L$ that we use in 
	Section \ref{sec:heights} to study the height function defined by $L$.

	We denote by $M$ the lattice of characters of $T$, dual to $N$, 
	and by $M_\RR$ the vector space $M\otimes_\ZZ \RR$. 
	Similarly, we set $\pic(X)_\RR:=\pic(X)\otimes_\ZZ\RR$. 
	We recall that there is an exact sequence (e.g. \cite[Theorem 4.2.1]{MR2810322})
	\[
	0
	\to M_\RR
	\to \bigoplus_{i=1}^s \RR D_i 
	\stackrel{\varphi}{\to}\pic(X)_\RR
	\to 0,
	\]
	such that the effective cone $\eff(X)$ of $X$ is the image under $\varphi$ of the cone generated by the 
	effective torus invariant divisors
	\[
	C:=\left \{\sum_{i=1}^s a_i D_i : 
	a_1,\dots, a_s\geq 0 \right \} \subseteq \bigoplus_{i=1}^s \RR D_i.
	\] 
	
	Since $L$ is semiample, 
	\[
	P_L:=\{m\in M_\RR: m(\gen{i})+a_i\geq 0\ \forall i\in\{1,\dots, s\}\}.
	\]	
	 is a polytope with vertices 
	$\{-u_{\sigma,L}:\sigma\in\Sigma_{\max}\}$ 
	by \cite[Proposition 4.3.8, Theorem 6.1.7]{MR2810322}. In particular,
	\begin{equation}
	\label{eq:P_L}
	P_L=\left\{\sum_{\sigma\in\Sigma_{\max}}-\lambda_\sigma u_{\sigma,L}:
	(\lambda_\sigma)_{\sigma\in\Sigma_{\max}}\in\RR_{\geq0}^{\Sigma_{\max}}, 
	\sum_{\sigma\in\Sigma_{\max}}\lambda_\sigma=1\right\}.
	\end{equation}

	\begin{lemma}
	\label{lem:effective_cone}
	For every $t\in\RR_{\geq0}$,
	\[
	(tL+M_\RR)\cap C=
	\left \{\sum_{\sigma\in\Sigma_{\max}}\lambda_\sigma L(\sigma):
	(\lambda_{\sigma})_{\sigma\in\Sigma_{\max}}\in (\RR_{\geq0})^{\Sigma_{\max}},
 	\sum_{\sigma\in\Sigma_{\max}}\lambda_{\sigma}=t \right \}.
 	\]
	\end{lemma}
	\begin{proof}
	By \eqref{eq:P_L} $L+P_L$ is the polytope with vertices $\{L(\sigma):\sigma\in\Sigma_{\max}\}$.
	Since $(tL+M_\RR)\cap C=tL+P_{tL}=t(L+P_L)$ 
	by \cite[Exercise 4.3.2]{MR2810322}, the statement follows.
	\end{proof}

	\begin{remark}
	\label{rem:L_weights}
	If $L$ is ample, then there exists a positive integer $t$ 
	such that $[t^{-1}L]=\left [\sum_{i=1}^s\frac 1{m_i} D_i \right]$ with $m_1,\dots,m_s\in\ZZ_{>0}$. 
	Indeed, $[L]$ has at least one representative of the form $\sum_{i=1}^sb_i D_i$ 
	with $b_1,\dots,b_s\in \QQ_{>0}$ by Lemma \ref{lem:effective_cone} 
	and Remark \ref{rem:basepointfree_L(sigma)}. 
	Hence, it suffices to choose any positive integer $t$ such that 
	$t b_i^{-1}\in\ZZ$ for all $i\in\{1,\dots,s\}$.
	\end{remark}

\subsubsection{The polytope $\widetilde P$}
	We investigate some polytopes associated to $[L]$ that we use for 
	the application of the hyperbola method
	in  Sections \ref{sec:heuristics}--\ref{sec:proof_thmtoric}.
	
	We identify $\RR^{s}$ with the space of linear functions on 
	$\bigoplus_{i=1}^s\RR D_i$ by defining
	$\mathbf t (D_i)=t_i$ for all $i\in\{1,\dots,s\}$ and all $\mathbf t=(t_1,\dots,t_s)\in\RR^s$.
	Under this identification, the dual of $\pic(X)$ is 
	the linear subspace $\widetilde H$ of $\RR^{ s }$ defined by 
 	\begin{equation}
 	\label{eq:dual_picard}
 	t_{j}
 	=\sum_{i\in\ic{\sigma}} \beta_{\sigma, i,j}t_i 
 	\quad \forall j\in \ii{\sigma},
 	\end{equation}
	for one, or equivalently all, $\sigma\in\Sigma_{\max}$ 
	(cf.~Lemma \ref{lem:relation_beta_sigma_sigma'}).	
	
	Let $\widetilde P \subseteq \RR^{s}$ be the polyhedron 
 	defined by 
 	\begin{equation}
 	\label{eq:widetildeP}
 	t_i\geq0 \ \forall i\in\{1,\dots,s\} \quad \text{and} \quad
 	\sum_{i=1}^s\alpha_{i,\sigma} t_i\leq 1 \ 
 	\forall \sigma\in \Sigma_{\max}.
 	\end{equation} 	
 	Then $\widetilde P$ is a full dimensional convex polytope by 
 	Remark \ref{rem:basepointfree_L(sigma)} and Assumption \ref{assp:L}.
 	Moreover, $\cone(\widetilde P)$ is dual to the cone 
 	$C$ defined above.
	For every $\sigma\in\Sigma_{\max}$, let \[
	\widetilde P_\sigma:=\widetilde P \cap 
	\cone \left (
	\widetilde P\cap \left \{
	\sum_{i=1}^s\alpha_{i,\sigma} t_i =1 
	\right \} \right ),
	\]
	so that 
	\begin{equation}	
	\label{eq:widetildeP_union}
	\bigcup_{\sigma \in \Sigma_{\max}}\widetilde P_\sigma=\widetilde P.
	\end{equation}
	We observe that the polytopes $\widetilde P$ and $\widetilde P_\sigma$ depend only 
	on the class of $L$ in $\pic(X)$ and not on the chosen representative  $\sum_{i=1}^s a_i D_i$.

	\begin{lemma}
	\label{lem:widetildePsigma}
	\begin{enumerate}[label = (\roman*), ref=(\roman*)]
	\item 
	\label{lem:item:widetildePsigma_intersection_two}
	For every $\sigma, \sigma'\in\Sigma_{\max}$, 
	\[
	\widetilde P_\sigma \cap \widetilde P_{\sigma'} = 
	\widetilde P_\sigma \cap
	\left \{ 
	t_{j}
 	= \sum_{i\in\ic{\sigma}} \beta_{\sigma, i, j}t_i 
 	\quad \forall j\in  \{l\in \ii{\sigma} : \alpha_{l, \sigma'}\neq0\}.
 	\right \}.
	\]
	\item 
	\label{lem:item:widetildePsigma_intersection_all}
	Under Assumption \ref{assp:L} we have 
	$\bigcap_{\sigma\in\Sigma_{\max}} \widetilde P_{\sigma} =\widetilde P \cap \widetilde H$.
	\item 
	\label{lem:item:widetildePsigma_defin}
	If $L$ is ample, then
	\[
	\widetilde P_\sigma 
	= \widetilde P \cap
	\left \{ 
	t_{j}
 	\leq \sum_{i\in\ic{\sigma}} \beta_{\sigma, i, j}t_i
 	\quad \forall j\in\ii{\sigma}
 	\right \}
 	\]
 	for every $\sigma\in\Sigma_{\max}$.
 	In particular, $\widetilde P_\sigma$ 
 	is the polytope in $\RR^{s}$ defined by 
 	\[
 	t_1,\dots,t_s \geq0, \quad 
 	t_{j}
 	\leq \sum_{i\in\ic{\sigma}} \beta_{\sigma, i,j}t_i
 	\quad \forall j\in\ii{\sigma}, \quad 
 	\sum_{i\in\ic{\sigma}} \alpha_{i,\sigma}t_i \leq 1.
 	\]
	\end{enumerate}
	\end{lemma}
	\begin{proof}
	By definition, $\widetilde P_{\sigma}$ is the set of elements $\mathbf t\in \widetilde P$ 
	such that $\mathbf t(L(\sigma'))\leq \mathbf t(L(\sigma))$
	for all $\sigma'\in\Sigma_{\max}$. By \eqref{eq:D_rho_sigma} we have
	\begin{equation*}
	\mathbf t ( L(\sigma) - L(\sigma') ) 
	 = \mathbf t \left(\sum_{j=1}^s \alpha_{j,\sigma'} (D_j(\sigma) - D_j) \right) 
	 = \sum_{j\in \ii{\sigma}} \alpha_{j,\sigma'}
	\left ( \left( 
	\sum_{i\in\ic{\sigma}}\beta_{\sigma,i,j} t_i 
	\right) - {t_j} \right)
	\end{equation*}
	Hence, \ref{lem:item:widetildePsigma_intersection_two} and 
	the inclusion $\supseteq$ in \ref{lem:item:widetildePsigma_defin} follow. 
	For the reverse inclusion in \ref{lem:item:widetildePsigma_defin} we fix $j\in\ii{\sigma}$. 
	By \cite[Lemma 8.9]{MR1679841}
	there is 
	$\sigma'\in\Sigma_{\max}$ such that 
	$\ii{\sigma}\cap\ii{\sigma'}=\ii{\sigma}\smallsetminus\{j\}$. 
	Then 
	\[
	 \alpha_{j,\sigma'}
	 \left(
	 {t_j}
	-\sum_{i\in\ic{\sigma}}\beta_{\sigma,i,j} t_i
	\right) = \mathbf t ( L(\sigma') - L(\sigma) ) \leq 0
	\]
	for all $\mathbf t\in\widetilde P_\sigma$ and $\alpha_{j,\sigma'}>0$ 
	by Remark \ref{rem:basepointfree_L(sigma)}. 
	Part \ref{lem:item:widetildePsigma_intersection_all} follows 
	from \ref{lem:item:widetildePsigma_intersection_two}, 
	as $\bigcap_{\sigma\in\Sigma_{\max}} \widetilde P_{\sigma} =\widetilde P_{\sigma'} \cap \widetilde H$ 
	for every $\sigma'\in\Sigma_{\max}$ by \ref{lem:item:widetildePsigma_intersection_two} 
	together with Assumption \ref{assp:L}, and we conclude by \eqref{eq:widetildeP_union}.
	\end{proof}
	
	\begin{lemma}
	\label{lem:face_maximization_problem}
	Assume that $L$ is ample. Let $\varpi=(\varpi_1,\dots,\varpi_s)\in\RR^s_{>0}$. 
	Let $\widetilde F$ be the face of $\widetilde P$ where the maximum value 
	$a(L,\varpi)$ of $\sum_{i=1}^s\varpi_i t_i$ is attained. 
	Then 
	\begin{enumerate}[label=(\roman*), ref=(\roman*)]
	\item 
	\label{item:lem:face_maximization_problem_general}
	$a(L,\varpi)>0$ and $\widetilde F\subseteq \widetilde H$.
	\item 
	\label{item:lem:face_maximization_problem_special}
	If, additionally, $[L]=[\sum_{i=1}^s\varpi_i D_i]$ in $\pic(X)_{\RR}$, then $a(L,\varpi)=1$ 
	and $\widetilde F=\widetilde H \cap \{\sum_{i=1}^s\varpi_i t_i=1\}$. 
	In particular, $\widetilde F\cap \{t_1,\dots,t_s>0\}\neq\emptyset$.
	\end{enumerate}		
	\end{lemma}	
	\begin{proof}	
	Since $s\geq1$ and $\widetilde P$ is full dimensional, we have $a(L,\varpi)>0$.
	For every $\sigma\in\Sigma_{\max}$, let 
 	$\widetilde F_\sigma:= \widetilde F\cap \widetilde P_\sigma$.  	
	Fix $\sigma\in\Sigma_{\max}$. 
	Let $\mathbf t\in \widetilde P_\sigma \smallsetminus \widetilde H$.
	By \eqref{eq:dual_picard} and Lemma \ref{lem:widetildePsigma}  
	there exists $j\in\ii{\sigma}$ such that 
	$t_{j}<\sum_{i\in\ic{\sigma}}\beta_{\sigma,i,j} t_i$.
	Let $t'_{j}:=\sum_{i\in\ic{\sigma}}\beta_{\sigma,i,j} t_i$. 
	For each $i\in\{1,\dots,s\}\smallsetminus\{j\}$ let $t'_i:=t_i$. 
	Then $(t'_{1},\dots,t'_{s}) \in \widetilde P_\sigma$, and 
	$\sum_{i=1}^s\varpi_i t'_i 
	> \sum_{i=1}^s\varpi_i t_i $ by construction. 
	Hence $\mathbf t\notin \widetilde F_\sigma$. 
	Thus $\widetilde F_\sigma \subseteq \widetilde H$, which implies 
	$\widetilde F_{\sigma} \subseteq \widetilde F_{\sigma'}$ for all $\sigma'\in\Sigma_{\max}$. 
	Since this proof works for every $\sigma \in \Sigma_{\max}$, we conclude that
	$\widetilde F_{\sigma} = \widetilde F_{\sigma'}$ for all $\sigma, \sigma'\in\Sigma_{\max}$. 
	Now \ref{item:lem:face_maximization_problem_general} follows, because
	$\widetilde F=\bigcup_{\sigma\in\Sigma_{\max}}\widetilde F_\sigma$.	
	
	For \ref{item:lem:face_maximization_problem_special} we recall that 
	$\alpha_{i,\sigma}=\varpi_i+\sum_{j\in\ii{\sigma}}\varpi_j\beta_{\sigma,i,j}$ 
	for all $i\in\ic{\sigma}$.  Hence, we have
	$\sum_{i\in\ic{\sigma}}\alpha_{i,\sigma}t_i=\sum_{i=1}^s\varpi_i t_i$ 
	for all $\mathbf t\in \widetilde H$ and for all $\sigma\in\Sigma_{\max}$.
	Since $\widetilde H$ is the subspace of $\RR^s$ dual to $\pic(X)_{\RR}$, 
	a torus invariant divisor $D$ satisfies $\mathbf t(D)=0$ for all $\mathbf t\in\widetilde H$
	if and only if $D$ is a principal divisor. Since $D_1,\dots, D_s$ are not principal divisors, then 
	$\widetilde H \cap \{t_1,\dots,t_s>0\}\neq\emptyset$. 
	Let $\mathbf t \in \widetilde H$ with $t_1,\dots,t_s>0$, 
	up to rescaling $\mathbf t$ by a positive real number we  can assume that
	$\sum_{i=1}^s\varpi_i t_i=1$, and hence $\mathbf t\in\widetilde F$.
	\end{proof}

\subsubsection{The geometric constant}
        We compute certain volumes of polytopes that appear in the leading constant 
	of the asymptotic formula \ref{eq:mainthm_asympt}.
	
	Fix $\sigma\in\Sigma_{\max}$.
	Since $X$ is smooth, we know that $\pic(X)=\bigoplus_{i\in\ic{\sigma}}\ZZ[D_i]$. 
	We identify  $\RR^r$ with the space of linear functions on 
	$\bigoplus_{i\in\ic{\sigma}}\RR[D_{i}]$ by defining 
	$$\mathbf z\left (\sum_{i\in\ic{\sigma}}a_i[D_{i}]\right ):=\sum_{i\in\ic{\sigma}}a_iz_i$$ 
	for all $\mathbf z=(z_i)_{i\in\ic{\sigma}}\in\RR^r$.
	Let $\lambda_{[L]}:\RR^r\to\RR$ be the evaluation at  $[L]$;
	that is, $\lambda_{[L]}(\mathbf z)=\sum_{i\in\ic{\sigma}}\alpha_{i,\sigma} z_i$. 
	Fix $\tilde i\in\ic{\sigma}$ such that $\alpha_{\tilde i,\sigma}\neq0$.
	The change of variables $x=\lambda_{[L]}(\mathbf z)$, 
	$\mathrm d x=\alpha_{\tilde i,\sigma}\mathrm d z_{\tilde i}$, gives
	\[
	\int_{\RR^r} g\prod_{i\in\ic{\sigma}}\mathrm d z_i=
	\int_\RR\left ( 
	\int_{z_{\tilde i}=\left (x-\sum_{i\in\ic{\sigma}, i\neq\tilde i}\alpha_{i,\sigma} z_i \right )/\alpha_{\tilde i,\sigma}}
	g\  \alpha_{\tilde i,\sigma}^{-1}\prod_{i\in\ic{\sigma}, i\neq\tilde i}\mathrm dz_i\right)\mathrm d x
	\]
	for all integrable functions $g:\RR^r\to \RR$.
	\begin{lemma}
	\label{lem:alpha_peyre}
	The volume
	\[
	\alpha(L):= 
	\int_{\eff(X)^*\cap \lambda_{[L]}^{-1}(1)} 
	\alpha_{\tilde i,\sigma}^{-1}\prod_{i\in	\ic{\sigma}, i\neq\tilde i}\mathrm dz_i
	\]
	is positive and independent of the choice of $\tilde i$ and of the choice of $\sigma$.
	\end{lemma}
	\begin{proof}
	The transversal intersection $\eff(X)^*\cap \lambda_{[L]}^{-1}(1)$  
	is an $(r-1)$-dimensional polytope, hence the volume is positive.
	The independence of the choice of $\tilde i$ is clear. 
	The independence of the choice of $\sigma$ is a consequence of 
	Lemma \ref{lem:relation_beta_sigma_sigma'}. 
	\end{proof}
	 	
	\begin{lemma}
	\label{lem:alpha_constant}
	Assume that $L$ is ample and $[L]=[\sum_{i=1}^s\varpi_i D_i]$ with $\varpi_1,\dots,\varpi_s>0$. 
	For $\delta\geq0$, let $H_\delta$ be the hyperplane defined by $\sum_{i=1}^s\varpi_i t_i=1-\delta$.
	Then for $\delta>0$ small enough, 
	$\meas_{s-1}\left(H_\delta\cap\widetilde P\right)= c \delta^{s-r} +O(\delta^{s-r+1})$, where
	$\meas_{s-1}$ is the $(s-1)$-dimensional measure on $H_\delta$ given by 
	$\prod_{1\leq i\leq s, i\neq\tilde i}(\varpi_i\mathrm dt_i)$ 
	for any choice of $\tilde i\in\{1,\dots,s\}$, and 
	$$
	c=\frac{\alpha(L)}{(s-r)!}\sum_{\sigma\in\Sigma_{\max}}\prod_{i\in\ic{\sigma}}\varpi_i.
	$$
	\end{lemma}	
	\begin{proof}
	Since $L$ is ample, the decomposition \eqref{eq:widetildeP_union} 
	and  Lemma \ref{lem:widetildePsigma} give
	$$
	\meas_{s-1}\left(H_\delta\cap\widetilde P\right)
	= \sum_{\sigma\in\Sigma_{\max}}\meas_{s-1}\left(H_\delta\cap\widetilde P_\sigma\right).
	$$

	Fix $\sigma \in\Sigma_{\max}$. 
	Let $V_{\delta, \sigma}:=\meas_{s-1}(H_\delta\cap\widetilde P_\sigma)$. 
	By the choice of $L$ we have 
	$\alpha_{i,\sigma}=\varpi_i+\sum_{j\in\ii{\sigma}}\varpi_j\beta_{\sigma,i,j}$ for all $i\in\ic{\sigma}$, 
	and hence,
 	$$
 	\sum_{i=1}^s\varpi_i t_i
 	=\sum_{i\in\ic{\sigma}}\alpha_{i,\sigma} t_i
 	-\sum_{j\in\ii{\sigma}}\varpi_j\left(\sum_{i\in\ic{\sigma}}\beta_{\sigma,i,j}t_i-t_j\right)
 	$$ 
 	for every $\mathbf t \in\RR^s$. 
 	Then   $H_0\cap \widetilde P_\sigma\subseteq \widetilde H$ 
 	by Lemma \ref{lem:widetildePsigma}\ref{lem:item:widetildePsigma_defin}. 
 	Fix $\xi=(\xi_1,\dots,\xi_s)\in H_0\cap \widetilde P_\sigma$, and fix $\tilde i\in\ic{\sigma}$. 
 	Then $V_{\delta,\sigma}=\meas_{s-1}\left(\left(H_\delta\cap \widetilde P_\sigma\right)+\delta\xi\right)$ 
 	is the volume of the polytope given by
	\[
 	t_i\geq \delta\xi_i \ \forall i\in\{1,\dots,s\}, \ 
 	t_{j}
 	\leq \sum_{i\in\ic{\sigma}} \beta_{\sigma, i,j}t_i 
 	\ \forall j\in\ii{\sigma}, \ 
 	\sum_{i\in\ic{\sigma}} \alpha_{i,\sigma}t_i \leq 1+\delta, \ \sum_{i=1}^s\varpi_i t_i=1,
 	\]
 	with respect to the measure $\prod_{1\leq i\leq s, i\neq\tilde i}(\varpi_i\mathrm dt_i)$.  
 	
 	For all $i\in\ic{\sigma}$, let $u_i=t_i$. 
 	For all $j\in\ii{\sigma}$, let $u_j=\left(\sum_{i\in\ic{\sigma}} \beta_{\sigma, i,j}t_i -t_j \right)/\delta$, i.e.,
 	$$
 	u_j
 	= \delta^{-1}\left(\sum_{i\in\ic{\sigma},i\neq\tilde i} \beta_{\sigma, i,j}t_i 
 	+\frac{\beta_{\sigma,\tilde i , j}}{\varpi_{\tilde i}}
 	\left(1- \sum_{1\leq i\leq s, i\neq \tilde i} \varpi_i t_i \right)  -t_j \right).
 	$$
	Let $g(\mathbf u):=\sum_{j\in\ii{\sigma}}\varpi_j u_j$ 
	and $h(\mathbf u):= 1-\sum_{i\in\ic{\sigma},i\neq \tilde i} \alpha_{i,\sigma}u_i$.
	Then $\delta^{r-s}V_{\delta, \sigma}$ is the volume of the polytope given by
 	\begin{gather*}
 	u_j\geq0 \ \forall j\in\ii{\sigma}, 
 	\quad 
 	g(\mathbf u) \leq 1,
	\\ 	
 	u_i\geq \delta\xi_i \ \forall i\in\ic{\sigma}\smallsetminus\{\tilde i\}, 
 	\quad 
 	\sum_{ i\in\ic{\sigma},i\neq\tilde i} \alpha_{i,\sigma} u_i \leq 1+\delta g(\mathbf u),
 	\\
 	\sum_{i\in\ic{\sigma}, i\neq\tilde i}\beta_{\sigma,i,j}u_i 
 	+\frac{\beta_{\sigma,\tilde i,j}}{\alpha_{\tilde i,\sigma}}
 	h(\mathbf u)
 	\geq \delta \left(u_j+\xi_j-\frac{\beta_{\sigma,\tilde i,j}}{\alpha_{\tilde i,\sigma}} g(\mathbf u) \right) 
 	\ \forall j\in\ii{\sigma},
 	\end{gather*}
 	with respect to the measure 
 	$\varpi_{\tilde i}\alpha_{\tilde i,\sigma}^{-1}
 	\prod_{1\leq i\leq s, i\neq\tilde i}(\varpi_i\mathrm du_i)$, as
 	\begin{align*}
 	\left|\det\left((\partial u_j/\partial t_l)_{1\leq j,l\leq s, j\neq\tilde i,l\neq\tilde i}\right)\right|
 	&=\delta^{r-s}\det\left(\left(\beta_{\sigma, \tilde i, l}\varpi_l/\varpi_{\tilde i} + \delta_{l,j}\right)_{j,l\in\ii{\sigma}}\right)\\
 	&=\delta^{r-s} \left(1+ \sum_{j\in\ii{\sigma}}\beta_{\sigma, \tilde i, j}\varpi_j/\varpi_{\tilde i} \right ) 
 	= \delta^{r-s}\alpha_{\tilde i,\sigma}/\varpi_{\tilde i},
 	\end{align*}
 	where $\delta_{l,j}=0$ if $l\neq j$ and $\delta_{j,j}=1$.
 	By dominated convergence 
 	we can compute $c=\lim_{\delta\to 0}\delta^{r-s}V_{\delta, \sigma}$ 
 	as the volume of the polytope given by
 	\begin{gather*}
 	u_j\geq0 \ \forall j\in\ii{\sigma}, 
 	\quad 
 	g(\mathbf u)\leq 1, \\
 	u_i\geq0 \ \forall i\in\ic{\sigma}, 
 	\quad \sum_{i\in\ic{\sigma}}\beta_{\sigma,i,j}u_i\geq 0 \ \forall j\in\ii{\sigma}, 
 	\quad \sum_{i\in\ic{\sigma}}\alpha_{i,\sigma}u_i=1,
 	\end{gather*}
 	with respect to the measure 	
 	$\varpi_{\tilde i}\alpha_{\tilde i,\sigma}^{-1}\prod_{1\leq i\leq s, i\neq\tilde i}(\varpi_i\mathrm du_i)$.
 	We conclude as
 	\[
 	\int_{u_j\geq0 \ \forall j\in\ii{\sigma}, g(\mathbf u)\leq 1} 
 	\prod_{j\in\ii{\sigma}}(\varpi_j\mathrm d u_j)
 	=\frac{1}{(s-r)!}.
 	\qedhere
 	\]
\end{proof}	 	
	 	
\subsection{Heights}
\label{sec:heights}
	Now we study the height associated to a semiample $\QQ$-divisor $L$ on $X$. 
	Let $t$ be a positive integer such that $tL$ has integer coefficients and is base point free.
	By \cite[Proposition 4.3.3]{MR2810322} we have 
	$H^0(X, tL)=\bigoplus_{m\in P_{tL}\cap M}\KK\chi^m$, where $P_{tL}$ and $M$ 
	are defined in Section \ref{sec:polytopes_toric} and $\chi^m\in \KK[T]$ is 
	the character of $T$ corresponding to $m$.

	Let $H_{tL}:X(\KK)\to\RR_{\geq0}$ be the pullback of the exponential Weil height under 
	the morphism $X\to\PP(H^0(X,tL))$ defined by the basis of $H^0(X,tL)$ corresponding 
	to $P_{tL}\cap M$. 
	We define $H_{L}:=(H_{tL})^{1/t}$. 
	We observe that this definition agrees with \cite[\S 2.1]{MR1369408}. 
	\begin{prop}
	\label{prop:height}
	For every $\mathbf y \in \mathscr Y({\KK})$, we have 
	\[
	H_L(\pi(\mathbf y))
	=\prod_{\nu\in\Omega_{\KK}} \sup_{\sigma\in\Sigma_{\max}} \left|\mathbf y^{L(\sigma)}\right|_{\nu}.
	\]
	For every $\mathfrak c\in\mathcal{C}^{r}$ and  
	$\mathbf y\in \mathscr Y^{\mathfrak c}(\mathcal{O}_{\KK})$, we have 
	\[
	H_L(\pi^{\mathfrak c}(\mathbf y))
	=\idealnorm \left(\mathfrak c^{-L}\right)
	\prod_{\nu\in\Omega_\infty} \sup_{\sigma\in\Sigma_{\max}} \left| \mathbf y^{L(\sigma)}\right|_{\nu}.
	\]
	\end{prop}
	\begin{proof}
	By definition of $H_L$ and of $\pi$ we have, for $\mathbf y \in \mathscr Y({\KK})$, 
	\[
	H_L(\pi( \mathbf y ))
	=\prod_{\nu\in\Omega_{\KK}} \sup_{m\in P_{tL}\cap M} \left| \mathbf y^{tL+(\chi^m)} \right|_{\nu}^{1/t}.
	\]
	Let $m\in P_{tL}\cap M$. 
	By \eqref{eq:P_L} there are $\lambda_\sigma\in\RR_{\geq0}$ for 
	$\sigma\in\Sigma_{\max}$ such that 
	$\sum_{\sigma\in\Sigma_{\max}}\lambda_\sigma=1$ 
	and  $m=-\sum_{\sigma\in \Sigma_{\max}}\lambda_\sigma u_{\sigma,tL}$, so that 
	$\mathbf y^{tL+(\chi^m)}= \mathbf y^{\sum_{\sigma\in \Sigma_{\max}}\lambda_\sigma tL(\sigma)}$. 
	This proves the first statement.
	For the second statement we argue as in the proof of \cite[Proposition 2]{MR3514738}.
	Fix $\mathfrak c\in\mathcal{C}^{r}$ and  
	$\mathbf y\in \mathscr Y^{\mathfrak c}(\mathcal{O}_{\KK})$. 
	For every prime ideal $\mathfrak p$ of $\OO_{\KK}$ we write $v_\p$ for the 
	associated valuation. Then 
	\[
	\min_{\sigma\in\Sigma_{\max}} v_\p \left( \mathbf y^{L(\sigma)}\right)
	=\min_{\sigma\in\Sigma_{\max}} v_\p \left( \mathbf y^{L(\sigma)}\mathfrak c^{-L(\sigma)}\right)
	+ v_\p \left(\mathfrak c^L\right)=v_\p\left(\mathfrak c^L\right),
	\]
	where the first equality holds as $[L(\sigma)]=[L]$ in $\pic(X)_{\RR}$, and the second equality 
	follows from \eqref{eq:coprimality_condition} as 
	$\mathbf y^{L(\sigma)} \in \mathfrak c^{L(\sigma)}$ for all $\sigma\in\Sigma_{\max}$.
	\end{proof}

	The following lemma will ensure the Northcott property for $H_L$.

	\begin{lemma}
	\label{lem:northcott}	
	If $L$ satisfies Assumption \ref{assp:L},
	then there is $\alpha>0$ such that for every 
	$\mathfrak c\in\mathcal C^r$ and  $B>0$,  every point 
	$ \mathbf y\in \bigoplus_{i=1}^s \mathfrak c^{D_i}$ with 
	\begin{equation}
	\label{eq:height_condition_affine_space}
	\prod_{\nu\in\Omega_\infty} \sup_{\sigma\in\Sigma_{\max}} \left| \mathbf y^{L(\sigma)}\right|_{\nu}
	\leq \idealnorm \left(\mathfrak c^{L}\right) B
	\end{equation}
	and $y_1,\dots,y_s\neq0$  satisfies 
	\[
	\prod_{\nu\in\Omega_\infty}| \mathbf y_i|_\nu\leq 
	\idealnorm \left(\mathfrak c^{D_{i}}\right)B^\alpha
	\]
 	for all $i\in\{1,\dots,s\}$.
	\end{lemma}
	\begin{proof}
	Let $ \mathbf y\in \bigoplus_{i=1}^s \mathfrak c^{D_i}$ 
	such that \eqref{eq:height_condition_affine_space} holds 
	and $y_1,\dots,y_s\neq0$.
	Fix $i\in\{1,\dots,s\}$ and choose $\sigma\in \Sigma_{\max}$ such that 
	$\alpha_{i,\sigma}>0$. 
	Recall that $\prod_{\nu\in\Omega_\infty}|y_i |_\nu=\idealnorm(y_i \OO_{\KK})$ 
	by the product formula. 
	Since $y_{j}\in\mathfrak c^{D_{j}}$ 
	for all $j\in\{1,\dots,s\}$, we have
	\[
	\idealnorm(y_i\OO_{\KK})^{\alpha_{i,\sigma}}
	\idealnorm \left(\mathfrak c^{L(\sigma)-\alpha_{i,\sigma}D_i}\right)
	\leq \idealnorm \left(\mathbf y^{L(\sigma)}\OO_{\KK}\right)
	\leq \prod_{\nu\in\Omega_\infty} \sup_{\sigma\in\Sigma_{\max}} \left| \mathbf y^{L(\sigma)} \right|_{\nu}
	\leq \idealnorm \left(\mathfrak c^{L}\right)B.
	\]
	Hence, $\idealnorm(\mathbf y_i\OO_{\KK})
	\leq \idealnorm\left (\mathfrak c^{D_\rho}\right)B^{1/\alpha_{i,\sigma}}$.
	\end{proof}

\subsection{Campana points}
\label{sec:campana}
	
	From now on we assume that $\KK=\QQ$. 
	For every $i\in\{1,\dots,s\}$, we fix a positive integer $m_i$ 
	and we denote by $\mathscr D_i$ the closure of $D_i$ in $\mathscr X$.
	Let $\mathbf m:=(m_1,\dots,m_s)$ and 
	\[
	\Delta:=\sum_{i=1}^s\left(1-\frac 1{m_i}\right) \mathscr D_i.
	\]
	The support of the restriction of $\Delta$ to the fibers over $\spec\ZZ$  is a strict 
	normal crossing divisor (see for example \cite[\S 5.1]{MR2740045}), 
	hence $(\mathscr{X},\Delta)$ is a Campana orbifold as in \cite[Definition 3.1]{PSTVA}.
	We denote by $(\mathscr{X},\Delta)(\ZZ)$ the set of Campana 
	$\mathbb Z$-points as in \cite[Definition 3.4]{PSTVA}, and 
	by $\mathscr{Y}(\ZZ)_\mathbf{m}$ the preimage of 
	$(\mathscr{X},\Delta)(\ZZ)$ under $\pi|_{\mathscr{Y}(\ZZ)}$. 
	Then a point of $\mathscr{Y}(\ZZ)$ with coordinates 
	$(y_1,\dots,y_s)$ belongs to $\mathscr{Y}(\ZZ)_\mathbf{m}$ 
	if and only if $y_i$ is nonzero and $m_i$-full for all $i\in\{1,\dots,s\}$. 
	
	For every semiample divisor $L$ that satisfies Assumption \ref{assp:L} and every $B>0$, let 
	$N_{\mathbf m, L}(B)$ be the number of points in $(\mathscr{X},\Delta)(\ZZ)$ 
	of height $H_L$ at most $B$.
	Since $\pi:\mathscr Y\to\mathscr X$ is a $\GG_m^r$-torsor, we have
	\[
	N_{\mathbf m,L}(B)=\frac 1{2^r}\ \sharp\left\{\mathbf y\in\mathscr{Y}(\ZZ)_{\mathbf{m}}: H_L( \pi(\mathbf y ))\leq B \right\}.
	\]

\subsection{Heuristics}
\label{sec:heuristics}
	In this subsection we give a heuristic argument based on the hyperbola method 
	in support of \cite[Conjecture 1.1]{PSTVA} for split toric varieties.
	
	We assume that $-(K_X+\Delta)$ is ample and that $L$ is a big and semiample $\mathbb Q$-divisor, 
	not necessarily equal to $-[K_X+\Delta]$ in $\pic(X)_{\RR}$, that satisfies Assumption \ref{assp:L}.
	We recall that \cite[Conjecture 1.1]{PSTVA} for $(\mathscr X,\Delta)$ predicts the asymptotic formula
	\begin{equation}
	\label{eq:conjecture}
	N_{\mathbf m,L}(B)\sim c B^{a(L)}(\log B)^{b(L)-1}, \quad B\to+\infty,
	\end{equation}
	where $c$ is a positive constant,
	\begin{equation}
	a(L):=\inf\{t\in \RR: t[L]+[K_X+\Delta] \in \eff(X)\}
	\end{equation}
	and $b(L)$ is the codimension of the minimal face of $\eff(X)$ that contains 
	$a(L)[L]+[K_X+\Delta]$.
	In particular, $b(L)$ is a positive integer, and $a(L)$ is a positive real number, as $-[K_X+\Delta]$ is ample. 	Since $\eff(X)$ is closed in the euclidean topology, the infimum in 
	the definition of $a(L)$ is actually a minimum.
	
\subsubsection{Combinatorial description of $a(L)$}
\label{sec:combinatorial_description_a}

	Recall the notation introduced in Section \ref{sec:polytopes_toric}. 
	We now give a characterization of $a(L)$ as the solution of certain linear programming problems.
	\begin{proposition}
	\label{prop:description_a}
	\begin{enumerate}[label=(\roman*), ref=(\roman*)]
	\item 
	\label{item:global_a}
	The number $a(L)$ is the minimal value of the function 
	$\sum_{\sigma\in\Sigma_{\max}}\lambda_\sigma$ subject to the conditions
	\begin{gather}
	\label{eq:item:global_a_condition_variables}
	\lambda_{\sigma}\geq0,\quad\forall \sigma\in\Sigma_{\max},\\
	\label{eq:item:global_a_condition_effective}
	\sum_{\sigma\in\Sigma_{\max}}\lambda_\sigma \alpha_{i,\sigma}
	\geq \frac 1{m_i}, \quad \forall i\in\{1,\dots,s\}.
	\end{gather}
	\item 
 	\label{item:local_a}
	Assume, additionally, that  $L$ is ample. Fix $\sigma\in\Sigma_{\max}$ and 
	define $\gamma_i:=\sum_{j\in\ii{\sigma}}\frac {\beta_{\sigma,i,j}}{m_j}$. 
	Then $a(L)$ is the minimal value of the function $\lambda_0$ subject to the conditions
  	\begin{gather}
  	\label{eq:item:local_a_condition_variables}
  	\lambda_0, \lambda_j\geq0, \quad\forall j\in\ii{\sigma},\\  
  	\label{eq:item:local_a_condition_effective}
  	\lambda_0 \alpha_{i,\sigma} -\sum_{j\in \ii{\sigma}}\beta_{\sigma,i,j}\lambda_j
  	\geq \frac 1{m_i}+\gamma_i, \quad \forall i\in\ic{\sigma}.
  	\end{gather}
	\end{enumerate}
	\end{proposition}
	\begin{proof}
	To prove part \ref{item:global_a} we observe that for $t\in\RR$ the condition 
	\begin{equation}
	\label{eq:a_condition_def}
	t[L]+[K_X+\Delta] \in \eff(X) 
	\end{equation}
 	is equivalent to $\varphi^{-1}(t[L]+[K_X+\Delta])\cap C\neq\emptyset$. 
 	Now, $\varphi^{-1}(t[L]+[K_X+\Delta])=tL+K_X+\Delta+M_\RR$. 
 	If $D\in (tL+K_X+\Delta+M_\RR)\cap C$, then $D-(K_X+\Delta)\in C$ as 
 	$-(K_X+\Delta)=\sum_{i=1}^s\frac 1{m_i}D_i\in C$ and $C$ is a cone.  
 	Then \eqref{eq:a_condition_def} holds if and only if there exists a divisor 
 	$D'\in (tL+M_\RR)\cap C$ such that $D'+K_X+\Delta\in C$. 
 	By Lemma \ref{lem:effective_cone} this is equivalent to the existence of 
 	$\lambda_\sigma\in\RR_{\geq0}$ for all $\sigma\in\Sigma_{\max}$ such that 
 	$\sum_{\sigma_\in\Sigma_{\max}}\lambda_\sigma=t$ and 
 	$\sum_{\sigma\in\Sigma_{\max}}\lambda_\sigma L(\sigma)+K_X+\Delta\in C$.

	Now we prove part \ref{item:local_a}. Condition \eqref{eq:a_condition_def} is 
	equivalent to the existence of $D\in C$ such that $t[L]+[K_X+\Delta]=\varphi(D)$ in 
	$\pic(X)_\RR$. Since $X$ is proper and smooth, the last equality is 
	equivalent to $tL(\sigma)+(K_X+\Delta)(\sigma)=D(\sigma)$. 
	Write $D=\sum_{i=1}^s\lambda_i D_i$. 
	Then $D\in C$ if and only if $\lambda_1,\dots,\lambda_s\geq0$.
	We have 
	\[
	tL(\sigma)+(K_X+\Delta)(\sigma)-D(\sigma)
	=\sum_{i\in\ic{\sigma}}\left( t \alpha_{i,\sigma} -\frac 1{m_i}-\lambda_{i}-\gamma_i
	-\sum_{j\in \ii{\sigma}}\beta_{\sigma,i,j}\lambda_j\right)D_{i}.
	\]
	Using the fact that $\lambda_{i}\geq0$ for all $i\in\ic{\sigma}$ if $D\in C$, 
	we see that condition \eqref{eq:a_condition_def} is equivalent to the existence of 
	$\lambda_j\in\RR_{\geq0}$ for all $j\in\ii{\sigma}$ that satisfy the conditions in 
	the statement for $\lambda_0=t$.
	\end{proof}

\subsubsection{Heuristic argument for $a(L)$}
	Next we give a heuristic argument in support of \cite[Conjecture 1.1]{PSTVA}
	(and \cite[\S3.3]{MR1032922}) 
	regarding the expected exponent $a(L)$ of $B$ in the asymptotic 
	formula \eqref{eq:conjecture} for  split toric varieties over $\QQ$.

Up to a positive constant, $N_{\mathbf m,L}(B)$ is the cardinality $S$ of the set of $m_i$-full positive integers $y_i$ for $i\in\{1,\dots,s\}$ that satisfy the conditions $\mathbf y^{L(\sigma)}\leq B$ for all $\sigma\in\Sigma_{\max}$. We recall that $\mathbf y^{L(\sigma)}=\prod_{i=1}^s y_i^{\alpha_{i,\sigma}}$. 

One of the ideas of the hyperbola method is to  dissect the region of summation for the variables $y_1,\dots,y_s$, into different boxes. Assume that we consider a box where say $y_i \sim B_i$ (here we mean that for example $B_i\leq y_i\leq 2B_i$ for $i\in\{1,\dots,s\}$), and let $B_i=B^{t_i}$. What contribution do such vectors $\mathbf y=(y_1,\dots,y_s)$ give to computing the cardinality $S$? First we note that the contribution from this box is
\[
\text{box contribution} = \prod_{i=1}^sB_i^{\frac 1{m_i}}=B^{\sum_{i=1}^s t_i\frac1{m_i} }.
\]
In order for this to be a box that we count by $S$, the parameters $(t_1,\dots,t_s)$ need to satisfy
\begin{equation}
\label{eq:global_heuristic_linearprogramming_height_condition}
\sum_{i=1}^s t_i \alpha_{i,\sigma}\leq 1, \quad \forall \sigma\in\Sigma_{\max}
\end{equation}
and
\begin{equation}
\label{eq:global_heuristic_linearprogramming_variables}
t_1,\dots,t_s\geq0.
\end{equation}
In order to find the size of $S$ we hence have the following linear programming problem $\mathcal{P}$:
Maximize the function 
\begin{equation}
\label{eq:maximize_global_function}
\sum_{i=1}^s t_i\frac1{m_i}
\end{equation}
under the conditions \eqref{eq:global_heuristic_linearprogramming_height_condition} and \eqref{eq:global_heuristic_linearprogramming_variables}. 
The conditions \eqref{eq:global_heuristic_linearprogramming_height_condition} and \eqref{eq:global_heuristic_linearprogramming_variables} define a polytope $\widetilde P$ in $\RR^s$ and by the theory of linear programming we know that the maximum of the function $\sum_{i=1}^s t_i\frac1{m_i}$ is obtained on at least one of its vertices.

The dual linear programming problem $\mathcal{D}$ is given by the following problem:
Minimize the function
\[
\sum_{\sigma\in\Sigma_{\max}}\lambda_\sigma
\]
under the conditions \eqref{eq:item:global_a_condition_variables} and \eqref{eq:item:global_a_condition_effective}.
By the strong duality property in linear programming \cite[Chapter 6]{Dantzig}, both problems have a finite optimal solution and these values are equal. Since $a(L)$ is positive, by Proposition \ref{prop:description_a}\ref{item:global_a} it is the solution of the dual linear programming problem $\mathcal{D}$ and also of $\mathcal{P}$.

\subsubsection{Heuristic argument for $b(L)$}
	Now we give a heuristic argument in support of \cite[Conjecture 1.1]{PSTVA}
	(and \cite[\S3.3]{MR1032922}) 
	regarding the expected exponent $a(L)$ of $B$ in the asymptotic 
	formula \eqref{eq:conjecture} for  split toric varieties over $\QQ$. 
	We keep the setting introduced above.

 	If we cover the region of summation $\mathbf y^{L(\sigma)}\leq B$ for all $\sigma\in\Sigma_{\max}$ 
 	by dyadic boxes, then the maximal value of the count is attained on boxes, 
 	that are located at the maximal face $\widetilde{F}$ of the polytope $\widetilde{P}$ where the function 
 	in (\ref{eq:maximize_global_function}) is maximized. 
 	Working with a dyadic dissection this suggests that the leading term should be 
 	of order $B^{a(L)} (\log B)^k$, where $k$ is equal to the dimension of the face $\widetilde{F}$. 
 	The next proposition shows that $k=b(L)-1$. 
 	Hence, the heuristic expectation we obtained from the hyperbola method matches 
 	the prediction in \cite[Conjecture 1.1]{PSTVA}.

	We recall from Lemma \ref{lem:face_maximization_problem} that  
	$\widetilde F\subseteq\widetilde H$, where $\widetilde H$ 
	is the space of linear functions on $\pic(X)_\RR$. 
	With this identification, the cone generated by $\widetilde P\cap \widetilde H$ is the space 
	of linear functions on $\pic(X)_{\RR}$ that are nonnegative on $\eff(X)$ 
	(i.e., the cone  in $\widetilde H$ dual to $\eff(X)$) by \cite[Proposition 1.2.8]{MR2810322}.

	\begin{prop}
	\label{prop:description_b}
	The cone generated by $\widetilde F$ is dual to the minimal face of $\eff(X)$ that contains 
	$a(L)[L]+[K_X+\Delta]$.  In particular, $b(L)= \dim \widetilde F+1$.
	\end{prop}
	\begin{proof}
	Fix $\sigma\in\Sigma_{\max}$. 
	Then $\widetilde P\cap \widetilde H=\widetilde P_\sigma\cap \widetilde H$.
	Let $\mathbf t=(t_1,\dots,t_s)\in \cone \left( \widetilde P\cap \widetilde H \right)$.
	We recall that $[L] = \sum_{i\in\ic{\sigma}} \alpha_{i,\sigma} [D_{i}]$ and 
	$[K_X+\Delta] = -\sum_{i\in\ic{\sigma}} \left( \frac 1{m_i} + \gamma_i \right)[D_{i}]$, so that 
	\[
	\mathbf t( a(L)[L] + [K_X+\Delta] ) 
	= a(L) \sum_{i\in\ic{\sigma}}\alpha_{i,\sigma} t_i 
	- \sum_{i\in\ic{\sigma}} \left( \frac 1{m_i} + \gamma_i \right) t_i.
	\]
	If $\mathbf t\in \widetilde F$, then $\mathbf t ( a(L)[L] + [K_X+\Delta] ) =0$. 
	Conversely, if $\mathbf t\neq0$ and $t ( a(L)[L] + [K_X+\Delta] ) =0$, 
	then $\alpha:=\sum_{i\in\ic{\sigma}}\alpha_{i,\sigma} t_i > 0$ as $\frac 1{m_i} + \gamma_i>0$ 
	and $t_i\geq0$ for all $i\in\ic{\sigma}$. 
	So $(\alpha^{-1}t_1,\dots,\alpha^{-1}t_s)\in \widetilde F$, and $\mathbf t\in \cone(\widetilde F)$.
	Thus, $\cone(\widetilde F)$ is the face of $\cone(\widetilde P\cap \widetilde H)$ defined by 
	\[
	\mathbf t( a(L)[L] + [K_X+\Delta] ) =0.
	\]
	By \cite[Definition 1.2.5, Proposition 1.2.10]{MR2810322} the face of 
	$\eff(X)$ dual to $\cone(\widetilde F)$ is
	the smallest face of $\eff(X)$ that contains $a(L)[L] + [K_X+\Delta]$, and
	 \begin{equation*}b(L)=\dim \cone( \widetilde F ) = \dim \widetilde F+1.\qedhere\end{equation*}
	\end{proof}

\subsection{Proof of Theorem \ref{thm:toric}}
\label{sec:proof_thmtoric}

	From now on we work in the setting of Theorem \ref{thm:toric}. 
	In particular, $L=-(K_X+\Delta)=\sum_{i=1}^s m_i^{-1}D_i$, and 
	\[
	N(B)=\frac 1{2^r}\ \sharp
	\left\{\mathbf y\in\mathscr{Y}(\ZZ)_{\mathbf{m}}: H( \mathbf y )\leq B \right\},
	\]
	where 
	$H( \mathbf y ):=\sup_{\sigma\in\Sigma_{\max}} \left| \mathbf y^{L(\sigma)}\right|$, 
	by Section \ref{sec:campana} and Proposition \ref{prop:height}.
	
\subsubsection{M\"obius inversion}
	
	For $B>0$ and $\mathbf d\in(\ZZ_{>0})^{s}$, 
	let $A(B,\mathbf d)$ be the set of points
	$\mathbf y=(y_1,\dots,y_s)\in\ZZ^{s}$ such that 
	$H(\mathbf y)\leq B$,  $y_i$ is nonzero and $m_i$-full and $d_i\mid y_i$ for all $i\in\{1,\dots,s\}$. 
	We observe that $A(B, \mathbf d)$ is a finite set by Lemma \ref{lem:northcott}.
	Then  
	\begin{equation}
	\label{eq:moebius_inversion}
 	N(B)=\frac 1{2^r}\sum_{\mathbf d\in(\ZZ_{>0})^{s}}\mu(d)\sharp A(B, \mathbf d),
	\end{equation}
	where $\mu$ is the function introduced in \cite[Definition and proposition 11.9]{MR1679841}. 
		
\subsubsection{The estimate}

	Under the assumptions of Theorem \ref{thm:toric}, we can apply Theorem \ref{lemhyp3} 
	to obtain an estimate for the cardinality of the sets $A(B,\mathbf d)$ as follows.
	\begin{prop}
	\label{prop:A(B,d)}
	There is $B_0>0$ such that for every $s$-tuple $\mathbf d=(d_1,\dots, d_s)$ 
	of squarefree positive integers,  $B>B_0$ and  $\epsilon>0$, we have
	\begin{multline*}
	\sharp A(B,\mathbf d)= 2^s 
	\alpha(L)
	\left(\sum_{\sigma\in\Sigma_{\max}}\prod_{i\in\ic{\sigma}}m_i^{-1}\right)
	\left(\prod_{i=1}^sc_{m_i,d_i}\right)  B (\log B)^{r-1} 
	\\+O_{\varepsilon}\left(
	\left(\prod_{i=1}^s d_i\right)^{-\frac23+\varepsilon} B (\log B)^{r-2}(\log\log B)^{s-1}
	\right),
	\end{multline*}
	where $\alpha(L)$ is defined in Lemma \ref{lem:alpha_peyre} and $c_{m_i,d_i}$ 
	is defined in \eqref{eq:c_md_explicit}.
	\end{prop}	
	\begin{proof}
	We write $\sharp A(B,\mathbf d) = 2^{s}\sharp A'(B,\mathbf d)$, where
	 $A'(B,\mathbf d)$ is the set of positive integers $y_1,\dots,y_s$ that satisfy the conditions
	\[ 
	d_i \mid y_i, \quad y_i \text{ is $m_i$-full} \quad\forall i \in\{1,\dots,s\}
	\]
	and the inequalities
	\begin{equation}
	\label{eq:height_condition}
	\prod_{i=1}^s y_i^{\alpha_{i,\sigma}} \leq B, \quad \forall \sigma\in\Sigma_{\max}.
	\end{equation}
	We observe that 
	\[
	\sharp A'(B,\mathbf d)= 
	\sum_{ y_1,\dots, y_s\in\ZZ_{>0} : \eqref{eq:height_condition} } 
	f_{\mathbf m, \mathbf d}(y_1,\dots,y_s),
	\]
	where $f_{\mathbf m, \mathbf d}$ is the function defined in Subsection \ref{sec:counting_function_f}.
	The function $f_{\mathbf m, \mathbf d}$ satisfies  Property I by Lemma \ref{lem:counting_function_f} and \eqref{eq:c_md_estimate}, 
	see the remarks after Lemma \ref{lem:counting_function_f}. 
	Assumption \ref{assp10} is satisfied as \eqref{eq:widetildeP} defines an $s$-dimensional polytope.	
	Assumption \ref{assp11} is satisfied 
	by Lemma \ref{lem:face_maximization_problem}\ref{item:lem:face_maximization_problem_special}.
	Hence, we can combine Theorem \ref{lemhyp3} and Lemma \ref{lem:counting_function_f} to compute
	\begin{multline}
	\sharp A'(B,\mathbf d)= (s-1-k)! \left(\prod_{i=1}^sc_{m_i,d_i}\right) c B^{a}(\log B)^k 
	\\+O\left(
	\left(\prod_{i=1}^s d_i\right)^{-\frac23+\varepsilon} B^{a}(\log B)^{k-1}(\log\log B)^{s-1}
	\right),
	\end{multline}
	where $a=1$ and $k=r-1$ by 
	Lemma \ref{lem:face_maximization_problem}\ref{item:lem:face_maximization_problem_special}, and 
	$$c=\frac{\alpha(L)}{(s-r)!}\sum_{\sigma\in\Sigma_{\max}}\prod_{i\in\ic{\sigma}}m_i^{-1}$$ 
	by Lemma \ref{lem:alpha_constant}.
	\end{proof}

	We combine the proposition above  with the M\"obius inversion to obtain an estimate for $N(B)$.

	\begin{prop}
	\label{prop:final_estimate}
	For sufficiently large $B>0$, we have
	$$
	N(B) = c B(\log B)^{r-1} + O(
	B(\log B)^{r-2}(\log\log B)^{s-1}),
	$$
 	where 
 	\begin{equation}
 	\label{eq:leading_constant}
 	c=2^{s-r}\alpha(L)\left(\sum_{\sigma\in\Sigma_{\max}}\prod_{i\in\ic{\sigma}}m_i^{-1}\right)
 	\sum_{\mathbf d\in(\ZZ_{>0})^{s}}\mu(d)
 	\left(\prod_{i=1}^sc_{m_i,d_i}\right).
 	\end{equation}
	\end{prop}
	\begin{proof}
	By \eqref{eq:moebius_inversion} and Proposition \ref{prop:A(B,d)}
	we have
	$$\left| 
	N(B)-
	c B(\log B)^{r-1} 
 	\right| 
 	/ \left(B(\log B)^{r-2}(\log\log B)^{s-1} \right) \ll_{\varepsilon}
 	\sum_{\mathbf d\in(\ZZ_{>0})^{s}}\frac{|\mu(d)|}
	{\prod_{i=1}^s d_i^{\frac23-\varepsilon}}.
	$$
	For $\varepsilon=\frac1{12}$, the right hand side converges by \cite[Lemma 11.15]{MR1679841}.
	\end{proof}	
	
	Theorem \ref{thm:toric} follows from Proposition \ref{prop:final_estimate} 
	and the interpretation of the leading constant that we carry out in the following section.
	
\subsubsection{The leading constant}

	Here we prove that the leading constant \eqref{eq:leading_constant} can be written as
	\begin{equation}
	\label{eq:expected_leading_constant}
	\alpha(X, \Delta)\tau(X,\Delta)\prod_{i=1}^s\frac 1{m_i},
	\end{equation}
	where $\alpha(X,\Delta):=\alpha(-(K_X+\Delta))$ has been introduced in 
	Lemma \ref{lem:alpha_peyre}, and
	\[
	\tau(X,\Delta)=\int_{\overline {\mathscr X(\ZZ)}}H_{\Delta}(x) \delta_{\Delta}(x) \tau_X,
	\]
	where $\tau_X$ is the Tamagawa measure on the set of adelic points $X(\mathbb A_\QQ)$ 
	defined in \cite[Definition 2.8]{MR2740045}, 
	$H_\Delta$ is the height function defined by the divisor $\Delta$ as in \cite[\S2.1]{MR1369408}, 
	$\overline {\mathscr X(\ZZ)}$ is the closure of the set of rational points $\mathscr X(\ZZ)$ 
	inside the space of adelic points on $X$, 
	and $\delta_{\Delta}=\prod_{p\in\Omega_f}\delta_{\Delta, p}$, 
	where $\delta_{\Delta,p}$ is the characteristic function of the set of 
	Campana points in $\mathscr{X}(\ZZ_p)$ for each finite place $p$.

	We observe that by  \cite[Proposition 2.4.4]{MR1369408}, 
	the product \eqref{eq:expected_leading_constant} agrees with the expectation formulated 
	in \cite[\S3.3]{PSTVA} provided that the domains of integration in the definitions of 
	$\tau(X,\Delta)$ (i.e., $\{x\in \overline {\mathscr X(\ZZ)} : \delta_{\Delta}(x)=1\}$) coincide. 

	\begin{prop}
	\label{prop:constant}
	We have 
	$$\tau(X,\Delta)=\left(2^{s-r}\sum_{\sigma\in\Sigma_{\max}}\prod_{j\in\ii{\sigma}} m_j\right)
	\sum_{\mathbf d\in(\ZZ_{>0})^s}\left( \mu(\mathbf d) \prod_{i=1}^sc_{m_i,d_i} \right)>0.$$
	\end{prop}
	\begin{proof}
	For every prime number $p$, we denote by $\Fr_p$ the geometric Frobenius acting 
	on $\pic(\mathscr 	X_{\overline{\FF}_p})\otimes_\ZZ\QQ$, and for $t\in\CC$
	we define $L_p(t,\pic(X_{\overline \QQ})):=\det(1-p^{-t} \Fr_p)^{-1}$. 
 	Since $X$ is split, and hence $\mathscr{X}_{\overline{\FF}_p}$ is split, 
 	we have $L_p(t,\pic(X_{\overline \QQ}))=(1-p^{-t})^{-r}$ for every prime $p$.
	Let $\lambda_p:=L_p(1,\pic(X_{\overline \QQ}))=(1-p^{-1})^{-r}$, 
	and define $\lambda:=\lim_{t\to 1}(t-1)^{r}\prod_{p\in\Omega_f}L_p(t,\pic(X_{\overline \QQ}))$. 
	Then $\lambda=\lim_{t\to 1}(t-1)^{r}\zeta(t)^r=1$ by properties of the residue at $1$ 
	of the Riemann zeta function. Let $\lambda_{\infty}:=1$. 
	By \cite[Remark 2.9(b)]{MR2740045}, we have
	$\tau_X=\lambda \prod_{v\in\Omega_\QQ}\lambda_v^{-1}\tau_{X,p}$, where $\tau_{X,v}$ 
	denotes the local measure on $X(\QQ_v)$ defined in \cite[\S2.1.7]{MR2740045}.

	Since split toric varieties satisfy weak approximation (e.g., \cite[\S2]{MR2029861}) 
	and $\mathscr X$ is projective, we get 
	\[
	\tau(X,\Delta)=\int_{X(\RR)}H_{\Delta, \infty}(x)\tau_{X,\infty}
	\prod_{p \in \Omega_f}\lambda_p^{-1}
	\int_{ {\mathscr X(\ZZ_p)}}H_{\Delta,p}(x)\delta_{\Delta,p} (x) 	\tau_{X,p},
	\]
	where 
	$H_\Delta=\prod_{v\in\Omega_\QQ} H_{\Delta,v}$ as in \cite[\S2.1]{MR1369408}. 
	We recall that the global metrization used to define $H_\delta$ in \cite[\S2.1]{MR1369408} 
	is equivalent to the one determined by \cite[Proposition and definition 9.2]{MR1679841}. 

	Since $\tau_{X,\infty}$ is the measure used in \cite[Proposition 9.16]{MR1679841} for the real place, 
	we have
	\begin{gather*}
	\int_{X(\RR)}H_{\Delta, \infty}(x)\tau_{X,\infty}
	=\sum_{\sigma\in\Sigma_{\max}}\left(2^{s-r}\prod_{j \in \ii{\sigma}}
	\int_{0\leq x_j\leq 1}x_j^{\frac 1{m_j}-1}\mathrm d x_j\right)\\
	=2^{s-r}\sum_{\sigma\in\Sigma_{\max}}\prod_{j\in \ii{\sigma}} m_j.
	\end{gather*}

	Moreover, if we use the $\GG_m^r$-torsor structure on  
	$\pi:\mathscr Y(\ZZ_p)\to\mathscr X(\ZZ_p)$ together 
	with \cite[Corollary 2.23, Propositions 9.7,  9.13 and 9.14]{MR1679841}  we get
	\begin{gather*}
	\int_{\mathscr X(\ZZ_p)}H_{\Delta,p}(x)\delta_{\Delta,p} (x) \tau_{X,p}
	=\left(\int_{\GG_m^r(\ZZ_p)}\mathrm d z\right)^{-1} 
	\int_{\mathscr Y(\ZZ_p)}H_{\Delta,p}(\pi(y))\delta_{\Delta,p} (\pi(y)) \prod_{i=1}^s\mathrm d y_i,
	\end{gather*}
	where $dz$ and $\prod_{i=1}^s\mathrm d y_i$ are the Haar measures on $\QQ_p^r$ and $\QQ_p^s$, 
	respectively, induced by the Haar measure on $\QQ_p$ normalized such that $\ZZ_p$ has volume $1$, 
	and $\delta_{\Delta,p}\circ \pi$ is the characteristic function of the property
	\[
	v_p(y_i)>0 \quad  \Rightarrow \quad v_p(y_i)\geq m_i, \qquad \forall i\in\{1,\dots,s\}.
	\]
	We have $\int_{\GG_m^r(\ZZ_p)}\mathrm d z=(1-p^{-1})^r$. 
	Let $\chi$ be the characteristic function of $\mathscr Y$. By \cite[Lemma 11.15]{MR1679841}
	we have
	\begin{gather*}
	\int_{\mathscr Y(\ZZ_p)}H_{\Delta,p}(\pi(\mathbf y))\delta_{\Delta,p} (\pi(\mathbf y)) 
	\prod_{i=1}^s\mathrm d y_i
	=\int_{\ZZ_p^N}\chi(\mathbf y)\delta_{\Delta,p} (\pi(\mathbf y)) 
	\prod_{i=1}^s | y_i |_p^{\frac 1{m_i}-1} \prod_{i=1}^s\mathrm d y_i\\
	=\sum_{(\mathbf e)\in\{0,1\}^s}\mu((p^{e_1},\dots,p^{e_s})) 
	\int_{p^{e_i}\mid y_i, 1\leq i\leq s}\delta_{\Delta,p} (\pi(\mathbf y)) 
	\prod_{i=1}^s | y_i |_p^{\frac 1{m_i}-1} \prod_{i=1}^s\mathrm d y_i\\
	=\sum_{(\mathbf e)\in\{0,1\}^s}\mu((p^{e_1},\dots,p^{e_s}))  
	\prod_{i=1}^s \left(
	(1-e_i)\int_{\ZZ_p^{\times}}\mathrm d y_i 
	+ \sum_{j=m_i}^\infty\int_{p^j\ZZ_p^\times}|y_i|_p^{\frac 1{m_i}-1} \mathrm d y_i
	\right)\\
	=\sum_{(\mathbf e)\in\{0,1\}^s}\mu((p^{e_1},\dots,p^{e_s})) 
 	\prod_{i=1}^s\left(
 	\left (1-p^{-1} \right)\left(1-e_i +\sum_{j=m_i}^\infty p^{-\frac j{m_i}}\right)
	\right).
	\end{gather*}
 	We observe that the product 
 	$\left (1-p^{-1} \right)\left (1-e_i +\sum_{j=m_i}^\infty p^{-\frac j{m_i}}\right)$ equals
	\begin{gather*}
	\begin{cases}
	1+\sum_{j=m_i+1}^{2m_{i}-1}p^{-\frac j{m_i}} 
	& \text{ if } e_i=0\\
 	\left(1+\sum_{j=m_i+1}^{2m_{i}-1}p^{-\frac j{m_i}} \right ) 
 	\left (1+p-p^{\frac {m_i-1}{m_i}} \right)^{-1}
 	& \text{ if } e_i=1.
	\end{cases}
	\end{gather*}
	So, remembering \eqref{eq:C_m_explicit} and \eqref{eq:G_m_explicit}, we get
	\begin{multline*}
	\tau({X,\Delta})
	\left(2^{s-r}\sum_{\sigma\in\Sigma_{\max}}\prod_{j\in\ii{\sigma}} m_j\right)^{-1} 
	\\
	=
	\left(\prod_{i=1}^s C_{m_i}\right)
	\left (
	\prod_{p\in\Omega_f}\sum_{\mathbf e\in\{0,1\}^s} \mu ((p^{e_1},\dots,p^{e_s})) 
	\prod_{i=1}^s\left (p^{-1}+G_{m_i}\left (p^{-\frac 1{m_i}}\right )\right )^{e_i}
	\right).
	\end{multline*}
	We conclude by \cite[Lemma 11.15(e)]{MR1679841} as 
	$\prod_{p\mid d_i}\left(p^{-1}+G_{m_i}\left(p^{-\frac 1{m_i}}\right)\right)\ll_\epsilon d_i^{-1+\epsilon}$
	 by \eqref{eq:c_md_estimate} for all $i\in\{1,\dots,s\}$.

	To show that $\tau({X,\Delta})$ is positive it suffices to observe that 
	$$
	\int_{\mathscr Y(\ZZ_p)}H_{\Delta,p}(\pi(y))\delta_{\Delta,p} (\pi(y)) \prod_{i=1}^s\mathrm d y_i 
	\geq 
 	\int_{\mathscr (\ZZ_p^\times)^s} \prod_{i=1}^s\mathrm d y_i =(1-p^{-1})^s >0,
	$$
	as the integral on the right is the restriction of the integral on the left to the subset 
	$(\ZZ_p^\times)^s	\subseteq \mathscr Y(\ZZ_p)$.
	\end{proof}
	
\section*{Acknowledgements}
We thank for their hospitality the organizers of the trimester program 
	``Reinventing Rational Points'' at the Institut Henri Poincar\'e and the organizers of the workshops ``Rational Points'' 2022 and 2023 at Schney, where we made significant progress on this project. We thank the referees for their comments, that improved the exposition of this article. 
	The first author is supported by a NWO grant VI.Vidi.213.019.
	The second author is supported by a NWO grant 016.Veni.173.016.

\bibliographystyle{alpha}
\bibliography{bibliography_revision}
\end{document}